\documentclass[12pt]{article}
\usepackage{amsmath, amsthm, amssymb, stmaryrd, fullpage}

\newtheorem{theorem}{Theorem}[subsection]
\numberwithin{equation}{theorem}
\newtheorem{lemma}[theorem]{Lemma}
\newtheorem{cor}[theorem]{Corollary}
\newtheorem{prop}[theorem]{Proposition}
\theoremstyle{definition}
\newtheorem{defn}[theorem]{Definition}
\newtheorem{example}[theorem]{Example}
\newtheorem{convention}[theorem]{Convention}
\newtheorem{remark}[theorem]{Remark}
\newtheorem{hypothesis}[theorem]{Hypothesis}
\newtheorem{notation}[theorem]{Notation}

\newcommand{\calE}{\mathcal{E}}
\newcommand{\calO}{\mathcal{O}}
\newcommand{\calR}{\mathcal{R}}
\newcommand{\calZ}{\mathcal{Z}}

\newcommand{\gothm}{\mathfrak{m}}
\newcommand{\gotho}{\mathfrak{o}}
\newcommand{\gothp}{\mathfrak{p}}

\newcommand{\AAA}{\mathbb{A}}
\newcommand{\CC}{\mathbb{C}}
\newcommand{\DD}{\mathbb{D}}
\newcommand{\FF}{\mathbb{F}}
\newcommand{\QQ}{\mathbb{Q}}
\newcommand{\RR}{\mathbb{R}}
\newcommand{\ZZ}{\mathbb{Z}}

\newcommand{\del}{\partial}

\newcommand{\be}{\mathbf{e}}
\newcommand{\bv}{\mathbf{v}}
\newcommand{\bw}{\mathbf{w}}

\newcommand{\dual}{\vee}

\DeclareMathOperator{\abs}{abs}
\DeclareMathOperator{\based}{based}
\DeclareMathOperator{\bd}{bd}

\DeclareMathOperator{\Der}{Der}
\DeclareMathOperator{\End}{End}
\DeclareMathOperator{\Frac}{Frac}
\DeclareMathOperator{\Gal}{Gal}
\DeclareMathOperator{\gr}{gr}
\DeclareMathOperator{\Hom}{Hom}
\DeclareMathOperator{\irreg}{irreg}
\DeclareMathOperator{\inte}{inte}
\DeclareMathOperator{\length}{length}
\DeclareMathOperator{\rank}{rank}
\DeclareMathOperator{\Real}{Re}
\DeclareMathOperator{\Spec}{Spec}
\DeclareMathOperator{\spect}{sp}
\DeclareMathOperator{\Sym}{Sym}
\DeclareMathOperator{\Trace}{Trace}

\begin{document}

\title{Good formal structures for flat meromorphic connections, I:
Surfaces}
\author{Kiran S. Kedlaya}
\date{December 8, 2009}
\maketitle

\begin{abstract}
We give a criterion
under which one can obtain a good decomposition 
(in the sense of Malgrange) of a formal flat connection
on a complex analytic or algebraic variety of arbitrary dimension.
The criterion is stated in terms of the spectral behavior of differential operators,
and generalizes Robba's construction of the
Hukuhara-Levelt-Turrittin decomposition in the one-dimensional case.
As an application, we prove the 
existence of good formal structures for flat meromorphic 
connections on surfaces after suitable blowing up; this 
verifies a conjecture of
Sabbah, and extends a result of Mochizuki for algebraic connections.
Our proof uses a finiteness argument on the valuative tree associated
to a point on a surface, in order to verify the numerical criterion.
\end{abstract}

\section*{Introduction}

The Hukuhara-Levelt-Turrittin decomposition theorem
gives a classification of differential modules over the field
$\CC((z))$ of formal Laurent series resembling the decomposition of a
finite-dimensional vector space equipped with a linear endomorphism
into generalized eigenspaces. It implies that after adjoining
a suitable root of $z$, one can express any differential module as a successive
extension of one-dimensional modules.
This classification serves as the basis for
the asymptotic analysis of meromorphic connections around a 
(not necessarily regular) singular point.
In particular, it leads to a coherent description of the \emph{Stokes phenomenon},
i.e., the fact that the asymptotic growth of horizontal sections 
near a singularity must be described using different asymptotic series depending
on the direction along which one approaches the singularity.
(See \cite{varadarajan} for a beautiful exposition of this material.)

The purpose of this series of papers
is to give some higher-dimensional analogues of the Hukuhara-Levelt-Turrittin
decomposition for \emph{irregular} flat formal meromorphic connections on complex 
analytic or 
algebraic varieties. 
(The regular case is already well understood by work of Deligne \cite{deligne}.)
We do not discuss asymptotic analysis or the Stokes
phenomenon; these has been treated in the two-dimensional case by Sabbah
\cite{sabbah} (building on work of Majima \cite{majima}), 
and one expects the higher-dimensional case to behave similarly.

This paper separates naturally into two parts.
In the remainder of this introduction, we discuss these two parts 
individually, then append some further remarks.

\subsection{Local structure theory}

In the first part of the paper (\S~\ref{sec:diff alg}--\ref{sec:numerical}),
we develop a 
numerical criterion for the existence of a \emph{good decomposition}
(in the sense of Malgrange \cite{malgrange-lille}) of a formal flat meromorphic
connection at a point where the polar divisor has normal crossings. 
This criterion is inspired by the treatment of
the original decomposition theorem given by Robba \cite{robba-hensel}
using spectral properties of differential operators on nonarchimedean rings;
our treatment depends heavily on joint work with
Xiao \cite{kedlaya-xiao} concerning differential modules on some nonarchimedean
analytic spaces.

The criterion can be formulated as follows.
Let $M$ be a finite projective module over
$R = \CC \llbracket x_1,\dots,x_n \rrbracket[x_1^{-1},\dots,x_m^{-1}]$
equipped with a continuous flat connection. For each exceptional divisor on a toric
blowup of $\Spec R$, one can compute the \emph{irregularity} 
of $M$ and of $M^\dual \otimes M$ using spectral norms of differential
operators. The criterion asserts that $M$ admits a good decomposition
(after pullback along a finite \'etale cover of $\Spec R$)
if and only if the variation of the irregularity over the space of toric
exceptional divisors is consistent with such a decomposition.

\subsection{Application to surfaces}

In the second part of the paper (\S~\ref{sec:valuative}--\ref{sec:surfaces}),
we apply the numerical criterion against
a conjecture of Sabbah  \cite[Conjecture~2.5.1]{sabbah}
concerning formal meromorphic connections on a two-dimensional
complex algebraic or analytic variety. 
We say that such a
connection has a \emph{good formal structure} at some point if it
acquires a good decomposition after pullback along a finite cover ramified only
over the polar divisor. In general, even if the polar divisor has normal crossings,
one only has good formal structures away from some discrete set, the set of
\emph{turning points}. However, Sabbah conjectured
that one can replace the given surface with a suitable blowup in such
a way that the pullabck connection admits good formal structures everywhere.

In the case of an algebraic connection on an algebraic variety,
this conjecture was proved by Mochizuki \cite[Theorem~1.1]{mochizuki}.
Mochizuki's proof uses reduction modulo a large prime and analysis
of the resulting $p$-curvature. It can be extended to some nonalgebraic
cases, inovlving power series over a
subring of $\CC$ which is \emph{finitely generated} over $\ZZ$.
However, the latter restriction is essential for reduction mod $p$,
so one cannot hope to treat the general analytic problem this way.

We give a proof of Sabbah's conjecture in full,
as an application of the numerical
criterion described above. To control the turning points, we
interpret irregularity of a differential module as a function of a
\emph{valuative tree} (in the language of Favre and Jonsson 
\cite{favre-jonsson}), i.e., a one-dimensional nonarchimedean
analytic space in the sense of Berkovich. 
This interpretation is partly inspired by
recent work of Baldassarri and di Vizio
\cite{baldassarri-divizio}.

\subsection{Further remarks}

Later in this series,
we will establish an analogue of Sabbah's conjecture for higher
dimensional varieties, using the same numerical criterion.
In the algebraic case, such an analogue has been given by
Mochizuki \cite[Theorem~19.5]{mochizuki2} using analytic methods.

We conclude this introduction by 
pointing out an analogy between this circle of ideas and
a corresponding problem in the theory of $p$-adic differential modules.
The latter is what we call the ``semistable reduction problem''
for overconvergent $F$-isocrystals, which we treated recently in 
the papers \cite{kedlaya-part1, kedlaya-part2, kedlaya-part3, kedlaya-part4}.
We thank Yves Andr\'e for the suggestion to transpose some ideas
from that work into the present context.

\subsection*{Acknowledgments}
The author thanks the Tata Institute for Fundamental Research
for its hospitality during July-August 2008.
Thanks to Liang Xiao for finding an error in the proof
of Theorem~\ref{T:regular} in an early draft.
Thanks also to Francesco Baldassarri, Joseph Gubeladze,
 and Takuro Mochizuki for helpful comments.
Financial support was provided by NSF CAREER grant DMS-0545904,
the MIT NEC Research Support Fund,
and the MIT Cecil and Ida Green Career Development Professorship.

\section{Differential algebra}
\label{sec:diff alg}

We start with some definitions and notations regarding
differential rings, fields, and modules,
with an emphasis on spectral constructions.
Note that all rings we consider
will be unital and commutative unless otherwise specified.

\setcounter{theorem}{0}
\begin{convention}
For $M$ a free module over a ring $R$, $T: M \to M$ a function,
and $\be_1,\dots,\be_n$ a basis of $M$, the \emph{matrix of action} of $T$
on $M$ is defined as the $n \times n$ matrix $N$ over $R$ satisfying
$T(\be_j) = \sum_{i} N_{ij} \be_i$. That is, we use the basis to identify
$M$ with column vectors of length $n$.
\end{convention}

\begin{defn}
Let $R$ be a ring equipped with a norm $|\cdot|$.
For $\rho = (\rho_1,\dots,\rho_n) \in (0, +\infty)^n$, the
\emph{$\rho$-Gauss norm} $|\cdot|_\rho$ on the polynomial ring
$R[T_1,\dots,T_n]$
is defined by the formula
\begin{equation} \label{eq:gauss}
\left|
\sum_{i_1,\dots,i_n} R_{i_1,\dots,i_n} T_1^{i_1} \cdots T_n^{i_n} 
\right|_r
= \max_{i_1,\dots,i_n} \{ |R_{i_1,\dots,i_n}| \rho_1^{i_1} \cdots \rho_n^{i_n}
\}.
\end{equation}
We will use the formula \eqref{eq:gauss}
also to define Gauss norms in some other
settings (twisted polynomials, power series).
\end{defn}

\subsection{Differential rings and modules}

\begin{defn}
Let $R$ be a ring.
Equip $\Hom_\ZZ(R,R)$ with the $R$-module structure given by
\[
(r_1 f)(r_2) = r_1 f(r_2) \qquad (r_1, r_2 \in R; f \in \Hom_\ZZ(R,R)).
\]
The \emph{module of absolute derivations} $\Der(R)$ is the 
$R$-submodule of $\Hom_\ZZ(R,R)$ consisting of maps
$\del: R \to R$ satisfying the Leibniz rule:
\[
\del(r_1 r_2) = \del(r_1) r_2 + r_1 \del(r_2) \qquad (r_1, r_2 \in R).
\]
There is a canonical $R$-linear
isomorphism $\Der(R) \cong \Hom_R(\Omega_{R/\ZZ}, R)$, for
$\Omega_{R/\ZZ}$ the module of absolute K\"ahler differentials of $R$.
The \emph{Lie bracket} on $\Der(R)$ is the map
$[\cdot,\cdot]: \Der(R) \times \Der(R) \to \Der(R)$ defined by
\[
[\del_1, \del_2](r) = \del_1(\del_2(r)) - \del_2(\del_1(r)) \qquad
(\del_1,\del_2 \in \Der(R); r \in R);
\]
this satisfies the Jacobi identity, and so gives $\Der(R)$ the structure
of a Lie algebra over $\ZZ$ (but not over $R$, since the bracket is not
$R$-linear).

A \emph{differential ring/field/domain} is a ring/field/domain $R$ equipped
with a Lie algebra $\Delta_R$ over $\ZZ$, a homomorphism
of Lie algebras $\Delta_R \to \Der(R)$ 
(which equips each $\del \in \Delta_R$ with the action of a derivation on $R$)
and an $R$-module structure on $\Delta_R$ for which the map 
$\Delta_R \to \Der(R)$ is $R$-linear. Note that this forces
\[
[\del_1, r\del_2] = r[\del_1, \del_2] + \del_1(r) \del_2 \qquad
(\del_1,\del_2 \in \Delta_R; r \in R).
\]
(In practice, the map
$\Delta_R \to \Der(R)$ will be injective, so the previous equation
will be automatic.)
\end{defn}

\begin{defn} \label{D:twisted}
For $R$ a differential ring, let
$R\{\Delta_R\}$ denote the \emph{twisted universal enveloping algebra}
of the Lie algebra $\Delta_R$. It may be described as the free associative
algebra over $\ZZ$ generated by the abelian group $R \oplus \Delta_R$, modulo
the two-sided ideal generated by the relations
\begin{align*}
r_1 r_2 - r_2 r_1 & \qquad (r_1, r_2 \in R) \\
\del_1 \del_2 - \del_2 \del_1 - [\del_1, \del_2] & \qquad (\del_1, \del_2 \in \Delta_R) \\
\del r - r \del - \del(r) & \qquad (\del \in \Delta_R, r \in R).
\end{align*}
The \emph{degree} of a nonzero element $x$ of $R\{\Delta_R\}$ is the smallest
nonnegative integer $s$ such that $x$ can be written as a sum of elements
of $R\{\Delta_R\}$, each of the form $r \del_1 \cdots \del_i$ for some $i \leq s$
and some $\del_1,\dots,\del_i \in \Delta_R$.
Note that the degree is always finite.
Let $R\{\Delta_R\}^{(s)}$ be the $R$-submodule of $R\{\Delta_R\}$
consisting of elements of degree at most $s$; these may be thought of as
\emph{differential operators of order at most $s$}.
\end{defn}

\begin{defn}
For $R$ a differential ring, a \emph{differential module} over $R$
is a left $R\{\Delta_R\}$-module $M$. 
In concrete terms, $M$ is an $R$-module equipped with 
an action of the Lie algebra $\Delta_R$ on $M$ (as a $\ZZ$-module), which on one hand
is $R$-linear in the sense that
\[
(r\del)(m) = r \del(m) \qquad (r \in R, \del \in \Delta_R, m \in M)
\]
and on the other hand is compatible with the derivation action of
$\Delta_R$ on $M$ in the sense that
\[
\del(rm) = \del(r) m + r \del(m) \qquad (\del \in \Delta_R, r \in R, m \in M).
\]
If $M$ is finite as an $R$-module, we call $M$ a 
\emph{finite differential module} over $R$. We write $H^0(M)$ for the joint kernel
of the $\del \in \Delta_R$ on $M$; the elements of $H^0(M)$ are also called the
\emph{horizontal} elements of $M$.

Let $M, N$ be differential modules over a differential ring $R$. We equip
$M \otimes_R N$ with the differential module structure
\[
\del(m \otimes n) = \del(m) \otimes n
+ m \otimes \del(n) \qquad (\del \in \Delta_R, m \in M, n \in N).
\]
We equip $\Hom_R(M,N)$ with the differential module structure
\[
\del(f)(m) = \del(f(m)) - f(\del(m)) \qquad
(\del \in \Delta_R, f \in \Hom_R(M,N), m \in M).
\]
In particular, the dual $M^\dual = \Hom_R(M, R)$ 
is again a differential module.
Note that if $M$ is a differential module over $R$ whose underlying $R$-module
is projective, then endomorphisms $M \to M$ of differential modules
are in bijection with horizontal elements of $\End(M) = M^\dual \otimes_R M$.
\end{defn}

\begin{example}
For any differential ring $R$
and any $r \in R$, 
we obtain a differential module $E(r)$ over $R$ whose underlying module
is free on one generator $\bv$, by specifying
\[
\del(\bv) = \del(r)\bv \qquad (\del \in \Delta_R).
\]
Note that $E(r) \otimes E(s) \cong E(r+s)$ and $E(r)^\dual \cong E(-r)$.
\end{example}

\begin{remark}
Let $R$ be a differential ring. Then for any multiplicative subset $S$
of $R$, we may extend the action of each $\del \in \Delta_R$ to 
the localization $S^{-1} R$ by declaring that
\[
\del(r s^{-1}) = \del(r) s^{-1} - \del(s) r s^{-2} \qquad (r \in R, s \in S).
\]
To confirm that this is well-defined, we must check that if $r \in R$ maps to
$0$ in $S^{-1} R$, then so does $\del(r)$. To see this, pick $s \in S$
such that $rs = 0$. Then $rs^2 = 0$ also, so
\[
0 = \del(rs^2) = \del(r) s^2 + 2 rs \del(s) = \del(r) s^2.
\]
We may thus view $S^{-1} R$ as a differential ring equipped with the 
module of derivations $S^{-1} \Delta_R$.
This is well-defined because if $\del \in \Delta_R$ maps to zero in
$S^{-1} \Delta_R$, then 
the image of $\del$ on $S^{-1} R$ is killed by some $s \in S$, so 
$\del$ acts as the zero derivation on $S^{-1} R$.
\end{remark}

\subsection{Locally simple differential rings}

\begin{defn}
We say a differential ring $R$ is \emph{locally simple} if
for each prime ideal $\gothp$ of $R$,
the local ring $R_\gothp$ is simple as a differential ring,
i.e., it contains no nonzero proper ideal stable under the action of 
$\Delta_R \otimes_R R_\gothp$.
\end{defn}

We use the following criterion to verify that a differential ring
is locally simple.
\begin{defn}
A noetherian local ring $R$ with maximal ideal $\gothm$
is \emph{regular} if $\dim_{R/\gothm} \gothm/\gothm^2 = \dim(R)$.
In this case, the natural map
\begin{equation} \label{eq:regular graded}
\Sym_{R/\gothm} \gothm/\gothm^2 \to \gr(R) = 
\bigoplus_{n=0}^\infty \gothm^n/\gothm^{n+1}
\end{equation}
of graded rings is an isomorphism \cite[Theorem~14.4]{matsumura}.
\end{defn}

\begin{prop} \label{P:locally simple}
Let $R$ be a regular local ring with maximal ideal $\gothm$
and residue field $\kappa$ having characteristic $0$.
Suppose $R$ is equipped with the structure of a differential ring in 
such a way that the pairing
\[
\Delta_R \times (\gothm/\gothm^2) \to \kappa
\]
is nondegenerate on the right.
Then $R$ is simple as a differential ring.
\end{prop}
\begin{proof}
Choose $x_1,\dots,x_r \in R$ which form a basis of $\gothm/\gothm^2$
over $\kappa$. 
By our hypothesis on $R$, we can find $\del_1,\dots,\del_r \in \Delta_R$
such that for $i,j \in \{1,\dots,r\}$,
\[
\del_i(x_j) \equiv \begin{cases} 1 \pmod{\gothm} & (i=j) \\
0 \pmod{\gothm} & (i \neq j).
\end{cases}
\]
For each nonnegative integer $j$,
let $S_j$ be the set of $(e_1,\dots,e_r) \in \ZZ$ with $e_1,\dots,e_r \geq 0$
and $e_1 + \cdots + e_r = j$. By \eqref{eq:regular graded}, the 
quantities $x_1^{e_1} \cdots x_r^{e_r}$ for $(e_1,\dots,e_r) \in S_j$
project to a basis of
$\gothm^j/\gothm^{j+1}$. Thus for each $y \in \gothm^j$, we can choose $c(y,E)\in R$
for each $E \in S_j$ so that
\[
y \equiv \sum_{E= (e_1,\dots,e_r)\in S_j} c(y, E) x_1^{e_1} \cdots x_r^{e_r} 
\pmod{\gothm^{j+1}},
\]
and the $c(y,E)$ are uniquely determined modulo $\gothm$.

Let $I$ be a nonzero differential ideal of $R$.
Let $j$ be the largest nonnegative integer such that $I \subseteq \gothm^j$.
Suppose by way of contradiction that $j > 0$; we can then pick
$y \in I \setminus \gothm^{j+1}$. 
Pick $i \in \{1,\dots,r\}$ such that $c(y,E) \not\equiv 0 \pmod{\gothm}$
for some $E
= (e_1,\dots,e_r) \in S_j$ with $e_i > 0$. Then
\[
\del_i(y) \cong \sum_{E \in S_j} e_i x_i^{-1} c(y, E) x_1^{e_1} \cdots 
x_r^{e_r} \pmod{\gothm^j}.
\]
In other words, for $(e_1,\dots,e_r) \in S_{j-1}$,
we have
\[
c(\del_i(y), (e_1,\dots,e_r)) \equiv (e_i + 1) c(y, (e_1,\dots,e_{i-1},
e_i + 1, e_{i+1},\dots,e_r)) \pmod{\gothm},
\]
and moreover this is nonzero modulo $\gothm$ for 
some choice of $(e_1,\dots,e_r) \in S_{j-1}$
(since $\kappa$ is of characteristic $0$).
We conclude that $\del_i(y) \notin \gothm^j$,
contradicting the choice of $j$.
Hence $j$ must equal $0$; this implies that $R$ is simple
as a differential ring, as desired.
\end{proof}

\begin{remark}
In Proposition~\ref{P:locally simple}, the hypothesis that the residue
field have characteristic $0$ is essential. For instance, if
$R = \FF_p \llbracket x \rrbracket$ and $\Delta_R$ is generated by
$\frac{\del}{\del x}$, then $x^p$ generates
a nonzero proper differential ideal. To get around this, one must work not
just with powers but also with \emph{divided powers} of the derivations,
i.e., expressions like $\del^n/n!$. (These are often called
\emph{Hasse-Schmidt derivations}.)
\end{remark}

\begin{example} \label{exa:locally simple}
Consider one of the following cases.
\begin{enumerate}
\item
Let $R$ be a smooth algebra over a field $k$ of characteristic zero,
viewed as a differential ring with $\Delta_R = \Hom_R(\Omega_{R/k}, R)$.
\item
Let $R$ be the local ring of a complex manifold, equipped with all
holomorphic derivations.
\item
Replace either of the previous examples with the completion along an ideal.
\end{enumerate}
In all of these cases,
Proposition~\ref{P:locally simple} implies that the resulting differential
ring is locally simple.
\end{example}

\begin{prop}
Let $R$ be a locally simple 
differential ring. Then any finite differential module
over $R$ is locally free.
\end{prop}
\begin{proof}
We immediately reduce to the case where $R$ is a local ring
with maximal ideal $\gothm_R$ and residue field $\kappa_R$.
Let $M$ be a finite differential module over $R$. 
Let $\be_1,\dots,\be_d$ be a basis of $M/\gothm_R M$ over $\kappa_R$;
then $\be_1,\dots,\be_d$ generate $M$ over $R$ by Nakayama's lemma
\cite[Theorem~2.2]{matsumura}.
Let $N$ be the $R$-submodule of $R^d$ consisting of tuples
$(c_1,\dots,c_d)$ such that $c_1 \be_1 + \cdots + c_d \be_d = 0$.
For $(c_1,\dots,c_d) \in N$, we have
\[
0 = \del(c_1 \be_1 + \cdots + c_d \be_d)
= \del(c_1) \be_1 + \cdots + \del(c_d) \be_d +
c_1 \del(\be_1) + \cdots + c_d \del(\be_d).
\]
By writing each $\del(\be_i)$ as an $R$-linear combination of 
$\be_1,\dots,\be_d$,
we obtain $(c'_1,\dots,c'_d) \in N$ with
\[
c'_i - \del(c_i) \in (c_1,\dots,c_d)R \qquad (i=1,\dots,d).
\]
This implies that the ideal generated by $c_1,\dots,c_d$ for all
$(c_1,\dots,c_d) \in N$ is a \emph{differential} ideal of $R$. 
If this ideal is nonzero,
then it must be the unit ideal by the hypothesis that $R$ is locally simple,
but this would produce a relation $c_1 \be_1 + \cdots + c_d \be_d = 0$
with $c_i \notin \gothm_R$ for some $i$, contradicting the linear
independence of $\be_1,\dots,\be_d$ in $M/\gothm_R M$. Hence the ideal
must be zero, which implies $N = 0$; that is, $M$ is free over $R$.
\end{proof}

\subsection{Based differential rings}

For explicit calculations, we will often need to fix generators of the
algebra of derivations; this corresponds to the geometric action of fixing
local coordinates on a smooth algebraic or analytic variety.

\begin{defn} \label{D:based}
A \emph{based differential ring of order $h$} is a differential ring
$R$ equipped with a basis $\del_1, \dots, \del_h$ of $\Delta_R$ as an
$R$-module, such that $[\del_i, \del_j] = 0$ for $i,j \in \{1,\dots,h\}$.
Given such data, we may identify $R\{\Delta_R\}$ with the 
\emph{ring of twisted polynomials} $R\{T_1,\dots,T_h\}$,
with $T_i$ corresponding to $\del_i$.
In case $h=1$, we often write $\del$ instead of $\del_1$.
\end{defn}

\begin{defn}
Let $(F, \del)$ be a based differential field of order 1. 
For $V$ a differential module over $F$ of rank $d$, a \emph{cyclic vector}
of $V$ is an element $\bv \in V$ such that
$\bv, \del(\bv), \dots, \del^{d-1}(\bv)$ are linearly independent over $F$.
Each cyclic vector determines an isomorphism
$V \cong F\{T\}/F\{T\}P(T)$ for some monic polynomial $P$,
under which $\bv$ maps to the class of $1 \in F\{T\}$.
\end{defn}

\begin{lemma} \label{L:cyclic}
Let $(F, \del)$ be a based differential field of order $1$,
such that $F$ has characteristic $0$ and $\del$ acts via a nonzero derivation
on $F$.
Then every differential module over $F$ contains at least one cyclic
vector.
\end{lemma}
\begin{proof}
This is the cyclic vector theorem 
\cite[Theorem III.4.2]{dgs},
\cite[Theorem~5.4.2]{kedlaya-course}.
\end{proof}

\subsection{Scales and absolute scales}

We now introduce nonarchimedean norms and define some spectral
invariants associated to $\nabla$-modules.
This largely follows \cite[Chapter~6]{kedlaya-course}.

\begin{defn}
For any abelian group $M$, a \emph{nonarchimedean seminorm} on $M$
is a nonconstant function $|\cdot|: M \to [0, +\infty)$
satisfying the following conditions.
\begin{enumerate}
\item[(a)]
For $x,y \in M$, $|x-y| \leq \max\{|x|, |y|\}$.
\item[(b)]
We have $|0| = 0$.
\end{enumerate}
Such a function is a \emph{nonarchimedean norm} on $M$
if in addition the following condition holds.
\begin{enumerate}
\item[(c)]
For any nonzero $x \in M$, $|x| > 0$.
\end{enumerate}
Two nonarchimedean (semi)norms $|\cdot|_1, |\cdot|_2$ on the same group $M$ are
\emph{metrically equivalent} if there exist $c_1,c_2 > 0$ such that
for all $x \in M$,
\[
c_1 |x|_1 \leq |x|_2 \leq c_2 |x|_1.
\]
This evidently defines an equivalence relation.
\end{defn}

\begin{defn}
For $M$ an abelian group equipped with a nonarchimedean (semi)norm $|\cdot|$, 
an endomorphism $T: M \to M$ is \emph{bounded} if there exists some
$c \geq 0$ such that for all $m \in M$, $|T(m)| \leq c|m|$.
The set of such $c$ has a least element, called the \emph{operator (semi)norm} of
$T$ on $M$, and denoted $|T|_M$.
For $T$ a bounded endomorphism on $M$,
for any nonnegative integers $s,t$, we have $|T^{s+t}|_M \leq |T^s|_M |T^t|_M$;
from this observation, it follows by an elementary analysis argument
(Fekete's lemma; see \cite[Problem~98]{polya-szego}) that 
\[
\lim_{s \to \infty} |T^s|_M^{1/s} = \inf_s \{|T^s|_M^{1/s}\}.
\]
This quantity is called the \emph{spectral radius} of $T$ on $M$,
and denoted $|T|_{\spect,M}$. Note that different but metrically equivalent
seminorms on $M$ give may 
different operator seminorms for the same $T$,
but give the same spectral radii.
\end{defn}

\begin{defn}
For $R$ a ring, a nonarchimedean (semi)norm on the additive group of $R$
is \emph{multiplicative} if it satisfies the following condition.
\begin{enumerate}
\item[(d)]
For $x,y \in R$, $|xy| = |x| |y|$.
\end{enumerate}
For $R$ a ring equipped with a nonarchimedean seminorm $|\cdot|$,
and $M$ an $R$-module, a seminorm $|\cdot|_M$ on $M$ is \emph{compatible with $R$}
if for all $r \in R$ and $m \in M$, $|rm|_M = |r| |m|_M$.
\end{defn}

\begin{lemma} \label{L:metric equiv}
Let $F$ be a field complete under a nonarchimedean norm $|\cdot|$.
Then for any finite-dimensional vector space $V$, any two nonarchimedean
norms on $V$ compatible with $F$ are metrically equivalent.
\end{lemma}
\begin{proof}
See \cite[Theorem~1.3.6]{kedlaya-course} or \cite[Proposition~4.13]{schneider}.
\end{proof}

\begin{defn}
By a \emph{nonarchimedean differential ring},
we will mean a differential ring $R$
equipped with a multiplicative nonarchimedean norm $|\cdot|$ under
which the action of each $\del \in \Delta_R$ is bounded.
We add the modifier \emph{complete} if $R$ is complete under $|\cdot|$.
We also allow substituting
more restrictive words for ``ring'' in this definition,
such as ``domain'' or ``field''.
\end{defn}

\begin{defn}
For $F$ a complete nonarchimedean differential field, $\del \in \Delta_F$,
and $V$ a nonzero finite differential module over $F$, we define the
\emph{spectral radius} $|\del|_{\spect,V}$ to be the spectral radius of $\del$
on the additive group of $V$ equipped with any nonarchimedean norm
compatible with $F$ (the choice being immaterial thanks to 
Lemma~\ref{L:metric equiv}). Note that 
$|\del|_{\spect,V} \geq |\del|_{\spect,F}$
(see \cite[Lemma~6.2.4]{kedlaya-course}).

In case $|\del|_{\spect,F} > 0$,
we define the \emph{scale} of $\del$ on $V$
as the ratio $|\del|_{\spect,V}/|\del|_{\spect,F}$;
it is always at least 1.
If $F$ is based of order $h$,
and $|\del_i|_{\spect,F} > 0$ for all $i$,
we define the \emph{based scale}
of $V$ to be the maximum of the scales of
$\del_1,\dots,\del_h$ on $V$.
\end{defn}

\begin{lemma} \label{L:spectral norm}
Let $F$ be a complete nonarchimedean differential field
with residue field of characteristic $0$.
Then for any $\del \in \Delta_F$ and any positive integer $n$, 
$|\del^n|_F = |\del|_F^n$; consequently, $|\del|_{\spect,F} = |\del|_F$.
\end{lemma}
\begin{proof}
We proceed by induction on $n$, the case $n=1$ being evident. Given the claim
for some $n$, for each $\epsilon > 0$  we can find some $x \in F^\times$
with $|\del^n(x)| \geq (1-\epsilon) |\del|_F^n |x|$.
We will show that either $y=x$ or $y=x^2$ satisfies
$|\del^{n+1}(y)| \geq (1-\epsilon)^2 |\del|_F^{n+1} |y|$;
this will yield the desired result.

We may assume that $|\del^{n+1}(x)| < (1-\epsilon) |\del|_F |\del^n(x)|$, as otherwise
$y=x$ works. Write
\[
\del^{2n}(x^2) = \sum_{i=0}^{2n} \binom{2n}{i} \del^i(x) \del^{2n-i}(x).
\]
The summand with $i=n$ has norm $|\del^n(x)|^2$ because the residue field
of $F$ has characteristic $0$. On the other hand, each summand with
$i > n$ has norm
\begin{align*}
|\del^i(x)| |\del^{2n-i}(x)|
&\leq |\del|_F^{i-n-1}| |\del^{n+1}(x)| |\del|_F^{2n-i} |x| \\
&< (1-\epsilon) |\del|_F^n |\del^n(x)| |x| \\
&\leq |\del^n(x)|^2,
\end{align*}
and similarly for $i < n$. It follows that $|\del^{2n}(x^2)| = |\del^n(x)|^2
\geq (1-\epsilon)^2 |\del|_F^{2n} |x|^2$,
so 
\[
|\del^{n+1}(x^2)| \geq |\del|_F^{-n+1} |\del^{2n}(x^2)|
\geq (1-\epsilon)^2 |\del|_F^{n+1} |x^2|.
\]
Hence $y=x^2$ works, proving the claim.
\end{proof}

One can also give a related definition in the unbased case.
This is suggested by work of Baldassarri and di Vizio 
\cite{baldassarri-divizio}.
\begin{defn}
Let $F$ be a complete nonarchimedean differential
field.
Let $V$ be a nonzero finite differential module on $F$.
Let $L(V)$ be the (not necessarily commutative)
ring of bounded endomorphisms of the additive group of $V$.
Let $D_V: F\{\Delta_F\} \to L(V)$ be the ring homomorphism
induced by the action of $\Delta_F$ on $V$.
For $s$ a nonnegative integer, let $D_{V,s}: F\{\Delta_F\}^{(s)} \to L(V)$
be the restriction of $D_V$ to $F\{\Delta_F\}^{(s)}$.

Equip $F\{\Delta_F\}$ with the norm $|D_F(\cdot)|_F$,
i.e., the operator norm for the action on $F$.
Choose a norm on $V$ compatible with $F$; this defines an operator norm on
$L(V)$. 
Using these norms, we can compute the operator norm $|D_{V,s}|$.
For any nonnegative integers $s,t$, we have the inequality
\[
|D_{V,s+t}| \leq |D_{V,s}| |D_{V,t}|.
\]
Again by Fekete's lemma, we have
\[
\lim_{s \to \infty} |D_{V,s}|^{1/s} = \inf_s \{|D_{V,s}|^{1/s}\}.
\]
We call this limit the \emph{absolute scale} of $V$; it again is independent
of the choice of the norm on $V$.
\end{defn}

\begin{defn} \label{D:irregularity}
Let $F$ be a complete nonarchimedean differential
field.
Let $V$ be a differential module of finite rank $n$ over $F$.
Let $V_1,\dots,V_m$ be the Jordan-H\"older constituents of $V$
in the category of differential modules over $F$.
Define the \emph{absolute scale multiset} of $V$ to be
the multiset consisting
of the absolute scale of $V_i$ with multiplicity $\dim_F V_i$
for $i = 1,\dots,m$.

Write the absolute scale multiset of $V$ as $\{s_1, \dots,s_n\}$
with $s_1 \geq \dots \geq s_n$. For $i=1,\dots,n$, define the
\emph{$i$-th (absolute) partial irregularity} of $V$ as 
\[
\irreg_i(V) = \sum_{j=1}^i \log s_j.
\]
We also call $\irreg_1(V)$ the \emph{(absolute) Poincar\'e-Katz rank} of $V$.
We also call $\irreg_n(V)$ the \emph{(absolute) irregularity} of $V$
and denote it by $\irreg(V)$.

If $F$ is based of order $h$ and $|\del_i|_{\spect,F} > 0$ for $i=1,\dots,h$,
we may similarly define the \emph{based scale multiset} 
of $V$, and refer to the \emph{based irregularity} and so forth.
\end{defn}

\begin{remark}
Note that our definition of irregularity depends on the normalization of
the absolute value. This dependence will be convenient when we study irregularity
as a function of a varying norm.
\end{remark}

\subsection{Derivations of rational type}

The absolute scale of a differential module is typically
difficult to compute. It is easier to compute when one can choose
a particularly nice basis of derivations.
\begin{defn}
Let $R$ be a nonarchimedean based differential
$\QQ$-algebra of order $h$.
(The corresponding definition in case of positive residual characteristics is
somewhat more delicate; see \cite[Definition~1.4.1]{kedlaya-xiao}.)
For $u_1,\dots,u_h \in R$, we say that the derivations $\del_1,\dots,\del_h$ 
are \emph{of rational type} with respect to $u_1,\dots,u_h$ if
the following conditions hold.
\begin{enumerate}
\item[(a)]
For $i,j \in \{1,\dots,h\}$,
\[
\del_i(u_j) = \begin{cases} 1 & i = j \\
0 & i \neq j. \end{cases}
\]
\item[(b)]
For $i \in \{1,\dots,h\}$, $|\del_i|_R \leq |u_i|^{-1}$. (By (a),
this implies $|\del_i|_R = |u_i|^{-1}$.)
\end{enumerate}
\end{defn}

\begin{lemma} \label{L:rational type norm}
Let $F$ be a complete nonarchimedean based differential
field of order $h$ with residue field of characteristic $0$,
such that $\del_1,\dots,\del_h$ are of rational type with respect
to some $u_1,\dots,u_h \in F$. Then
the norm $|D_F(\cdot)|_F$ on $F\{\Delta_F\} = 
F\{T_1,\dots,T_h\}$ coincides with the
$(|u_1|^{-1},\dots,|u_h|^{-1})$-Gauss norm
(see \eqref{eq:gauss}).
 In particular,
$|\del_i|_F = |\del_i|_{\spect,F} = |u_i|^{-1}$ for $i=1,\dots,h$.
\end{lemma}
\begin{proof}
On one hand, $|D_F(\cdot)|_F$ is bounded above by the
$(|u_1|^{-1},\dots,|u_h|^{-1})$-Gauss norm because
\[
|D_F(\del_1^{i_1} \cdots \del_h^{i_h})|_F
\leq |D_F(\del_1)|_F^{i_1} \cdots |D_F(\del_h)|_F^{i_h}
\leq |u_1|^{-i_1}\dots|u_h|^{-i_h}.
\]
On the other hand, given $P \in F\{T_1,\dots,T_h\}$ nonzero, write
\[
P = \sum_{j_1,\dots,j_h=0}^\infty P_{j_1,\dots,j_h} T_1^{j_1} \cdots T_h^{j_h}.
\]
Of the tuples $(k_1,\dots,k_h)$ for which $|P_{k_1,\dots,k_h} u_1^{-k_1}
\cdots u_h^{-k_h}|$ is maximal,
choose one which is minimal for the componentwise term order on $\ZZ^h$.
Let us consider the expression
\begin{equation} \label{eq:rational}
P_{j_1,\dots,j_h} \del_1^{j_1}\cdots \del_h^{j_h} (u_1^{k_1} \cdots u_h^{k_h})
\end{equation}
for some nonnegative integers $j_1,\dots,j_h$.
\begin{itemize}
\item
If $(j_1,\dots,j_h) = (k_1,\dots,k_h)$, then \eqref{eq:rational} equals
$k_1! \cdots k_h! P_{k_1,\dots,k_h}$.
This has norm $|P_{k_1,\dots,k_h}|$ because
$F$ has residual characteristic $0$, which forces $|k_1! \cdots k_h!| = 1$.
\item
If $(j_1,\dots,j_h) \neq (k_1,\dots,k_h)$ but
$|P_{j_1,\dots,j_h} u_1^{-j_1} \cdots u_h^{-j_h}| = |P_{k_1,\dots,k_h}
u_1^{-k_1}\cdots u_h^{-k_h}|$,
then there exists $i \in \{1,\dots,h\}$ for which $j_i > k_i$.
This forces \eqref{eq:rational} to equal 0.
\item
If
$|P_{j_1,\dots,j_h} u_1^{-j_1} \cdots u_h^{-j_h}| < |P_{k_1,\dots,k_h}
u_1^{-k_1}\cdots u_h^{-k_h}|$,
then \eqref{eq:rational} either equals 0 or has norm
\[
|P_{j_1,\dots,j_h} u_1^{k_1-j_1} \cdots u_h^{k_h-j_h}| < |P_{k_1,\dots,k_h}|.
\]
\end{itemize}
These together imply that
\[
|D_F(P)(u_1^{k_1}\cdots u_h^{k_h})| = |P_{k_1,\dots,k_h}|,
\]
and so
$|D_F(\cdot)|_F$ is bounded below by the
$(|u_1|^{-1},\dots,|u_h|^{-1})$-Gauss norm. 
\end{proof}

\begin{prop} \label{P:scale to absolute}
Let $F$ be a complete nonarchimedean based differential
field of order $h$ with residue field of characteristic $0$,
such that $\del_1,\dots,\del_h$ are of rational type with respect
to some $u_1,\dots,u_h \in F$. Then
for any nonzero finite
differential module $V$ over $F$, the based scale of $V$ equals the
absolute scale of $V$.
\end{prop}
\begin{proof}
(Compare \cite[Lemma~2.1.2]{andre-baldassarri} or
\cite[Proposition~6.3.1]{kedlaya-course}.)
Let $S_V^{\based}$ and $S_V^{\abs}$ denote the based scale 
and absolute scale of $V$,
respectively.
By taking $T_i^s \in F\{T_1,\dots,T_h\}^{(s)}$, we obtain the inequality
\[
|\del_i^s|_{V} \leq |D_{V,s}| |\del_i^s|_F = |D_{V,s}| |u_i|^{-s}\qquad 
(i=1,\dots,h),
\]
with the last equality requiring Lemma~\ref{L:spectral norm}
or Lemma~\ref{L:rational type norm}.
We may take $s$-th roots of both sides and then
take the limit as $s \to \infty$ to deduce
\begin{equation} \label{eq:bdv1}
S_V^{\based} = |u_i| \lim_{s \to \infty} |\del_i^s|_V^{1/s}
\leq \lim_{s \to \infty} |D_{V,s}|^{1/s} = S_V^{\abs}.
\end{equation}

Fix a norm on $V$ compatible with $F$.
Given $\epsilon \in (0, |D_{V,s}|)$, choose $c>0$ such that for $i=1,\dots,h$
and $j \geq 0$,
\[
|\del_i^{j}|_V \leq c 
(|\del_i|_{\spect,V} + \epsilon |u_i|^{-1} )^{j}.
\]
Given a nonnegative integer $s$,
choose $P = \sum_{j_1,\dots,j_h} P_{j_1,\dots,j_h} T_1^{j_1}\cdots T_h^{j_h}
\in F\{T_1,\dots,T_h\}^{(s)}$ 
nonzero such that $|D_V(P)|_V \geq |D_F(P)|_F (|D_{V,s}| - \epsilon)$.
Then by Lemma~\ref{L:rational type norm},
\begin{align*}
\max_{j_1,\dots,j_h} \{|P_{j_1,\dots,j_h} u_1^{-j_1} \cdots u_h^{-j_h}| (|D_{V,s}| - \epsilon)\}
&= |D_F(P)|_F (|D_{V,s}| - \epsilon) \\
&\leq |D_V(P)|_V \\
&\leq \max_{j_1,\dots,j_h} \{|P_{j_1,\dots,j_h} D_V(T_1^{j_1}\cdots T_h^{j_h})|_V\} \\
&\leq \max_{j_1,\dots,j_h} \left\{|P_{j_1,\dots,j_h}| 
c^h \prod_{i=1}^h (|\del_i|_{\spect,V} + \epsilon |u_i|^{-1} )^{j_i}\right\}.
\end{align*}
For the index $(j_1,\dots,j_h)$ which maximizes the right side, we have
\[
|u_1|^{-j_1} \cdots |u_h|^{-j_h} (|D_{V,s}| - \epsilon) \leq c^h
\prod_{i=1}^h (|\del_i|_{\spect,V} + \epsilon |u_i|^{-1})^{j_i}.
\]
Since $1 = |\del_i|_{\spect,F} |u_i| \leq |\del_i|_{\spect,V} |u_i| \leq S_V^{\based}$
and $j_1 + \cdots + j_h \leq s$, this implies
\[
|D_{V,s}| - \epsilon \leq c^h (S^{\based}_V + \epsilon)^s.
\]
Taking $s$-th roots of both sides, then taking the limit as $s \to \infty$,
yields
\begin{equation} \label{eq:bdv2}
\limsup_{s \to \infty} (|D_{V,s}| - \epsilon)^{1/s} 
\leq S_V^{\based} + \epsilon
\end{equation}
for any $\epsilon > 0$. Hence \eqref{eq:bdv2} holds also with $\epsilon = 0$;
this and \eqref{eq:bdv1} imply that
$S_V^{\abs} = S_V^{\based}$.
\end{proof}

\subsection{Newton polygons and spectral radii}

It was originally observed by Robba \cite{robba-hensel} that the usual
theory of Newton polygons for univariate
polynomials over a complete nonarchimedean field carries over nicely
to twisted polynomials. 
This provides an important mechanism for computing based scales,
and by extension absolute scales (in the rational type case).

\begin{defn}
Let $F$ be a nonarchimedean based differential field of order $1$.
For $P(T) = \sum_i P_i T^i \in F\{T\}$ nonzero, define the
\emph{Newton polygon} of $P$ to be the boundary of the lower convex hull
of the set
\[
\bigcup_i \{(x,y) \in \RR^2: x 
\geq -i, y \geq -\log |P_i| - (x+i) \log |\del|_F\}
\]
(with the convention that $(x+i) \log |\del|_F = 0$ for $x = -i$,
even if $|\del|_F = 0$).
This agrees with the usual Newton polygon except that all slopes greater
than $-\log |\del|_F$ are replaced with $-\log |\del|_F$.
For $r \leq -\log |\del|_F$, the \emph{multiplicity} of $r$ as a slope of
(the Newton polygon of) $P$ is defined as the horizontal width of the 
segment of the polygon having slope $r$.
One can check that the multiplicity of $r$ as a slope of 
the product $PQ$ equals the sum of the multiplicities of $r$ as a slope
of $P$ and of $Q$ \cite[Lemma~6.4.2]{kedlaya-course}.
\end{defn}

\begin{lemma} \label{L:factor slopes}
Let $F$ be a complete nonarchimedean based differential field of order $1$.
Then for any nonzero $P(T) \in F\{T\}$, there exists a unique factorization
$P = P_1 \cdots P_m$ such that the Newton polygon of $P_i$ has all slopes
equal to some value $r_i$, and $r_1 < \cdots < r_m$.
\end{lemma}
\begin{proof}
See \cite[Th\'eor\`eme~2.4]{robba-hensel}, 
\cite[Corollary~3.2.4]{kedlaya-part3},
or \cite[Theorem~6.4.4]{kedlaya-course}.
\end{proof}

\begin{prop} \label{P:read slopes}
Let $F$ be a complete nonarchimedean based differential field of order $1$
of characteristic $0$, such that $|\del|_F = |\del|_{\spect,F} > 0$.
\begin{enumerate}
\item[(a)]
Any finite differential module $V$ over $F$ admits a unique decomposition
\[
V = \bigoplus_{s \geq 1} V_s
\]
as a direct sum of differential submodules, such that every nonzero subquotient
of $V_s$ has based scale $s$.
\item[(b)]
For any isomorphism $V \cong F\{T\}/F\{T\}P(T)$ of left $F\{T\}$-modules,
and any $s \geq 1$, the multiplicity of $-\log s -\log |\del|_F$ as a slope
of $P$ equals $\dim_F V_s$.
\end{enumerate}
\end{prop}
\begin{proof}
Since $\del$ acts via a nonzero derivation on $F$,
we may apply Lemma~\ref{L:cyclic} to exhibit an isomorphism
$V \cong F\{T\}/F\{T\}P(T)$. By Lemma~\ref{L:factor slopes}
applied in both $F\{T\}$ and its opposite ring, we obtain a decomposition
of $V$ corresponding to the distinct slopes of $P$.
Both parts then follow from a calculation of Christol and Dwork; see
\cite[Corollary~6.5.4]{kedlaya-course}. (An analogous
calculation had been made previously by Malgrange \cite[\S 1]{malgrange}.)
\end{proof}

In the rational type setting, we may obtain a corresponding conclusion for
absolute scales.
\begin{prop} \label{P:read absolute slopes}
Let $F$ be a complete nonarchimedean based differential
field of order $h$ with residue field of characteristic $0$,
such that $\del_1,\dots,\del_h$ are of rational type with respect
to some $u_1,\dots,u_h \in F$. Then any 
finite differential module $V$ over $F$ admits a unique decomposition
\[
V = \bigoplus_{s \geq 1} V_s
\]
as a direct sum of differential submodules, such that every nonzero subquotient
of $V_s$ has absolute scale $s$.
\end{prop}
\begin{proof}
By Proposition~\ref{P:scale to absolute}, the based scale and absolute scale
agree for any nonzero finite differential module over $F$. We may thus work 
with
based scales hereafter. 

For $i=1,\dots,h$, since $|\del_i|_{\spect,F} = |\del_i|_F = 
|u_i|^{-1}$ by Lemma~\ref{L:rational type norm},
we may perform the decomposition of 
Proposition~\ref{P:read slopes}
by viewing $F$ as a based differential field of order 1 equipped
only with $\del_i$. However, since the decomposition
is unique, it must be respected by $\del_j$ for $j \neq i$.

Now take the minimal common refinement of the decompositions obtained 
for $i=1,\dots,h$.
Each summand of the resulting decomposition has the property that all of
its nonzero subquotients have the same based scale. We may thus group terms
to get a decomposition of the desired form, which is evidently unique.
\end{proof}

We also obtain an invariance of absolute scale multisets under suitable
extensions of the differential field.
\begin{lemma} \label{L:base change}
Let $F \subseteq F'$ be an inclusion of complete nonarchimedean based differential
fields of order $h$ with residue fields of characteristic $0$,
such that $\del_1,\dots,\del_h$ are of rational type with respect
to some $u_1,\dots,u_h \in F$. 
Then for any finite differential module $V$ over $F$,
the absolute scale multisets of $V$ and $V \otimes_F F'$ coincide.
\end{lemma}
\begin{proof}
Again by Proposition~\ref{P:scale to absolute}, we may consider based scale
multisets instead.
Moreover, it suffices to check the claim 
for $V$ irreducible.

Fix $i \in \{1,\dots,h\}$, and temporarily equip $F$ solely with $\del_i$.
Since $V$ is irreducible,
any nonzero $\bv \in V$
is a cyclic vector
and so corresponds to an isomorphism $V \cong F\{T_i\}/F\{T_i\}P_i$.
By Proposition~\ref{P:read slopes}, the Newton polygon of $P_i$ must
consist of a single element, otherwise $V$ would be decomposable.
That is, the scale multiset of $\del_i$ on $V$ consists solely of $s_i$.
Our hypotheses on $F$ and $F'$ ensure that the Newton polygon of $P_i$
does not change when we extend scalars to $F'$. Thus the scale
multiset of $V \otimes_F F'$, viewed as a differential module over $F$ equipped
solely with $\del_i$, again consists of the single element $s_i$.

Now equip $F$ and $F'$ again with all of $\del_1,\dots,\del_h$.
By the previous paragraph, the scale multisets of 
$V$ and $V \otimes_F F'$ both consist of
the single element $\max_i \{s_i\}$, proving the claim.
\end{proof}

\section{Differential algebra over discretely valued fields}
\label{sec:diff alg fields}

In this section, we focus on differential algebra over complete discretely
valued fields of equal characteristic $0$, i.e., Laurent series fields
over a field of characteristic $0$. Our main goal is to extend certain
results, like the Hukuhara-Levelt-Turrittin decomposition theorem, to
the case where the field carries not only a derivation with respect to
the series parameter, but also some derivations acting on the base field.
Much of the content of this section already appears in
the book of Andr\'e and Baldassarri \cite[Chapter~2]{andre-baldassarri}.

Note that this section includes a running hypothesis; see
Hypothesis~\ref{H:diff alg}.

\setcounter{theorem}{0}
\begin{defn}
For $R \subseteq S$ an inclusion of rings, and $M$ a finite $S$-module,
an \emph{$R$-lattice} in $M$ is a finite $R$-submodule $N$ of $M$
such that the natural map $N \otimes_R F \to M$ is an isomorphism.
Note that if $S$ is a localization of $R$ and 
$M$ is torsion-free, then an $R$-lattice in $M$
is just a finite $R$-submodule of $M$ which spans $M$ over $S$.
\end{defn}

\subsection{Setup}

We start by instituting a running hypothesis
and then investigating its initial consequences.

\begin{hypothesis} \label{H:diff alg}
Throughout \S \ref{sec:diff alg fields}, 
let $\gotho_F$ be a complete discrete valuation ring with norm $|\cdot|$
(of arbitrary normalization),
equipped with the structure of a \emph{locally simple} 
nonarchimedean differential ring.
Let $\gothm_F$ be the maximal ideal of $\gotho_F$. Let
$K = \gotho_F/\gothm_F$ be the residue field of $\gotho_F$.
Let $F$ be the fraction field of $\gotho_F$, 
viewed as a complete nonarchimedean
differential field.
Assume that $\Delta_F = \Delta_{\gotho_F} \otimes_{\gotho_F} F$ is finite
dimensional over $F$, and that $\Delta_{\gotho_F}$ is saturated (i.e.,
any $\del \in \Delta_F$ under which $\gotho_F$ is stable in fact belongs
to $\Delta_{\gotho_F}$).
Let $K_0$ be the joint kernel of $\Delta_F$ on $F$.
\end{hypothesis}

\begin{lemma} \label{L:find rational type1}
For any generator $z$ of $\gothm_F$,
there exists $\del \in \Delta_{\gotho_F}$ of rational type with respect to $z$.
\end{lemma}
\begin{proof}
Since $\gotho_F$ is locally simple, there exists $\del_0 \in \Delta_{\gotho_F}$
such that $\del_0(z) \in \gotho_F \setminus \gothm_F$.
Then $\del = \del_0(z)^{-1} \del_0$ is of rational
type with respect to $z$.
\end{proof}

\begin{lemma} \label{L:find rational type}
Let $z$ be a generator of $\gothm_F$. Suppose $\del \in \Delta_{\gotho_F}$
is of rational type with respect to $z$.
\begin{enumerate}
\item[(a)]
For each $m \in \ZZ$, 
$z \del$ acts on $z^m \gotho_F/z^{m+1} \gotho_F$ as multiplication by $m$.
\item[(b)]
The action of $z\del$ on $\gothm_F$ is bijective.
\item[(c)]
The kernel of $\del$ is contained in $\gotho_F$ and
projects bijectively onto $\gotho_F/\gothm_F \cong K$.
That is, there exists an isometric isomorphism $F \cong K((z))$
under which $\del$ corresponds to $\frac{\del}{\del z}$.
\item[(d)]
For any $\del' \in \Delta_F$ with 
$|\del'|_F \leq 1$, we have $|[z\del,\del']|_F \leq |z|$.
\item[(e)]
For any generator $z'$ of $\gothm_F$ and any
$\del' \in \Delta_{\gotho_F}$ of rational type with respect to $z'$,
$|z'\del' - z\del|_F \leq |z|$.
\item[(f)]
The joint kernel on $\gotho_F/\gothm_F \cong K$ of the action of
all $\del' \in \Delta_F$ with $|\del'|_F \leq 1$ equals $K_0$.
\end{enumerate}
\end{lemma}
\begin{proof}
Since $\del \in \Delta_{\gotho_F}$,
for $u \in \gotho_F$, we have 
\[
(z\del)(z^m u) = m z^m u + z^m (z\del)(u) \equiv m z^m u \pmod{z^{m+1} \gotho_F}.
\]
This proves (a). Note that the action on $z^m \gotho_F/z^{m+1} \gotho_F$
is bijective for $m \neq 0$. This proves (b).

By (a), we have 
$\ker(\del) \subseteq \gotho_F$. On the other hand, for any $x \in \gotho_F$,
we have $z \del(x) \in \gothm_F$, so by (b) there exists a unique $y \in \gothm_F$
with $z \del(x) = z \del(y)$. Thus $x - y$ is the unique element of 
$\ker(\del)$ congruent to $x$ modulo  $\gothm_F$. This proves (c).

Given $\del' \in \Delta_F$ with $|\del'|_F \leq 1$,
for each $m \in \ZZ$, the actions of $\del'$ and $z \del$ on $z^m \gotho_F /z^{m+1}
\gotho_F$ commute (because by (a), the latter is multiplication by $m$). This proves
(d).

Given a generator $z'$ of $\gothm_F$ and
$\del' \in \Delta_F$ of rational type with respect to $z'$,
by (a), for each $m \in \ZZ$, the action of both $z'\del'$ and $z\del$ on 
$z^m \gotho_F/z^{m+1} \gotho_F$ is multiplication by $m$.
This proves (e).

Suppose $u \in \gotho_F/\gothm_F \cong K$ 
is in the joint kernel of all $\del' \in \Delta_F$ with
$|\del'|_F \leq 1$. Let $x \in \gotho_F$ be the image of $u$
under the identification $K \cong \ker(\del)$ from (c).
We claim that $\del'(x) = 0$ for all
$\del' \in \Delta_F$ with $|\del'|_F \leq 1$.
Suppose the contrary; choose $\del' \in \Delta_F$
with $|\del'|_F \leq 1$ to maximize $|\del'(x)|$.
Since $\del'(x)$ reduces modulo $\gothm_F$ to $\del'(u) = 0$,
we must have $|\del'(x)| = |z|^m$ for some positive integer $m$.
Then by (a), we have $|(z\del)(\del'(x))| = |z|^m$ also.
On the other hand, since $(z\del)(x) = 0$, we have
\[
(z\del)(\del'(x)) = [z\del, \del'](x).
\]
By (d), $|z^{-1} [z\del, \del']|_F \leq 1$, so by our choice of
$\del'$, we must have $|z^{-1}[z\del, \del'](x)| \leq |z|^{m}$.
This gives a contradiction, which forces
$\del'(x) = 0$ for all $\del' \in \Delta_F$ with $|\del'|_F \leq 1$.
This proves (f).
\end{proof}

\begin{remark} \label{R:puiseux}
Given any continuous isomorphism $F \cong K((z))$
(as provided by Lemmas~\ref{L:find rational type1}
and~\ref{L:find rational type}(c)), 
any finite extension $F'$ of $F$ can be embedded into $K'((z^{1/m}))$
for some finite extension $K'$ of $K$ and some positive integer $m$,
by Puiseux's theorem \cite[\S~V.4, exercise 2]{bourbaki-alg47}.
Consequently, the integral closure $\gotho_{F'}$
of $\gotho_F$ again satisfies Hypothesis~\ref{H:diff alg}.
\end{remark}

\subsection{Regular differential modules}
\label{subsec:regulating1}

\begin{hypothesis}
Throughout \S~\ref{subsec:regulating1},
let $z$ be a generator of $\gothm_F$.
Let $\del \in \Delta_{\gotho_F}$
be a derivation of rational type with respect to $z$
(which exists by Lemma~\ref{L:find rational type1}).
Let $V$ be a finite differential module over $F$.
\end{hypothesis}

\begin{defn}
A subset $S$ of $K$ is \emph{prepared} if no nonzero integer
appears either as an element of $S$ or as a difference between two
elements of $S$. 
\end{defn}

\begin{defn}
Let $W$ be an $\gotho_F$-lattice in $V$ stable under $z\del$;
then $z\del$ acts on $W/zW$ as a $K$-linear transformation.
We say $W$ is a \emph{regulating lattice} in $V$ if the eigenvalues
of $z\del$ on $W$ are prepared;
these eigenvalues are called the \emph{exponents} of $W$.
We say $V$ is \emph{regular} if there exists a regulating lattice in $V$.
These definitions appear to depend on $z$
and $\del$, but in fact they do not; see Proposition~\ref{P:regulating}
and Corollary~\ref{C:regulating}.
\end{defn}

\begin{example} \label{exa:exp}
For $r \in F$, the canonical generator of $E(r)$ spans a regulating lattice
if and only if $r \in \gotho_F$, in which case the exponent of this lattice is 0.
\end{example}

\begin{lemma} \label{L:shearing}
Let $W$ be a $\gotho_F$-lattice in $V$ stable under
$z \del$.
Let $S$ be a subset of $\overline{K}$ stable under the absolute Galois group of 
$K$. Then there exists another $\gotho_F$-lattice $W'$ in $V$ stable under
$z\del$, such that the eigenvalues of $z\del$ on $W'/zW'$ are the same as
on $W/zW$ except that each $\alpha \in S$ is replaced by
$\alpha + 1$.
\end{lemma}
\begin{proof}
Since $S$ is Galois-stable, the direct sum of the generalized eigenspaces
of $(W/zW) \otimes_K \overline{K}$ corresponding to eigenvalues in $S$ 
(resp.\ not in $S$) descends to a subspace $X$ (resp.\ $Y$) 
of $W/zW$. Let $B_X$ (resp.\ $B_Y$) be a subset of $W$ lifting a basis 
of $X$ (resp.\ of $Y$).
Let $W'$ be the preimage of $Y$ in $W$; then $W'$ is also
stable under $z \del$, and a basis for $W'/zW'$ is given by
the images of $zB_X$ and $B_Y$. This proves the claim.
\end{proof}

\begin{remark}
The construction in Lemma~\ref{L:shearing} is classically known as a 
\emph{shearing transformation}.
\end{remark}

\begin{prop} \label{P:make regulating}
Suppose that there exists a $\gotho_F$-lattice $W$ in $V$ stable under
$z \del$. Then there exists a regulating lattice $W'$ in $V$ 
such that the eigenvalues of $z\del$ on $W/zW$ and on $W'/zW'$ project
to the same multisubset of $\overline{K}/\ZZ$.
\end{prop}
\begin{proof}
Suppose that $\alpha, \beta \in \overline{K}$ are Galois-conjugate
over $K$ and $\alpha - \beta \in \ZZ$. For $L$ the Galois closure
of $K(\alpha,\beta)$ over $K$, we have
\begin{align*}
[L:K] (\alpha - \beta) &= \Trace_{L/K}(\alpha - \beta) \\
&= \Trace_{L/K}(\alpha) - \Trace_{L/K}(\beta) = 0
\end{align*}
and so $\alpha = \beta$. 

By the previous paragraph, any two eigenvalues of $z\del$ on $W/zW$ which 
differ by a nonzero integer belong to different Galois orbits over $K$. Hence
by applying Lemma~\ref{L:shearing} repeatedly,
we can obtain a new $\gotho_F$-lattice $W'$ stable under $z\del$, such that the eigenvalues
of $z\del$ on $W'/zW'$ are prepared but are congruent modulo $\ZZ$
to the eigenvalues of
$z\del$ on $W/zW$.
This proves the claim.
\end{proof}

\begin{prop} \label{P:regulating}
Let $W$ be a regulating lattice in $V$.
Then $W$ is stable under any $\del' \in \Delta_F$ with $|\del'|_F \leq 1$.
\end{prop}
\begin{proof}
Since $\Delta_F$ is finite-dimensional over $F$, there must exist a
nonnegative integer 
$m$ such that for all $\del' \in \Delta_F$ with $|\del'|_F \leq 1$,
we have $\del'(W) \subseteq z^{-m} W$. Suppose $m>0$ is such an integer;
then for any $n \in \ZZ$ and $\bv \in W$,
\[
\del'(z^n \bv) = n z^{n-1} \del'(z) \bv + z^n \del'(\bv) \in z^{-m+n} W.
\]
Let $P(T) = \sum_i P_i T^i \in K[T]$ 
be the characteristic polynomial of $z\del$ on $W/zW$.
Choose a monic polynomial $\tilde{P}(T) = \sum_i \tilde{P}_i T^i \in \gotho_F[T]$
lifting $P(T)$.
Then for $\bv \in W$, $\tilde{P}(z\del)(\bv) \in zW$, so
\[
\sum_i \del'(\tilde{P}_i) (z\del)^i(\bv)
+ \sum_i \tilde{P}_i (z\del)^i(\del'(\bv))
+ \sum_i \tilde{P}_i [\del', (z\del)^i](\bv) = \del'(\tilde{P}(z\del)(\bv)) \in z^{-m+1}W.
\]
In this equation, $[\del',(z\del)^i] \in z \Delta_{\gotho_F}$ by repeated application
of
Lemma~\ref{L:find rational type}(d), so the third sum over $i$ belongs to
$z^{-m+1}W$, while the first sum belongs to $W$. Consequently, the second sum,
which is $\tilde{P}(z\del)(\del'(\bv))$, belongs to $z^{-m+1}W$. However, 
for any $\bw \in W$ and any $m \in \ZZ$,
\[
\tilde{P}(z\del)(z^{-m} \bw) = z^{-m} \tilde{P}(z \del - m)(\bw),
\]
and the action of $P(z \del-m)$ on $W/zW$ is invertible because $z\del$
acts with prepared eigenvalues. It follows that $\del'(\bv) \in z^{-m+1}W$.

In other words, if $m > 0$, then 
also $\del'(W) \subseteq z^{-m+1}W$
for all $\del' \in \Delta_F$ with $|\del'|_F \leq 1$. 
We conclude that $\del'(W)
\subseteq W$ for all such $\del'$, as desired.
\end{proof}

\begin{cor} \label{C:regulating}
Let $W$ be a regulating lattice in $V$.
For any generator $z'$ of $\gothm_F$ and any $\del' \in \Delta_{\gotho_F}$ of
rational type with respect to $z'$, $W$ is stable under $z' \del'$,
and the eigenvalues of $z'\del'$ on $W/zW$ are the same as those of $z \del$.
\end{cor}
\begin{proof}
The stability of $W$ under $z'\del'$ follows from Proposition~\ref{P:regulating}.
The equality of the eigenvalues follows from 
the fact that $|z'\del'-z\del|_F \leq |z|$ by Lemma~\ref{L:find rational type}(e),
so Proposition~\ref{P:regulating} implies that $W$ is stable under
$(z'\del'-z\del)/z$.
\end{proof}

\begin{prop} \label{P:regular}
For $V$ a nonzero finite differential module over $F$,
the following conditions are equivalent.
\begin{enumerate}
\item[(a)]
$|z\del|_{\spect,V} \leq 1$
(which implies $|z\del|_{\spect,V} = 1$).
\item[(b)]
$|\del|_{\spect,V} \leq |z|^{-1}$
(which implies $|\del|_{\spect,V} = |z|^{-1}$).
\item[(c)]
The absolute scale of $V$ is at most $1$ (which implies that it is
equal to $1$).
\item[(d)]
$V$ is regular.
\end{enumerate}
\end{prop}
\begin{proof}
Put $n = \dim_F V$.
Apply Lemma~\ref{L:cyclic} to produce a cyclic vector $\bv \in V$.
Given (a), the $\gotho_F$-lattice in $V$ spanned by
$(z\del)^i(\bv)$ for $i=0,\dots,n-1$ is stable under $z\del$ by
Proposition~\ref{P:read slopes}. 
Similarly, given (b), the $\gotho_F$-lattice in $V$ spanned by
$z^i \del^i(\bv)$ for $i=0,\dots,n-1$ is stable under $z\del$ by
Proposition~\ref{P:read slopes}. 
By
Proposition~\ref{P:make regulating}, there exists a regulating lattice in $V$.
Hence (a) and (b) each imply (d).

Given (d), choose a regulating lattice in $W$ and equip
$V$ with the supremum norm defined by a basis of $W$.
Then for any $\del' \in \Delta_F$ with $|\del'|_F \leq 1$,
$|\del'|_V \leq 1$ by Proposition~\ref{P:regulating}.
This implies (c).

Given (c), (a) and (b) follow at once. This completes
the argument.
\end{proof}
\begin{cor} \label{C:regular}
If $0 \to V_1 \to V \to V_2 \to 0$ is a short exact sequence
of finite differential modules over $F$, then
$V$ is regular if and only if both $V_1$ and $V_2$ are regular.
\end{cor}
\begin{proof}
If $V$ is regular, then so are $V_1$ and $V_2$ by Proposition~\ref{P:regular}(c).
Conversely, if $V_1$ and $V_2$ are regular, let
$W_1,W_2$ be regulating lattices in $V_1,V_2$.
Choose a basis of $W_2$, lift it to $V$, and
let $\tilde{W}_2$ be the $\gotho_F$-span of the result.
Then for $m$ a sufficiently large integer,
$W = z^{-m} W_1 + \tilde{W}_2$ is an $\gotho_F$-lattice of $V$ stable under
$z \del$. By Proposition~\ref{P:make regulating}, $V$ is regular.
\end{proof}

\begin{prop} \label{P:exponents in base}
Let $W$ be a regulating lattice in $V$. Then the characteristic
and minimal 
polynomials of $z \del$ on $W/zW$ have coefficients in $K_0$; consequently,
the exponents of $W$ belong to $\overline{K_0}$.
\end{prop}
\begin{proof}
Let $P(T) = \sum_i P_i T^i\in K[T]$ 
be the minimal polynomial of $z\del$ on $W/zW$.
For any $\del' \in \Delta_F$ with $|\del'|_F \leq 1$ and any $\bv \in W/zW$,
we then have
\[
0 = \del'(P(z\del)(\bv))
= \sum_i \del'(P_i) (z\del)^i(\bv)
+ \sum_i P_i [\del',(z\del)^i](\bv)
+ \sum_i P_i (z\del)^i(\del'(\bv)).
\]
Note that $|[\del',z\del]|_F \leq |z|$ by Lemma~\ref{L:find rational type}(d),
so Proposition~\ref{P:regulating} implies that
$[\del',(z\del)^i](\bv)$ vanishes in $W/zW$. Hence the second sum
on the right side vanishes. The third sum also vanishes because it is just
$P(z\del)(\del'(\bv))$. Thus the first sum must also vanish, proving that
the coefficients of $P$ are killed by any $\del' \in \Delta_F$
with $|\del'|_F \leq 1$. By Lemma~\ref{L:find rational type}(f),
the coefficients of $P$ must belong to $K_0$; this implies the same for 
the characteristic polynomial of $z\del$ on $W/zW$, as desired.
\end{proof}

The following construction, while useful, cannot be formulated
in a manner independent of the choice of $z$ and $\del$.

\begin{lemma} \label{L:fundamental solution}
Let $W$ be a regulating lattice in $V$.
Then there is a unique $K$-lattice $W_0$ in $W$ stable under $z \del$.
\end{lemma}
\begin{proof}
As in Lemma~\ref{L:find rational type}(c), we identify $F$ with $K((z))$
in such a way that $\del$ corresponds to $\frac{\del}{\del z}$.
Let $\be_1,\dots,\be_d$ be a basis of $W$.
Let $N$ be the matrix of action of $z \del$ on $\be_1,\dots,\be_d$.
We will show that there is a unique $d \times d$ 
matrix $U = \sum_{i=0}^\infty U_i z^i$ over $K \llbracket z \rrbracket$
with $U_0$ equal to the identity matrix, such that
\[
N U + z \del(U) = U N_0.
\]
To see this, let $T$ be the linear transformation $X \mapsto N_0 X - X N_0$ on 
$d \times d$ matrices over $K$. 
The eigenvalues of $T$ are the pairwise differences
between eigenvalues of $N_0$. Since $N_0$ has
prepared eigenvalues, $T+i$ is invertible for all $i \neq 0$;
consequently, if we have computed $U_j$ for $j < i$, then $U_i$ is determined 
uniquely by the equation
\begin{equation} \label{eq:solve de}
i U_i = U_i N_0 - N_0 U_i - \sum_{j=1}^i N_j U_{i-j}.
\end{equation}
Given the existence and uniqueness of $U$, we obtain $W_0$
as the $K$-span of the basis $\bv_1,\dots,\bv_d$ of $V$ given by
$\bv_j = \sum_i U_{ij} \be_i$.
This proves the desired result.
\end{proof}

\begin{remark} \label{R:formal product}
For $W_0$ a $K$-lattice in $V$ stable under $z\del$, we can formally
write each $\bv \in V$ as $\sum_{i \in \ZZ} \bv_i z^i$ with $\bv_i \in W_0$;
that is, we can embed $V$ into the product of $z^i W_0$ over $i \in \ZZ$.
In this representation, we can compute the action of $z\del$ by the formula
\[
(z\del)(\bv) = \sum_{i \in \ZZ} (z\del + i)(\bv_i) z^i.
\]
\end{remark}

\begin{prop} \label{P:same exponents}
Let $W, W'$ be $\gotho_F$-lattices in $V$ stable under $z\del$.
\begin{enumerate}
\item[(a)]
The exponents of $W$ and $W'$ coincide as multisubsets of
$\overline{K_0}/\ZZ$.
\item[(b)]
Given any prepared multisubset of $\overline{K_0}$ which 
is stable under the absolute Galois group of $K_0$ and
coincides with the
exponents of $W$ as a multisubset of $\overline{K_0}/\ZZ$, there exists 
a unique choice of $W'$ with this multisubset as its exponents.
\end{enumerate}
\end{prop}
\begin{proof}
Identify $F$ with $K((z))$ as in Lemma~\ref{L:find rational type}(c).
\begin{enumerate}
\item[(a)]
By Proposition~\ref{P:make regulating}, we may replace
$W$ and $W'$ by regulating lattices 
without changing their exponents modulo $\ZZ$. We may thus assume that
$W$ and $W'$ have prepared exponents.
We may also assume as in Lemma~\ref{L:find rational type}(c) that
$F \cong K((z))$; we may then extend the constant field from 
$K_0$ to $\overline{K_0}$, to reduce to the case where $K_0$ is algebraically 
closed.

By Lemma~\ref{L:fundamental solution}, there exist unique $(z\del)$-stable
$K$-lattices $W_0, W'_0$ in $W, W'$, respectively.
Decompose $W_0, W'_0$ into generalized
eigenspaces $W_{0,\lambda}, W'_{0,\lambda'}$ for $z \del$.
As in Remark~\ref{R:formal product},
view $V$ as an $F$-vector subspace of the product 
of $z^m W_0$ over $m \in \ZZ$.
On this product, for $\mu \in K_0$, $z \del - \mu$ fails to be invertible
if and only if $\mu = \lambda + m$ for some eigenvalue $\lambda$
of $z \del$ on $W_0$ and some $m \in \ZZ$. Consequently,
for each eigenvalue $\lambda'$ of $z \del$ on $W'_0$,
there exists an integer $m = m(\lambda')$ such that
$\lambda = \lambda' - m(\lambda')$ is an eigenvalue of $z \del$.
Put
\[
W''_0 = \bigoplus_{\lambda'} z^{-m(\lambda')} W'_{0, \lambda'};
\]
then $W'' = W''_0 \otimes_K \gotho_F$ is also
a regulating lattice in $V$, and $W''_0$ is a $(z\del)$-stable
$K$-lattice in $W''$. However, we must have
$z^{-m(\lambda')} W'_{0, \lambda'} \subseteq W_{0, \lambda}$ for
$\lambda = \lambda' - m(\lambda')$, so $W''_0 \subseteq W_0$.
Since these are both $K$-lattices in $V$, we must have
$W''_0 = W_0$ and $W = W''$. By construction, the exponents 
of $W = W''$ and $W'$
coincide as multisubsets of $\overline{K_0}/\ZZ$.
\item[(b)]
From the proof of (a), it is clear that we must take
$W' = W'_0 \otimes_K \gotho_F$ for 
$W'_0$ equal to the direct sum of $z^{m(\lambda)} W_{0,\lambda}$,
with $m(\lambda)$ being the unique integer such that
$\lambda + m(\lambda)$ appears in the chosen multiset.
\end{enumerate}
\end{proof}

\begin{defn}
For $V$ a regular finite differential module over $F$,
define the \emph{exponents} of $V$ to be the multiset of
exponents of a regulating lattice $W$ in $V$, viewed within
$\overline{K_0}/\ZZ$. By Proposition~\ref{P:same exponents}, this is
independent of the choice of the lattice.
For instance, for $r \in \gotho_F$, $E(r)$ is regular with exponent 0
by Example~\ref{exa:exp}.
\end{defn}

\subsection{Hukuhara-Levelt-Turrittin decompositions}

We now prove an extension of the usual Hukuhara-Levelt-Turrittin
decomposition theorem, following Levelt \cite{levelt}.

\begin{lemma}[Levelt] \label{L:levelt}
Let $V$ be a nonzero finite differential module over $F$. Then there
exist a finite extension $F'$ of $F$ and
an element $r \in F'$ such that
$E(-r) \otimes_{F'} (V \otimes_F F')$ has a (nonzero) regular
direct summand.
\end{lemma}
\begin{proof}
We induct on $\dim_F V$. 
Let $z$ be a generator of $\gothm_F$.
Let $\del \in \Delta_{\gotho_F}$
be a derivation of rational type with respect to $z$.
As in Lemma~\ref{L:find rational type}(c), identify $F$ with $K((z))$
so that $\del$ corresponds to $\frac{\del}{\del z}$.

Apply Lemma~\ref{L:cyclic} to choose an isomorphism
$V \cong F\{T\}/F\{T\}P(T)$ of differential 
modules over $F$ with respect to
$z \del$ alone, for some monic twisted polynomial $P(T) \in F\{T\}$. 
If $P(T)$ has more than one slope, then $V$ is decomposable
by Proposition~\ref{P:read slopes}, so we may invoke the induction hypothesis
to conclude.

Suppose instead that $P(T)$ has only one slope.
If that slope is 0, then $V$ is regular by Proposition~\ref{P:regular}
and we are done.
Otherwise, write
$P(T) = \sum_{j=0}^{d} P_j T^j$ with $P_d = 1$.
Define the integer $e \in \ZZ$ by $|P_0| = |z|^e$,
so that $|P_{d-j}| \leq |z|^{ej/d}$ for all $j$ (because $P$ has only
one slope).
Let $P_{j,0} \in K$ be the coefficient of $z^{e(d-j)/d}$ in the series expansion
of $P_j$, which we take to be $0$ if $e(d-j)/d \notin \ZZ$. 
Then each root of the untwisted 
polynomial $Q(T) = \sum_{j=0}^d P_{j,0} z^{e(d-j)/d} T^j$ has the form 
$c z^{e/d}$ for some
$c \in \overline{K}$. Moreover, for any such root $r$
lying in $F' = K'((z^{1/d}))$ for some finite extension $K'$ of $K$,
$E(-r) \otimes_{F'} (V \otimes_F F')$
has a cyclic vector with corresponding polynomial $P(T - r)$,
which has at least one slope strictly greater than the unique slope of $P$.
If $Q$ has more than one distinct root, 
then $P(T - r)$ also has a slope equal to the unique slope of $P$.
By Proposition~\ref{P:read slopes}, 
$V \otimes_F F'$ is decomposable; we may then
invoke the induction hypothesis to conclude. 

Otherwise, we must have
$r \in F$, and $E(-r) \otimes_F V$ has strictly smaller scale
with respect to $z \del$ than $V$ does. 
We repeat this process until we either encounter an invocation of the
induction hypothesis, or arrive at the case where $V$ is regular.
This process must terminate because otherwise 
the scale of $z \del$ would form a strictly increasing sequence taken from
a discrete bounded set. Thus the claim follows.
\end{proof}

\begin{defn} \label{D:twist-regular}
A finite differential module $V$ over $F$ is \emph{twist-regular} if 
$\End(V)$ is regular. For instance, if $V = E(r) \otimes_F R$
with $r \in F$ and $R$ regular, then $V$ is twist-regular.
Conversely, Theorem~\ref{T:Turrittin} below implies that every
twist-regular module has the form $E(r) \otimes_F R$ for suitable $r, R$.

Note that if $V$ is twist-regular and $r \in F$ is such that
$E(-r) \otimes_F V$ has a nonzero regular direct summand
$M$, then $E(-r) \otimes_F V$ must be regular. Namely, it suffices 
to check this for $r=0$; in this case, $M^\dual \otimes_F V$ is a direct
summand of the regular module $\End(V)$, so it is
regular by Corollary~\ref{C:regular}.
Then $V$ occurs as a direct summand of the regular
module $M \otimes_F (M^\dual \otimes_F V)$, so it too is regular.
\end{defn}

\begin{theorem} \label{T:Turrittin}
Let $V$ be a finite differential module over $F$.
\begin{enumerate}
\item[(a)]
For some finite extension $F'$ of $F$,
there exists a unique direct sum decomposition
$V \otimes_F F' = \oplus_{j \in J} V_j$ of differential modules,
in which for $j,k \in J$,
$V_j^\dual \otimes_{F'} V_k$ is regular if and only if $j = k$.
In particular, each summand is twist-regular.
\item[(b)]
For any decomposition as in (a), there exists $r_j \in F'$
such that $E(-r_j) \otimes_{F'} V_j$ is regular for each $j \in J$.
In particular, if $V$ itself 
is twist-regular, we can choose $r \in F$ such that $E(-r) \otimes_F V$
is regular.
\end{enumerate}
\end{theorem}
\begin{proof}
By Lemma~\ref{L:levelt}, for some $F'$, we can find $r \in F'$ such that
$V \otimes_F F'$ splits as a direct sum in which one summand becomes
regular upon twisting by $E(-r)$; in particular, this summand is 
twist-regular. By repeating this argument, we find that for some 
$F'$, $V \otimes_F F'$ splits as a direct sum of twist-regular summands.
From this, (a) follows easily.

To deduce (b), it suffices to check that if $V$ is twist-regular,
then we can find $r \in F$ such that $E(-r) \otimes_F V$ is regular.
By Lemma~\ref{L:levelt} again, for some $F'$, we can find $r \in F'$
such that $E(-r) \otimes_{F'} (V \otimes_F F')$ has a regular direct
summand; however, since $V$ is twist-regular, this forces
$E(-r) \otimes_{F'} (V \otimes_F F')$ itself to be regular
(as in Definition~\ref{D:twist-regular}).

Let $z$ be a generator of $\gothm_F$.
Apply Lemma~\ref{L:find rational type1}
to produce a derivation $\del \in \Delta_{\gotho_F}$
of rational type with respect to $z$.
As in Lemma~\ref{L:find rational type}(c), identify $F$ with $K((z))$
so that $\del$ corresponds to $\frac{\del}{\del z}$.
As in Remark~\ref{R:puiseux},
we may identify $F'$ with a subfield of $K'((z^{1/m}))$
for some finite extension $K'$ of $K$ and some positive integer $m$.
Since $r$ is only determined modulo $K' \llbracket z^{1/m} \rrbracket$
by the condition that $E(-r) \otimes_{F'} (V \otimes_F F')$ is regular,
we may as well take $r \in z^{-1/m} K' [z^{-1/m}]$.
In that case, each Galois conjugate $r'$ of $r$ also belongs to
$z^{-1/m} K'[z^{-1/m}]$, and in fact can be obtained by applying an automorphism
of $K'$ over $K$ and then replacing $z^{-1/m}$ with another $m$-th root
of $z^{-1}$. Moreover, $E(-r') \otimes_{F'} (V \otimes_F F')$ is also
regular, so for all $r'$ 
we must have 
\[
r' - r \in K' \llbracket z^{1/m} \rrbracket
\cap z^{-1/m} K'[z^{-1/m}] = 0.
\]
This forces $r \in F$, proving the claim.
\end{proof}

\begin{defn} \label{D:twist-regular exponents}
Let $V$ be a twist-regular finite differential module over $F$.
By Theorem~\ref{T:Turrittin}, there exists $r \in F$ such that
$E(-r) \otimes_F V$ is regular. We define the \emph{exponents} of
$V$ to be the exponents of $E(-r) \otimes_F V$, again as a multisubset
of $\overline{K_0}/\ZZ$. Note that this definition does not depend on the choice
of $r$: if $r' \in F$ is such that $E(-r') \otimes_F V$ is regular, then
$E(r)^\dual \otimes_F E(r') \cong E(r'-r)$ is regular,
and so has exponent 0 by Example~\ref{exa:exp}.
\end{defn}

Theorem~\ref{T:Turrittin} and Definition~\ref{D:twist-regular exponents}
together do not suffice to give a definition of the exponents of an arbitrary
finite differential module over $F$. Although we will not use the more
general definition here, we include it for completeness.

\begin{defn}
A finite differential module $V$ over $F$ is 
\emph{nearly twist-regular} if there exists a finite extension $F'$
of $F$ such that $V \otimes_F F'$ admits a decomposition as in
Theorem~\ref{T:Turrittin}(a), for which for any $j,k \in J$,
there is a Galois conjugate $V'_k$ of $V_k$ such that $V_j^\dual \otimes V'_k$
is regular. 
For $V$ nearly twist-regular, we define the \emph{exponents} of $V$ as follows.
Let $m$ be the ramification index of $F'$ over $F$.
Let $S$ be the exponents of $V_j$ for some $j \in J$.
We then define the exponents of $V$ to be the multiset consisting of
\[
\frac{\alpha + h}{m} \qquad (\alpha \in S; h \in \{0,\dots,m-1\}).
\]
This does not depend on the choice of $j \in J$.

By Theorem~\ref{T:Turrittin}, any finite differential module $V$ admits a 
unique minimal decomposition into nearly twist-regular summands. Using
the previous paragraph, we may define exponents of $V$;
this agrees with the construction of \emph{formal exponents}
given in \cite{corel}.

If $V$
descends to the ring of germs of meromorphic functions at the origin
in the complex plane, then one might hope that our algebraically defined
exponents coincide with the \emph{topological exponents}, i.e., the logarithms
of the eigenvalues of the monodromy transformation.
However, the formal and topological exponents need not coincide in general.
\end{defn}

\subsection{Deligne-Malgrange lattices}

With an algebraic definition of exponents in hand, one can give an algebraic
definition of Deligne-Malgrange lattices, following 
Malgrange's original construction of the \emph{reseau canonique}
\cite[\S 3]{malgrange-conn2}. 
We will not make essential use of this construction; however,
it figures prominently in some other work on this topic,
notably that of Mochizuki \cite{mochizuki2}.

\begin{defn}
Let $\tau: \overline{K_0}/\ZZ \to \overline{K_0}$ be a section of the quotient 
$\overline{K_0} \to \overline{K_0}/\ZZ$.
We say $\tau$ is \emph{admissible} if $\tau(0) = 0$,
$\tau$ is equivariant for the action of the absolute Galois group of $K_0$,
and for any $\lambda \in \overline{K_0}$ and
any positive integer $a$, we have
\begin{equation} \label{eq:admissible}
\tau(\lambda) - \lambda = 
\left\lceil \frac{\tau(a\lambda) - a\lambda}{a} \right\rceil.
\end{equation}
\end{defn}

\begin{example}[Malgrange] \label{ex:malgrange}
If $K_0 = \CC$, the section $\tau$ with image
$\{s \in \CC: \Real(s) \in [0,1)\}$ is admissible, by the following
argument. The truth of \eqref{eq:admissible} for a given
$\lambda$ is equivalent to its truth for $\lambda + 1$, so we may
reduce to the case $\Real(\lambda) \in [0,a)$.
In that case, the left side of \eqref{eq:admissible} equals 0.
Meanwhile, $\tau(a\lambda) - a\lambda \in \{1-a, \dots, 0\}$,
so the right side of \eqref{eq:admissible} also equals 0.
\end{example}

Following this example, we can construct admissible sections in all cases.
\begin{lemma} \label{L:admissible}
For any $K_0$, there exists an admissible section  $\tau$ of
the quotient $\overline{K_0} \to \overline{K_0}/\ZZ$.
\end{lemma}
\begin{proof}
For $\lambda \in \QQ$, we define $\tau(\lambda)$ as
in Example~\ref{ex:malgrange}.
To extend $\tau$ to all of $\overline{K_0}$, let $G$ be the group of 
affine transformations on $\overline{K_0}$ defined over $\QQ$, i.e., the
group of maps of the form $x \mapsto cx + d$ with $c,d \in \QQ$
and $c \neq 0$. The group $G$ acts freely on $\overline{K_0} \setminus \QQ$
and commutes with the Galois action;
choose a Galois-equivariant section $\sigma: 
(\overline{K_0} \setminus \QQ)/G \to \overline{K_0} \setminus \QQ$ 
of the quotient by the $G$-action.
For $\mu \in \overline{K_0} \setminus \QQ$, 
choose $c(\mu) \in \QQ \setminus \{0\}$ and $d(\mu) \in \QQ$
so that $\mu = c(\mu) \sigma(\mu) + d(\mu)$,
then set $\tau(\mu) = \tau(d(\mu))$.
\end{proof}

\begin{defn} \label{D:dm lattice}
Let $\tau: \overline{K_0}/\ZZ \to \overline{K_0}$ be an admissible 
section of the quotient $\overline{K_0} \to \overline{K_0}/\ZZ$;
such a $\tau$ is guaranteed to exist by Lemma~\ref{L:admissible}.
Let $V$ be a finite differential module over $F$. We define the 
\emph{Deligne-Malgrange
lattice} of $V$ with respect to $\tau$ as follows.

Suppose first that $V$ is regular. The Deligne-Malgrange lattice of $V$
is then the unique regulating lattice whose exponents are
in the image of $\tau$; the existence and uniqueness of such a regulating
lattice follow from Proposition~\ref{P:same exponents}.

Suppose next that $V$ is twist-regular. 
Apply Theorem~\ref{T:Turrittin} to construct $r \in F$ such that
$V' = E(-r) \otimes_F V$ is regular.
Let $W'$ be the Deligne-Malgrange lattice of $V'$.
The Deligne-Malgrange lattice of $V$ is defined as 
the lattice $W = \bv \otimes W'$ in $V \cong
E(r) \otimes_F V'$. This does not depend on the choice of $r$
because whenever $E(r)$ is regular, its exponent is zero
(Example~\ref{exa:exp}).

Suppose next that $V$ is a direct sum of twist-regular submodules $V_1
\oplus \cdots \oplus V_m$.
The Deligne-Malgrange lattice of $V$ is then defined as the
direct sum of the Deligne-Malgrange lattices of the $V_i$.

Suppose finally that $V$ is general. By Theorem~\ref{T:Turrittin},
there exists a finite extension $F'$ of $F$ such that $V' = V \otimes_F F'$
is a direct sum
of twist-regular submodules. Let $W'$ be the Deligne-Malgrange lattice of $V'$
as defined above. Identify $V$ with an $F$-subspace of $V'$,
then define the Deligne-Malgrange lattice of
$V$ to be $W' \cap V$. This is evidently a $\gotho_F$-lattice
in $V$; it does not depend on the choice of $F'$ because of
the admissibility condition on $\tau$.
(As in \cite[\S 3.3]{malgrange-conn2}, this reduces to the case of $V$
regular.)
\end{defn}

\subsection{Calculation of absolute scales}

Using the Hukuhara-Levelt-Turrittin decomposition, we can give some recipes
for computing absolute scales of finite differential modules over $F$.
The approach using lattices is suggested by
\cite[\S 1]{malgrange}, but ultimately differs from Malgrange's approach
because we interpret irregularity as a spectral measurement rather than
as the index of a linear operator.

\begin{defn} \label{D:degree}
Let $V$ be a finite differential module over $F$.
For $W$ a $\gotho_F$-lattice in $V$,
let $\Delta(W)$ be the $\gotho_F$-lattice in $V$ spanned by
$W$ together with the images $\del(W)$ for all $\del \in \Delta_F$
for which $|\del|_F \leq 1$. (This gives a finitely generated module
over $\gotho_F$ because $\Delta_F$ is finite dimensional over $F$.)
For $s$ a nonnegative integer,
define $\Delta^s(W)$ recursively by setting $\Delta^0(W) = W$
and $\Delta^{s+1}(W) = \Delta(\Delta^s(W))$.
Note that $\Delta^s(W)$ is the 
$\gotho_F$-span of the images of $W$ under all elements of
$F\{\Delta_F\}^{(s)}$ whose operator norm on $F$ is at most 1.
\end{defn}

\begin{lemma} \label{L:shift image}
Let $V$ be a finite differential module over $F$.
For $W$ an $\gotho_F$-lattice in $V$,
$\Delta(\gothm_F W) = \gothm_F \Delta(W)$.
\end{lemma}
\begin{proof}
If $\del \in \Delta_F$ satisfies $|\del|_F \leq 1$, then
for any $\bv \in W$ and any $z \in \gothm_F$, 
\[
\del(z\bv) = \del(z) \bv + z \del(\bv) \in \gothm_F W + \gothm_F \del(W)
\subseteq \gothm_F \Delta(W).
\]
Hence $\Delta(\gothm_F W) \subseteq \gothm_F \Delta(W)$.
In the other direction, for $\bv \in W$ and $z \in \gothm_F$,
\[
z \del(\bv) = \del(z\bv) - \del(z) \bv \in \del(\gothm_F W) + \gothm_F W
\subseteq \Delta(\gothm_F W).
\]
Hence $\gothm_F \Delta(W) \subseteq \Delta(\gothm_F W)$.
\end{proof}

\begin{lemma} \label{L:exp image}
Take $r \in F$.
\begin{enumerate}
\item[(a)]
For $W$ an $\gotho_F$-lattice in $E(r)$, 
$\Delta(W) = (\gotho_F + r \gotho_F)W$.
\item[(b)]
The absolute scale of $E(r)$ is equal to $\max\{1, |r|\}$.
\item[(c)]
Let $z$ be a generator of $\gothm_F$.
Let $\del \in \Delta_{\gotho_F}$
be a derivation of rational type with respect to $z$.
Then the scale of $z \del$ on $E(r)$ is also equal to $\max\{1, |r|\}$.
\end{enumerate}
\end{lemma}
\begin{proof}
All three claims follow from Proposition~\ref{P:regular} if $|r| \leq 1$, 
so we may assume $|r| > 1$ hereafter.
For (a), we have $W = z^m \gotho_F \bv$ for some $m \in \ZZ$,
and the claim to be shown is that $\Delta(W) = z^m r \gotho_F W$.
On one hand,
if $\del' \in \Delta_F$ satisfies $|\del'|_F \leq 1$, then
for $x \in \gotho_F$, $|\del'(z^m x)| \leq |z^m|$ and $|\del'(r)| \leq |r|$, so 
\[
\del'(z^m x \bv) = \del'(z^m x) \bv + z^m x \del'(r) \bv 
\in z^m \gotho_F \bv + z^m r \gotho_F \bv = z^m r \gotho_F \bv.
\]
On the other hand, by Lemma~\ref{L:find rational type}(a),
we have $|(z\del)(r)| = |r|$ and so $(z\del)(z^m \bv)$ generates $z^m r\gotho_F W$.

We now have (a). We also know that 
the absolute scale of $E(r)$ is at most $|r|$,
and that the scale of $z\del$ on $E(r)$ is at least $|r|$.
But the former is greater than or equal to the latter, so (b) and (c) follow.
\end{proof}

\begin{prop} \label{P:scale from h}
Let $z$ be a generator of $\gothm_F$.
Let $\del \in \Delta_{\gotho_F}$
be a derivation of rational type with respect to $z$.
Let $V$ be a nonzero finite differential module over $F$. 
Then the absolute scale multiset of $V$ coincides with 
the scale multiset of $z \del$ on $V$.
\end{prop}
\begin{proof}
We may assume $V$ is irreducible; in this case,
by Proposition~\ref{P:read slopes}
and Proposition~\ref{P:read absolute slopes}, the scale multiset of $z \del$ 
and the absolute scale multiset each
consist of a single element (else $V$ would be decomposable).
It thus suffices to check that the absolute scale of $V$ coincides
with the scale of $z \del$ on $V$.

On one hand, from the definition of scale and absolute scale,
the absolute scale of $V$ is at least the scale of $z \del$.
On the other hand, the absolute scale of $V$ is at most 
the absolute scale of $V \otimes_F F'$ for any finite extension $F'$ of $F$.
By Theorem~\ref{T:Turrittin}, we may choose $F'$ so that $V \otimes_F F'$
splits as a direct sum $\oplus V_j$ in which for each $j$, there is some
$r_j \in F'$ with $E(-r_j) \otimes_{F'} V_j$ regular.
The absolute scale of $V \otimes_F F'$ is the maximum of the absolute scales
of the $V_j$, which in turn is the same as the maximum of the absolute scales
of the $E(r_j)$. By Lemma~\ref{L:exp image}, 
this equals the maximum of the scales of $z\del$ on the $E(r_j)$,
which is the scale of $z\del$ on $V \otimes_F F'$. That in turn
equals the scale of $z\del$ on $V$ by Lemma~\ref{L:base change}
(applied to $F$ equipped with $\del$ alone).

We now have that the absolute scale of $V$ is bounded both above and 
below by the 
scale of $z\del$ on $V$. The two thus coincide, as desired.
\end{proof}

\begin{defn}
Let $V$ be a finite-dimensional $F$-vector space.
Let $W_1, W_2$ be $\gotho_F$-lattices in $V$. Define
\[
\ell(W_1, W_2) = \length_{\gotho_F} \frac{W_1}{W_1 \cap W_2}
- \length_{\gotho_F} \frac{W_2}{W_1 \cap W_2}.
\]
In particular, if $W_2 \subseteq W_1$, then
$\ell(W_1, W_2) = \length_{\gotho_F} W_1/W_2$. If we choose
an identification of $F$ with $K((z))$, we may replace the lengths
of $\gotho_F$-modules by dimensions of $K$-vector spaces.
\end{defn}

\begin{prop} \label{P:limit scale}
Let $V$ be a finite differential module over $F$.
Then for any $\gotho_F$-lattices $W_1, W_2$ in $V$
and any generator $z$ of $\gothm_F$,
\[
(-\log |z|) \lim_{s \to \infty} \frac{1}{s} \ell(\Delta^s(W_1), W_2) = 
\irreg(V).
\]
In particular, the limit exists and equals an integer.
\end{prop}
\begin{proof}
Since $\ell(\Delta^s(W_1), W_2) - \ell(\Delta^s(W_1), W_1)
= \ell(W_1, W_2)$ is independent of $s$, the existence and value of the limit
are unaffected by replacing $W_2$ by $W_1$. We thus assume $W_1 = W_2 = W$
hereafter.

Define
\begin{align*}
i_-(W) &= \liminf_{s \to \infty} 
\frac{1}{s} \ell(\Delta^s(W), W) \\
i_+(W) &= \limsup_{s \to \infty} 
\frac{1}{s} \ell(\Delta^s(W), W).
\end{align*}
We first check that these quantities are independent of the choice of $W$.
Let $W'$ be another $\gotho_F$-lattice in $V$;
then there exist integers $m,n \geq 0$ such that
$z^m W \subseteq W' \subseteq z^{-n} W$. 
By Lemma~\ref{L:shift image}, we have
$z^m \Delta^s(W) \subseteq \Delta^s(W') \subseteq z^{-n} \Delta^s(W)$
for each nonnegative integer $s$.
It follows that
\begin{align*}
\ell(\Delta^s(W'),W')
&= \ell(\Delta^s(W'), z^m \Delta^s(W))
+ \ell(z^m \Delta^s(W), z^m W) 
- \ell(W',z^m W)
\\
&\leq \ell(z^{-n} \Delta^s(W), z^m \Delta^s(W))
+ \ell(\Delta^s(W), W) \\
&= \ell(\Delta^s(W), W) + (m+n)\dim_F V.
\end{align*}
Since the last term is independent of $s$, we conclude that
\[
i_-(W') \leq i_-(W), \qquad i_+(W') \leq i_+(W).
\]
The same argument applies in the opposite direction, proving that
$i_-(W') = i_-(W)$ and $i_+(W') = i_+(W)$.

We next note that for any finite extension $F'$ of $F$ with ramification
index $m$,
\[
i_-(W \otimes_{\gotho_F} \gotho_{F'}) = m i_-(W), \qquad
i_+(W \otimes_{\gotho_F} \gotho_{F'}) = m i_+(W).
\]
By Theorem~\ref{T:Turrittin}, we may thus reduce to the case where
$V$ is a direct sum of twist-regular submodules.
Since we are free to choose $W$, we may choose it to be a sum of
lattices, one in each summand. We may thus reduce to the case where
$V = E(r) \otimes_F R$ for some $r \in F$ and $R$ regular.
In this case, take $W$ to be the tensor product of $\gotho_F \bv$
with a regulating lattice in $R$. Then by 
Lemma~\ref{L:exp image}, we have
$\Delta^s(W) = W + r^s W$, from which it follows that
\[
i_-(W) = i_+(W) = \max\{0, \log |r| / (-\log |z|)\}.
\]
This yields the desired result.
\end{proof}

\subsection{Decompletion}
\label{subsec:decompletion}

We will need to transfer some of the preceding results to certain 
incomplete discretely valued fields.

\begin{hypothesis} \label{H:Robba}
Throughout \S \ref{subsec:decompletion}, we specialize
Hypothesis~\ref{H:diff alg} as follows. 
Put $K = K_0((t))$ and $F = K((z))$, and view $F$ as a based 
differential field of order 1 with derivation 
$\del = \frac{\del}{\del z}$.
Let $v_t$ denote the $t$-adic valuation on $K$.
\end{hypothesis}

\begin{defn} \label{D:robba}
Under Hypothesis~\ref{H:Robba},
view $K_0((z))$ as a complete subfield of $F = K_0((t))((z))$.
Interpret $F$ as the ring of formal Laurent series 
$\sum_{i \in \ZZ} c_i t^i$ with bounded coefficients $c_i \in K_0((z))$
tending to 0 as $i \to -\infty$.

Let $\calR$ denote the ring 
of double series $\sum_{i \in \ZZ} c_i t^i$ with $c_i \in K_0((z))$
which converge for $t \in \overline{K_0((z))}$
in an annulus of the form $1-\epsilon \leq |t| < 1$
for some $\epsilon > 0$ (which may depend on the series).
The ring $\calR$ is called the \emph{Robba ring} over $K_0((z))$;
it is not comparable with $F$.

Let $\calR^{\bd}$ (resp.\ $\calR^{\inte}$)
denote the subring of $\calR$ consisting of
series $\sum_i c_i t^i$ with $\sup_i \{|c_i|\} < +\infty$
(resp.\ $\sup_i \{|c_i|\} \leq 1$).
The rings $\calR^{\bd}$ and $\calR^{\inte}$ are also called the
\emph{bounded Robba ring} and \emph{integral Robba ring} over $K_0((z))$; 
they may be viewed as subrings not only of
$\calR$ but also of $F$.
Note that the ring $\calR^{\inte}$ is an incomplete discrete valuation ring
with maximal ideal $z \calR^{\inte}$, and its fraction field is
$\calR^{\bd}$.
\end{defn}

\begin{lemma} \label{L:reduce kernel1}
Equip $\calR^{\inte}$ with the derivation $z \del$.
Let $M$ be a finite free differential module over $\calR^{\inte}$.
Then the unique $K$-lattice of $M \otimes_{\calR^{\inte}} \gotho_F$
stable under $z \del$ is also a $K$-lattice of $M$.
\end{lemma}
\begin{proof}
Let $\be_1, \dots, \be_d$ be a basis of $M$. Let $N$ be the matrix of action
of $z\del$ on $\be_1,\dots,\be_d$. As in Lemma~\ref{L:fundamental solution},
there exists a unique $d \times d$ matrix $U = \sum_{i=0}^\infty U_i z^i$
over $K \llbracket z \rrbracket$ with $U_0$ equal to the identity matrix,
such that
\[
NU + z \del(U) = U N_0.
\]
To prove the desired result, it suffices to show that $U$ has entries in
$\calR^{\inte}$ (as then will its inverse).

For $y$ an indeterminate, the expression $\det(T + y) \in K[y]$
is a polynomial of degree $d^2$. Viewed in $K_0[y]((t))$,
it is a series each of whose terms is bounded in degree by $d^2$. 
If the coefficient of some power of $t$ in this series fails to 
vanish at some $y \in \ZZ \setminus \{0\}$, then it fails to vanish at
all but finitely many $y \in \ZZ \setminus \{0\}$.
Since $\det(T+y)$ does not vanish for any $y \in \ZZ \setminus \{0\}$,
it follows that the
$t$-adic valuations of $\det(T + y)$ for $y \in \ZZ \setminus \{0\}$ must be
bounded above. 

This implies that there exists a constant $c$ such that for
any $d \times d$ matrix $X$ over $K$ and any nonzero integer $i$,
\begin{equation} \label{eq:solve de2}
v_t(X) \geq v_t(iX + N_0X - XN_0) - c.
\end{equation}
Since $N$ has entries in $\calR^{\inte}$, we can enlarge $c$ so as to have
$v_t(N_i) \geq -ic$ for all $i > 0$.
From \eqref{eq:solve de} and~\eqref{eq:solve de2}, it follows by induction on
$i$ that $v_t(U_i) \geq -2ic$ for $i \geq 0$: namely,
\begin{align*}
v_t(U_i) &\geq \min \{v_t(N_1) + v_t(U_{i-1}), \dots, v_t(N_i) 
+ v_t(U_0)\} - c\\
&\geq (-c-2(i-1)c) -c = -2ic.
\end{align*}
In particular,
$U$ has entries in $\calR^{\inte}$, as claimed. 
\end{proof}

\begin{lemma} \label{L:reduce kernel}
Equip $\calR^{\bd}$ with the derivation $\del$.
Let $V$ be a finite differential module over $\calR^{\bd}$
such that $V \otimes_{\calR^{\bd}} F$ is regular.
Then the kernels of $\del$ on $V, V \otimes_{\calR^{\bd}} F, V \otimes_{\calR^{\bd}} \calR$ all coincide.
\end{lemma}
\begin{proof}
Let $M$ be the intersection of $V$ with a regulating lattice in 
$V \otimes_{\calR^{\bd}} F$;
then $M$ is a finite free differential module over $\calR^{\inte}$
for the derivation $z \del$.
By Lemma~\ref{L:reduce kernel1}, $M$ contains a $(z \del)$-stable
$K$-lattice $W_0$.

Let $\bv$ be an element of the kernel of $\del$ on either 
$V \otimes_{\calR^{\bd}} F$
or $V \otimes_{\calR^{\bd}} \calR$.
By formally writing $\bv = \sum_i \bv^{(i)} z^i$ with each $\bv^{(i)} \in W_0$
as in Remark~\ref{R:formal product},
we see that $\bv^{(i)} = 0$ for $i \neq 0$.
Hence
$\bv \in W \subseteq V$, as desired.
\end{proof}

\subsection{More decompletion}
\label{subsec:more decompletion}

We continue the discussion of the previous section, but now add the effects
of additional derivations. This argument will be used at a key point to study
variation of irregularity (Theorem~\ref{T:convex}).

\begin{hypothesis}
Throughout \S \ref{subsec:more decompletion}, we specialize
Hypothesis~\ref{H:diff alg} as follows. 
Put $K = K_1((t))$ for some field $K_1$ of characteristic zero,
and put $F = K((z))$ equipped with the $z$-adic norm (of arbitrary
normalization).
Take $\Delta_F$ to be the module of derivations
generated by
$\del_0 = \frac{\del}{\del z}$,
$\del_1 = \frac{\del}{\del t}$,
and some commuting derivations $\del_2,\dots,\del_n$
on $K_1$ which act coefficientwise on $K$ and $F$.
\end{hypothesis}

\begin{defn}
For $r \geq 0$, define the function $|\cdot|_r: \calR^{\bd} \to [0, +\infty]$
by setting
\[
\left| \sum_i c_i t^i \right|_r = \sup_i \{|c_i| |z|^{ri}\}.
\]
Let $\calR^{\bd,r}$ be the subset
of $\calR^{\bd}$ on which $|\cdot|_r$ is finite; this is a subring
of $\calR^{\bd}$ on which $|\cdot|_r$ is a multiplicative norm.
Let $F_r$ be the completion of $\Frac(\calR^{\bd,r})$ under $|\cdot|_r$,
viewed as a based differential ring of order $n+1$
with derivations $\del_0,\dots,\del_n$; note that $F_0 = F$.
\end{defn}

\begin{lemma} \label{L:lower semicontinuous}
Let $S$ be a closed polyhedral subset of $\RR^n$ (i.e., the intersection
of finitely many closed halfspaces), and let $f: S \to \RR$
be a convex function. Then $f$ is upper semicontinuous; consequently,
$f$ is continuous if and only if it is lower semicontinuous.
\end{lemma}
\begin{proof}
See \cite[Theorem~10.2]{rockafellar}.
\end{proof}

\begin{lemma} \label{L:varies continuously}
For $f \in \calR^{\bd}$, the function $r \mapsto |f|_r$ is continuous
in some neighborhood of $r=0$.
\end{lemma}
\begin{proof}
If we formally write $f = \sum_{i,j} c_{ij} z^i t^j$ with $c_{ij} \in K_1$,
then $|f|_r$ is the supremum of $|z|^{i+rj}$ over all pairs $(i,j)$
for which $c_{ij} \neq 0$. In particular, $\log |f|_r$ is 
the supremum of a collection of affine functions of $r$.
Such a supremum is both lower semicontinuous and convex in $r$,
and so by Lemma~\ref{L:lower semicontinuous}
is continuous on a closed interval if and only if it is finite
there. But $|f|_r$ is finite both for $r=0$ and for some $r>0$,
so it is continuous in some neighborhood of $r=0$.
\end{proof}

\begin{prop} \label{P:continuous scales}
Let $M$ be a finite differential module over $\calR^{\bd,s}$ for some
$s>0$.
Then the absolute scale multiset of $M_r = M \otimes_{\calR^{\bd,s}} F_r$
varies continuously with $r$ at $r=0$.
\end{prop}
\begin{proof}
By Proposition~\ref{P:scale to absolute}, 
the absolute scale multiset of $M_r$
is equal to the based scale multiset for
the derivations $\del_0,\del_1$.
As in the proof of Proposition~\ref{P:read absolute slopes}, 
there is a direct sum decomposition of $M_0$ in which for 
$j \in \{0,1\}$, the scale multiset of $\del_j$
on each summand in the decomposition consists of a single element.
By Lemma~\ref{L:reduce kernel}, the projectors of this decomposition
also define a direct sum decomposition of $M \otimes_{\calR^{\bd,s}} \calR^{\bd}$.
Hence at the expense of lowering $s$, we may reduce to the case
where for $j \in \{0,1\}$, the scale multiset of $\del_j$ on $M_0$
consists of a single element. In particular, the absolute scale
multiset of $M_0$ consists of a single element.

We check that for $j \in \{0,1\}$,
as $r$ tends to 0,
each element of 
the scale multiset of $\del_j$ on $M_r$ tends to the scale
of $\del_j$ on $M_0$.
View
$F = \Frac(\calR^{\bd,s})$ as a based differential field of order 1 equipped
with just $\del_j$.
Apply Lemma~\ref{L:cyclic} to construct an isomorphism
$M \otimes_{\calR^{\bd,s}} F \cong F\{T\}/F\{T\}P$
for some monic twisted polynomial $P = T^d + \sum_{i=0}^{d-1} P_i T^i \in
F\{T\}$.
By Proposition~\ref{P:read slopes},
the scale multiset of $\del_j$ on $M_r$ is a continuous function
of the Newton polygon of $P$ under $|\cdot|_r$. However, 
by Lemma~\ref{L:varies continuously},
the vertices of this polygon vary continuously in $r$ at $r=0$.
This completes the check.

Let $S$ be the absolute scale of $M_0$.
On one hand, each element of the absolute scale multiset of $M_r$
is bounded above by the maximum over $j$ of the scale of $\del_j$
on $M_r$, which by the previous paragraph
tends to $S$ as $r$ tends to 0.
On the other hand, we can choose $j$ such that the scale of 
$\del_j$ on $M_0$ equals $S$; then each element 
of the absolute scale multiset of $M_r$ is bounded below by the least
element of the scale multiset of $\del_j$ on $M_r$.
This also tends to $S$ as $r$ tends to 0, because the scale multiset of
$\del_j$ on $M_0$ contains a single element.

We conclude that the entire absolute scale multiset of $M_r$
is bounded both above and below by functions which tend to $S$ as $r$
tends to 0. This proves the desired result.
\end{proof}

\section{Differential modules over power series rings}
\label{sec:power series}

In this section, we formulate some results about differential modules on
localized power series rings. These are derived from results in 
our joint paper \cite{kedlaya-xiao} with Liang Xiao.
See \S~\ref{subsec:power setup} for running notation and hypothesis in this 
section.

\subsection{Setup}
\label{subsec:power setup}

We set some running hypotheses for this section and the next.

\begin{notation} \label{N:numerical}
Throughout \S~\ref{sec:power series} and \S~\ref{sec:numerical},
set notation as follows.
Let $k$ be a field of characteristic zero.
For $n \geq m \geq 0$ integers, put
$R_{n,m} = k \llbracket x_1, \dots, x_n \rrbracket [ x_1^{-1}, \dots, 
x_m^{-1}]$. 
Unless otherwise specified,
view $R_{n,m}$ as a based differential ring equipped with the derivations
$\del_1, \dots, \del_n = \frac{\del}{\del x_1}, \dots, \frac{\del}{\del x_n}$.
Note that $R_{n,m}$ is a localization of the locally simple (by
Example~\ref{exa:locally simple}) noetherian 
ring $k \llbracket x_1,\dots,x_n \rrbracket$,
so it is also noetherian and locally simple. Hence any finite differential module
over $R_{n,m}$ is locally free (or equivalently projective, since
$R_{n,m}$ is noetherian).
\end{notation}

\begin{remark}
In some situations, we will need to equip $R_{n,m}$ instead with the
derivations $x_1\del_1,\dots,x_n\del_n$. With these derivations,
$R_{n,m}$ is not locally simple.
\end{remark}

\begin{remark} \label{R:proj free}
Since we will be working with finitely generated locally free modules over 
$R_{n,m}$, it would be useful to know that these are all free.
Unfortunately, there are only two cases where this is easy to prove.
\begin{itemize}
\item
If $m=0$, this holds by Nakayama's lemma.
\item
If $n=2$ and $m>0$, this holds because $R_{n,m}$ is 
a one-dimensional factorial noetherian domain, and hence a principal
ideal domain \cite[\S VII.3.1, Exemple~1(a)]{bourbaki-algcom}.
\end{itemize}
The case $m=1$ is known in general, but the proof is difficult;
see \cite{popescu} or its MathSciNet review for a summary. 
(Thanks to Joseph Gubeladze for the reference.)

One can at least say that all finitely generated locally free modules
over $R_{n,m}$ are stably free, i.e., the group $K_0(R_{n,m})$ vanishes;
this follows from the case $m=0$ (Nakayama's lemma again) by localization.
However, we will not need this result.
\end{remark}

\begin{defn}
For $r = (r_1,\dots,r_n) \in [0, +\infty)^n$,
write $e^{-r}$ for $(e^{-r_1},\dots,e^{-r_n})$.
Let $|\cdot|_r$ be the $(e^{-r})$-Gauss norm
on $R_{n,m}$, as in \eqref{eq:gauss} except we have changed the
subscript from $e^{-r}$ to $r$ for brevity.
Let $F_r$ be the completion of $\Frac R_{n,m}$ 
with respect to $|\cdot|_r$, viewed as a complete nonarchimedean
based differential ring of order $n$ using $\del_1,\dots,\del_n$.
\end{defn}

\begin{remark} \label{R:based absolute}
Note that $\del_1,\dots,\del_n$ are of rational type on $F_r$
with respect to $x_1,\dots,x_n$;
this means that the results of \cite{kedlaya-xiao}, which refer to based
scales, apply also to absolute scales thanks to 
Proposition~\ref{P:scale to absolute}. We will not point this out again.
\end{remark}

\begin{notation}
For $M$ a differential module over $R_{n,m}$
and $r \in [0, +\infty)^n$, we write $M_r$ as shorthand for 
$M \otimes_{R_{n,m}} F_r$. This may combine with other subscripts,
e.g., $M_{0,r}$ is shorthand for $M_0 \otimes_{R_{n,m}} F_r$.
\end{notation}

\subsection{Spectral variation}

We now set notation to record the irregularity of $M$ as a function
of a varying Gauss norm. Our main result is in the spirit of 
\cite[Theorem~3.3.9]{kedlaya-xiao} but is not an immediate corollary of that result,
so a bit of work is needed. 

\begin{defn} \label{D:spectral}
Let $M$ be a differential module over $R_{n,m}$ of finite rank $d$.
For $r \in [0, +\infty)^n$, 
define $f_1(M,r) \geq \cdots \geq f_d(M,r) \geq 0$ as the numbers such 
that the absolute scale multiset of $M_r$
consists of $e^{f_1(M,r)}, \dots, e^{f_d(M,r)}$. Define
$F_i(M,r) = f_1(M,r) + \cdots + f_i(M,r)$, so that $F_i(M,r)$
computes the $i$-th partial irregularity of $M_r$ (see
Definition~\ref{D:irregularity});
also put $F_0(M,r) = 0$.
Note that for any $\lambda \geq 0$, $f_i(M, \lambda r) = \lambda f_i(M,r)$.
\end{defn}

For a differential module over $k[x_1,\dots,x_n,x_1^{-1},\dots,x_m^{-1}]$,
the following result would be an immediate corollary of 
\cite[Theorem~3.3.9]{kedlaya-xiao} (using Remark~\ref{R:based absolute});
we may use that result plus a bit of extra argument to get what we need here.
\begin{theorem} \label{T:convex}
Let $M$ be a differential module over $R_{n,m}$ of finite rank $d$.
For $i=1,\dots,d$, the function $d! F_i(M,r)$ can be written
as $\max_{j=1}^h \{\lambda_j(r)\}$ for some
linear functionals $\lambda_1, \dots, \lambda_h$ which are integral
(i.e., which carry $\ZZ^n$ to $\ZZ$).
In particular, $F_i(M,r)$ is continuous, convex, and piecewise linear.
Moreover, for $j \in \{m+1,\dots,n\}$, if we fix $r_h$ for $h \neq j$,
then $F_i(M,r)$ is nonincreasing as a function of $r_j$ alone.
\end{theorem}
\begin{proof}
We proceed by induction on $n$, the case $n=1$ being 
immediate from the homogeneity property $f_i(M, \lambda r) = \lambda f_i(M,r)$.
It will be convenient to make the induction hypothesis slightly stronger by allowing
$k$ to carry finitely many derivations as well.
For each $j \in \{1,\dots,n\}$,
we may then use the induction hypothesis to deduce the desired properties
for the restriction of $F_i(M,r)$ to the hyperplane $r_j = 0$.
(Namely, we replace $k$ by $k((x_j))$, with the latter
carrying all of the derivations on $k$
plus $\frac{\del}{\del x_j}$.)

Using \cite[Theorem~3.3.9]{kedlaya-xiao},
we see that for any $\epsilon > 0$, the restriction of $d! F_i(M,r)$ to
$[\epsilon, +\infty)^d$ has the desired form. In particular, 
the restriction of $d! F_i(M,r)$ to
$(0, +\infty)^d$ is convex. If we then let
$T$ be the set of affine functionals $\lambda: \RR^n \to \RR$ for 
which $d! F_i(M,r) \geq \lambda(r)$ for all $r \in (0, +\infty)^d$,
the function
\begin{equation} \label{eq:spectral variation}
g(r) = \sup\{\lambda(r): \lambda \in T\}
\end{equation}
satisfies $g(r) = d! F_i(M,r)$ for $r \in (0, +\infty)^d$.

Note that \eqref{eq:spectral variation} defines a lower semicontinuous convex
function $g: [0, +\infty)^d \to \RR \cup \{+\infty\}$
which is finite on the interior of its domain.
By Lemma~\ref{L:lower semicontinuous},
$g$ is continuous at each point where its value is finite.
Consequently, for any given $r_0 \in [0, +\infty)^d$, if $g(r)$ tends to
a finite limit $L$ as $r$ approaches $r_0$ along some line, then 
$g(r_0)$ is finite (by lower semicontinuity), and so 
$g$ is continuous at $r_0$ and $g(r_0) = L$.
Using these observations plus Proposition~\ref{P:continuous scales},
we deduce that for any nonzero $r_0 \in [0, +\infty)^d \cap \QQ^d$,
$g$ is continuous at $r_0$ and $g(r_0) = d! F_i(M, r_0)$.
Since $g$ is convex (and $g(0) = d! F_i(M,0) = 0$ by homogeneity), 
this implies that $g$ is everywhere finite and hence
everywhere continuous.

For $j \in \{1,\dots,n\}$, consider the restriction of 
$d! F_i(M,r)$ to the hyperplane $r_j = 0$. By the induction hypothesis,
this restriction is continuous and convex. By the previous paragraph,
it agrees with $g(r)$ at every rational point of this hyperplane.
Hence $g(r) = d! F_i(M,r)$ everywhere on the hyperplane;
since we also have this equality on $(0, +\infty)^d$, we conclude that
$g(r) = d! F_i(M,r)$ for all $r \in [0, +\infty)^d$.

In particular, we now have that $F_i(M,r)$ is continuous and convex.
Moreover, for any $r \in [0, +\infty)^d \cap \QQ^d$, we have
\[
d! F_i(M,r) \in \ZZ r_1 + \cdots + \ZZ r_d.
\]
These conditions together with \cite[Theorem~2.4.2]{kedlaya-part3}
imply that on $[0,1]^d$, we have
$d! F_i(M,r)=  \max_{j=1}^h \{\lambda_j(r)\}$ for some integral
linear functionals $\lambda_1, \dots, \lambda_h$. By homogeneity,
we have the same conclusion for all $r \in [0, +\infty)^d$.
By continuity, the monotonicity assertion follows from the corresponding
assertion on $(0, +\infty)^d$, which again holds by
\cite[Theorem~3.3.9]{kedlaya-xiao}.
\end{proof}

\begin{remark} \label{R:m equals 0}
The case $m=0$ of Theorem~\ref{T:convex} has little content, as in
this case there exists a canonical horizontal isomorphism
$M \cong (M/(x_1,\dots,x_n)M) \otimes_k R_{n,0}$. In other words,
$M$ admits a basis of horizontal sections. This can be checked by induction
from the case $n=1$, which follows as in the proof of
Lemma~\ref{L:fundamental solution}. In any case,
we have $F_i(M,r) = 0$ for all $i$ and $r$.

More generally, for any $m$, we have
$F_i(M,r) = 0$ whenever $r_1 = \cdots = r_m = 0$.
Consequently, if $F_i(M,r)$ is linear in $r$,
then it must be constant in $r_{m+1},\dots,r_n$;
we will encounter this condition frequently, particularly in
the statement of the numerical criterion for good decompositions
(Theorem~\ref{T:criterion}).
\end{remark}

We need some corollaries of the convexity of $F_i(M,r)$.
\begin{cor} \label{C:drop affine}
With notation as in Theorem~\ref{T:convex}, 
suppose that $F_d(M,r)$ is linear in $r$,
and that $M$ decomposes as a direct sum
$\oplus_j M_j$ of differential submodules. Then for each $j$,
for $d_j = \rank(M_j)$, $F_{d_j}(M_j,r)$ is linear in $r$.
\end{cor}
\begin{proof}
By Theorem~\ref{T:convex}, $F_{d_j}(M_j,r)$ is convex in $r$.
On the other hand, on the right side of the equality
\[
F_{d_j}(M_j,r) = F_d(M,r) + \sum_{i \neq j} (-F_{d_i}(M_i,r)),
\]
each summand (including $F_d(M,r)$ by hypothesis) is concave in $r$.
Hence
$F_{d_j}(M_j,r)$ is affine (hence linear) in $r$, as desired.
\end{proof}

\begin{cor} \label{C:convex} 
With notation as in Theorem~\ref{T:convex}, 
suppose that $F_d(M,r)$ is linear in $r$.
Then there exists an index $j \in \{0,\dots,d\}$ with the following
properties.
\begin{enumerate}
\item[(a)]
For $i = j+1,\dots,d$, $f_i(M,r) = 0$ identically.
\item[(b)]
For $i = 1,\dots, j$, $f_i(M,r) > 0$ for all
$r$ with $r_1,\dots,r_n > 0$.
\item[(c)]
$F_j(M,r)$ is linear in $r$. 
\end{enumerate}
\end{cor}
\begin{proof}
By hypothesis, (a) and (c) hold for $j = d$.
Suppose now that $j \in \{0,\dots,d\}$ is any index such that 
(a) and (c) hold. If (b) also holds, we are done. Otherwise,
we must have $j > 0$, and there must exist some
$r = (r_1,\dots,r_n)$ with $r_1,\dots,r_n > 0$ 
such that $f_i(M,r) = 0$ for some $i \in \{1,\dots,j\}$; this
forces $f_j(M,r) = 0$. The function $F_{j-1}(M,r)$ is convex
(by Theorem~\ref{T:convex}), is bounded above by the linear function
$F_j(M,r)$ (because $f_j(M,r) \geq 0$), and is equal to $F_j(M,r)$
at an interior point of its domain. It follows that $F_{j-1}(M,r) = F_j(M,r)$
identically, so (a) and (c) also hold with $j$ replaced by $j-1$.
We conclude that all three clauses hold for some $j$, as desired.
\end{proof}

\subsection{Decomposition by spectral radius}

We next give a theorem that decomposes a differential module according
to its spectral invariants. Again, there is no result in 
\cite{kedlaya-xiao}
that accomplishes exactly what we need, so we cobble together
some results from \cite{kedlaya-xiao} with the decompletion
arguments of \S~\ref{subsec:decompletion}.

\begin{defn}
Let $S_{n,m}$ be the Fr\'echet completion of $k[x_1,\dots,x_n,x_1^{-1},\dots,x_m^{-1}]$
for the norms $|\cdot|_r$ for all $r \in (0, +\infty)^n$. In concrete terms,
$S_{n,m}$ consists of formal sums 
\[
\sum_{i_1,\dots,i_m=-\infty}^\infty \sum_{i_{m+1},\dots,i_n=0}^\infty c_{i_1,\dots,i_n} x_1^{i_1}\cdots x_n^{i_n}
\qquad (c_{i_1,\dots,i_n} \in k),
\]
such that for any $r_1,\dots,r_n > 0$ and any $C \in \RR$,
there are only finitely many indices $(i_1,\dots,i_n) \in \ZZ^n$
with $r_1 i_1 + \cdots + r_n i_n < C$ and $c_{i_1,\dots,i_n} \neq 0$.
\end{defn}

In this notation, \cite[Theorem~3.4.2]{kedlaya-xiao} implies the following.
\begin{theorem} \label{T:refined decomp weak}
Let $M$ be a differential module over $R_{n,m}$ of finite rank $d$.
Fix $l \in \{1,\dots,d-1\}$, and suppose that the following conditions hold.
\begin{enumerate}
\item[(a)]
The function $F_l(M, r)$ is linear.
\item[(b)]
We have $f_l(M,r) > f_{l+1}(M,r)$ for all $r \in (0, +\infty)^n$.
\end{enumerate}
Then there exists
a unique direct sum decomposition $M \otimes_{R_{n,m}} S_{n,m} = 
M_1 \oplus M_2$ such that for each $r \in (0, +\infty)^n$, the absolute 
scale multiset of $M_{1,r}$
consists of the largest $l$ elements of the absolute 
scale multiset of $M_r$.
\end{theorem}

\begin{remark} \label{R:refined decomp}
Note that in Theorem~\ref{T:refined decomp weak},
it suffices to check the inequality in (b) for a \emph{single} value of $r$.
That follows as in the proof of Corollary~\ref{C:convex}:
given (a), the function
\[
f_l(M,r) - f_{l+1}(M,r) = 2F_l(M,r) - F_{l+1}(M,r) - F_{l-1}(M,r)
\]
is concave (being a linear function minus a sum of convex functions),
but is bounded below by 0. Hence it can only take the value 0 at an interior
point of its domain if it is identically zero.
\end{remark}

We wish to descend the decomposition in Theorem~\ref{T:refined decomp weak}
from $S_{n,m}$ to $R_{n,m}$.
To do this, we use the following observations.
\begin{lemma} \label{L:Hadamard}
We have the inequality
\[
t \log |x|_r + (1-t) \log |x|_s \geq
\log |x|_{tr+(1-t)s} \qquad
(x \in S_{n,m}; t \in [0,1]; r,s \in [0, +\infty)^n).
\]
\end{lemma}
\begin{proof}
This reduces to the case for monomials, where the inequality
becomes an equality.
\end{proof}

\begin{lemma} \label{L:bounded}
Let $e_1,\dots,e_n$ denote the standard basis vectors of $\RR^n$.
For $x \in S_{n,m}$, the following are equivalent.
\begin{enumerate}
\item[(a)]
We have $x \in R_{n,m}$.
\item[(b)]
The function $r \mapsto |x|_r$ is bounded above on the simplex
\[
T = \{(x_1,\dots,x_n) \in (0, +\infty)^n:
x_1 + \cdots + x_n \leq 1 \}.
\]
\item[(c)]
The function $r \mapsto |x|_r$ is bounded above on any bounded subset of $(0, +\infty)^n$.
\item[(d)]
For $i=1,\dots,m$, $|x|_r$ remains bounded above
as $r$ approaches $e_i$ along some line
segment contained in $(0, +\infty)^n$.
\end{enumerate}
\end{lemma}
\begin{proof}
Given (a), by Lemma~\ref{L:Hadamard},
for $r \in T$, $|x|_r$ is bounded above by $\{1, |x|_{e_1},\dots,|x|_{e_n}\}$.
Hence (b) follows. Given (b), (c) follows because $|x|_{\lambda r} = |x|_r^\lambda$.
Given (c), (d) is evident.

Given (d), by Lemma~\ref{L:Hadamard}, $\log |x|_r$ is a convex 
function on $r \in (0, +\infty)^n$.
As in the proof of Theorem~\ref{T:convex},
we extend $\log |x|_r$ to a lower semicontinuous convex
function $f: [0, +\infty)^n \to 
\RR \cup \{+\infty\}$ by setting
\[
f(r) = \sup\{\lambda(r)\}
\]
with the supremum running over all affine functionals $\lambda: \RR^n \to \RR$ such that
$\log |x|_r \geq \lambda(r)$ for all $r \in (0, +\infty)^n$.
By (d), $f$ takes finite values at $r = e_1,\dots,e_m$;
it also takes values at $r = e_{m+1},\dots,e_n$ which are bounded above by $0$.
Since $f(\lambda r) = \lambda f(r)$ for all $\lambda \geq 0$,
we deduce that $f$ takes finite values on the one-dimensional boundary
facets of $[0, +\infty)^n$.
By convexity, $f$ takes values which are bounded above on any bounded subset of
$(0, +\infty)^n$.
Hence (b) and (c) follow.

Given (b) (and hence (d)), define $f$ as in the previous paragraph, and
put $j_i = \lfloor f(e_i) \rfloor$
for $i=1,\dots,m$. Since $f$ is convex, 
$y = x_1^{-j_1} \cdots x_m^{-j_m} x$ is an element of $S_{n,m}$
such that $|y|_r \leq 1$ for all $r \in (0, +\infty)^n$;
this forces $y \in k \llbracket x_1,\dots,x_n \rrbracket$ and
$x \in R_{n,m}$. Hence (a) follows.
\end{proof}

We can now refine Theorem~\ref{T:refined decomp weak} as follows.
(Again, hypothesis (b) can be relaxed as in Remark~\ref{R:refined decomp}.)
\begin{theorem} \label{T:refined decomp}
Let $M$ be a differential module over $R_{n,m}$ of finite rank $d$.
Fix $l \in \{1,\dots,d-1\}$, and suppose that the following conditions hold.
\begin{enumerate}
\item[(a)]
The function $F_l(M, r)$ is linear.
\item[(b)]
We have $f_l(M,r) > f_{l+1}(M,r)$ for all $r \in (0, +\infty)^n$.
\end{enumerate}
Then $M$ admits a unique direct sum decomposition $M_1 \oplus M_2$
such that for each $r \in (0, +\infty)^n$, the absolute 
scale multiset of $M_{1,r}$
consists of the largest $l$ elements of the absolute 
scale multiset of $M_r$.
\end{theorem}
\begin{proof}
The case $m=0$ holds by Remark~\ref{R:m equals 0},
so we assume hereafter that $m>0$. 
Let $\bv \in \End(M) \otimes_{R_{n,m}} S_{n,m}$ be the projector
onto the first summand in the decomposition of $M \otimes_{R_{n,m}} S_{n,m}$
provided by Theorem~\ref{T:refined decomp weak}.
Pick a basis $\be_1,\dots,\be_{d^2}$ of $\End(M) \otimes_{R_{n,m}} \Frac(R_{n,m})$ 
consisting of elements of $\End(M)$, and use it to define
supremum norms $|\cdot|_r$ on $\End(M)_r$ for each $r \in [0, +\infty)^n$.
Choose $f \in R_{n,m}$ so that 
$f \End(M) \subseteq R_{n,m} \be_1 + \cdots + R_{n,m} \be_{d^2}$.

Let $e_1,\dots,e_n$ denote the standard basis vectors of $\RR^n$.
We claim that $|\bv|_r$ remains bounded as $r$ approaches $e_1$
along the segment consisting of $(1-(n-1)t, t,\dots,t)$ for
$t \in (0, 1/(n-1))$, i.e., as $t \to 0^+$. To check this,
we set notation as in Definition~\ref{D:robba} with
\[
K_0 = k(x_3/x_2,\dots,x_n/x_2), \quad t = x_2, \quad z = x_1.
\]
Then $R_{n,m}$ embeds into $\calR^{\bd}$ while $S_{n,m}$ embeds into $\calR$.
By Lemma~\ref{L:reduce kernel}, when viewed as an element of 
$\End(M) \otimes_{R_{n,m}} \calR$, $\bv$ in fact belongs to 
$\End(M) \otimes_{R_{n,m}} \calR^{\bd}$.
This implies that $|\bv|_r$ remains bounded along the path in question.

Similarly, for $i=1,\dots,m$, $|\bv|_r$ remains bounded as $r$ approaches
$e_i$ along some line segment contained in $(0, +\infty)^n$. By
Lemma~\ref{L:bounded}, we deduce that $f\bv \in 
R_{n,m} \be_1 + \cdots + R_{n,m} \be_{d^2} \subseteq \End(M)$.

Since $\End(M)$ is locally free over $R_{n,m}$ (see Notation~\ref{N:numerical}), 
we may vary the choice of
the initial basis $\be_1,\dots,\be_{d^2}$ in such a way that $f$ varies
over a set of generators of the unit ideal in $R_{n,m}$. It follows that
$\bv \in \End(M)$, and so the decomposition of $M \otimes_{R_{n,m}} S_{n,m}$
in fact arises from a decomposition of $M$, as desired.
\end{proof}

\begin{remark} \label{R:refined decomp2}
In \cite[Theorem~3.4.4]{kedlaya-xiao}, one finds a result
similar to Theorem~\ref{T:refined decomp}, but with an important difference:
the inequality in (b) is imposed for all $r \in [0, +\infty)^n$, not just
for those $r$ with $r_1,\dots,r_n > 0$.
This weaker statement is made because it is the best possible statement
that also applies over a complete nonarchimedean base field of mixed
characteristics.
(Thanks to Liang Xiao for this observation.)
\end{remark}

\section{A numerical criterion}
\label{sec:numerical}

In this section, we establish a numerical criterion for existence of
a good decomposition of a finite differential module over the
localized power series ring $R_{n,m}$.
Throughout \S~\ref{sec:numerical}, we continue to
retain Notation~\ref{N:numerical}.

\subsection{Regular connections}

In this section, we prove the following theorem relating several natural
notions of regularity for finite differential modules over $R_{n,m}$.

\begin{notation}
Let $e_1,\dots,e_n$ denote the standard basis vectors of $\RR^n$.
For $i=1,\dots,m$, we write $F_{(i)}$ and $|\cdot|_{(i)}$ as shorthand for 
$F_{e_i}$ and $|\cdot|_{e_i}$. We also write
$\gotho_{(i)}$ as shorthand for $\gotho_{F_{(i)}} = \{x \in F_{(i)}:
|x|_{(i)} \leq 1\}$,
and let $\kappa_{(i)}$ denote the residue field of $\gotho_{(i)}$.
\end{notation}

\begin{lemma} \label{L:make lattice}
Let $M$ be a finite locally free $R_{n,m}$-module. For $i=1,\dots,m$,
put $V_i  = M \otimes_{R_{n,m}} F_{(i)}$, and let
$W_i$ be an $\gotho_{(i)}$-lattice in $V_i$.
Let $M_0$ be the set of $\bv \in M$ such that for $i=1,\dots,m$,
the image of $\bv$ in $V_i$ belongs to $W_i$. Then $M_0$ is a finite
$R_{n,0}$-module. (Remember that $R_{n,0} = k\llbracket x_1,\dots,x_n
\rrbracket$.)
\end{lemma}
\begin{proof}
Put $F = \Frac(R_{n,m})$
and $M_F = M \otimes_{R_{n,m}} F$.
Let $B = \{\bv_1, \dots, \bv_d\}$ be a basis of
$M_F$ over $F$.

For $i=1,\dots,m$,
$F$ is dense in $F_{(i)}$.
Hence by choosing a basis of $W_i$ and then approximating its elements
with elements of $M_F$,
we obtain a basis $B_i = \{\bv_{1,i},\dots,\bv_{d,i}\}$ 
of $W_i$ which is also a basis of
$M_F$.
Choose $g, h \in R_{n,m}$ so that:
\begin{itemize}
\item
each element of $B$, multiplied by $g$,
is in the $R_{n,m}$-span of $B_i$ for each $i$;
\item
each element of $M$, multiplied by $h$, is in the $R_{n,m}$-span of $B$.
\end{itemize}
Define the matrix $A_i$ over $R_{n,m}$ by the formula
$g\bv_l = \sum_{j=1}^d (A_i)_{jl} \bv_{j,i}$.
Pick any $\bv \in M_0$, write it as 
$\sum_{l=1}^d h^{-1} r_l \bv_l$ with $r_l \in R_{n,m}$,
and put $r_{j,i} = \sum_{l=1}^d (A_i)_{jl} r_l \in R_{n,m}$, so that
\[
\bv = \sum_{j=1}^d (gh)^{-1} r_{j,i} \bv_{j,i}.
\]
Since $\bv \in M_0$, we have $|(gh)^{-1} r_{j,i}|_{(i)} \leq 1$ for $j=1,\dots,d$.
Since $r_l = \sum_{l=1}^d (A_i^{-1})_{lj} r_{j,i}$ in $F$, we have
\[
|r_l|_{(i)} \leq \max_{j,l} \{|(A_i^{-1})_{lj}|_{(i)}\} |gh|_{(i)}.
\]
We can thus find $l_1,\dots,l_m \in \ZZ$ for which
$x_1^{l_1} \cdots x_m^{l_m} M_0$ is contained in the
$R_{n,0}$-span of $h^{-1} B$ within $M_F$.
Hence $M_0$ is contained in a finite
$R_{n,0}$-module; since $R_{n,0}$ is noetherian, $M_0$ is itself a finite
$R_{n,0}$-module.
\end{proof}

\begin{defn}
Let $M$ be a finite differential module over $R_{n,m}$. We say that
$M$ is \emph{regular}
if the equivalent conditions of Theorem~\ref{T:regular} (see below)
are satisfied.
\end{defn}

\begin{theorem} \label{T:regular}
Let $M$ be a differential module over $R_{n,m}$ of finite rank $d$.
Then the following are equivalent.
\begin{enumerate}
\item[(a)]
There exists a
free 
differential module $M_0$ over $R_{n,0} = k \llbracket x_1,\dots,x_n 
\rrbracket$ 
equipped with the derivations $x_1 \del_1, \dots, x_m \del_m, 
\del_{m+1}, \dots, \del_{n}$,
and an isomorphism $M \cong M_0 \otimes_{R_{n,0}} R_{n,m}$ of 
differential modules.
\item[(a$'$)]
As in (a), except that for $i=1,\dots,n$,
the (linear) action of $x_i \del_i$ 
on the $k$-vector space $V = M_0/(x_1,\dots,x_n) M_0$
has prepared eigenvalues.
\item[(b)]
$M$ is free and admits a basis on which $x_1 \del_1, \dots, x_m \del_m$
act via matrices over $k$ and $\del_{m+1},\dots,\del_n$
act via the zero matrix.
\item[(b$'$)]
As in (b), except that also the matrices have prepared eigenvalues.
\item[(c)]
Either $M=0$, or $f_1(M,r) = 0$ for all $r$. (By Theorem~\ref{T:convex},
it is equivalent to check just for $r=e_1,\dots,e_m$.)
\end{enumerate}
\end{theorem} 

\begin{proof}[Proof of Theorem~\ref{T:regular}]
We first note some easy implications: (a$'$)$\implies$(a),
(b)$\implies$(a),
(b$'$)$\implies$(b),
(b$'$)$\implies$(a$'$) are trivial,
while (a)$\implies$(c) is evident from the definition of $f_1(M,r)$
using the lattices generated by $M_0$.
To complete the circle, it suffices to prove (c)$\implies$(b$'$).

Assume (c).
For $i=1,\dots,m$, apply Proposition~\ref{P:regular} to
construct a regulating lattice $W_i$ in $M \otimes_{R_{n,m}} F_{(i)}$.
Let $M_0$ be the $R_{n,0}$-submodule of $M$ consisting of elements whose image in 
$M \otimes_{R_{n,m}} F_{(i)}$ lies in $W_i$ for $i=1,\dots,m$;
by Lemma~\ref{L:make lattice}, $M_0$ is a finite $R_{n,0}$-module.
For uniformity, for $i > m$, write $W_i$ for $M \otimes_{R_{n,m}} \gotho_{(i)}$.

For $i=1,\dots,n$,
identify $\kappa_{(i)}$ with the kernel of $\del_i$ on $F_{(i)}$,
so that $F_{(i)}$ is identified with $\kappa_{(i)}((x_i))$.
Apply Lemma~\ref{L:fundamental solution} to construct
a $(x_i \del_i)$-stable $\kappa_{(i)}$-lattice $W_{i,0}$ in $W_i$.
Let $P_i(T)$ be the characteristic polynomial of $x_i \del_i$
on $W_{i,0}$. By Proposition~\ref{P:exponents in base}, $P_i(T)$ belongs to 
$k[T]$; since $W_i$ is a regulating lattice, $P_i(T)$ has prepared roots.
(These roots are all zero in case $i > m$.)
Hence for $j=1,2,\dots$,
we can find a polynomial $Q_{i,j}(T) \in k[T]$ such that
\[
Q_{i,j}(T) P_i(T-1) \cdots P_i(T-j) \equiv 1 \pmod{P_i(T)}.
\]
Put
\[
R_{i,j}(T) = Q_{i,j}(T) P_i(T-1)\cdots P_i(T-j).
\]
Then $R_{i,j}(x_i \del_i)$ acts as the identity on $W_{i,0}$
but kills $x_i W_{i,0}, \dots, x_i^j W_{i,0}$. Consequently,
for $\bv \in W_i$,
\begin{equation} \label{eq:congruence}
R_{i,j+1} \left(x_i \del_i \right)(\bv) -
R_{i,j} \left(x_i \del_i \right)(\bv) \in x_i^j W_i.
\end{equation}

By virtue of \eqref{eq:congruence}, for any $\bv \in M_0$, the sequence
\[
\bv^{(j)} = \left( \prod_{i=1}^n R_{i,j} \left(x_i \del_i
\right) \right) (\bv) 
\qquad (j=1,2,\dots)
\]
has the property that
\[
\bv^{(j+1)} - \bv^{(j)} \in (x_1^j, \dots, x_n^j) M_0.
\]
Hence it converges in the $(x_1,\dots,x_n)$-adic topology on $M_0$ to a limit
$f(\bv)$. The resulting function $f: M_0 \to M_0$ factors
through $M_0/(x_1,\dots,x_n) M_0$; the resulting map
$f: M_0/(x_1,\dots,x_n) M_0 \to M_0$ is $k$-linear and horizontal,
and is a section of the projection $M_0 \to M_0/(x_1,\dots,x_n) M_0$.

For any $\bv \in M_0$, 
$P_i(x_i \del_i)(\bv)$ is divisible by $x_i$, so
$P_i (x_i \del_i)(f(\bv))
= f(P_i(x_i \del_i)(\bv)) = 0$.
Hence $P_i(x_i \del_i)$ kills $f(\bv)$. 
On the other hand, $P_i(x_i \del_i)$ acts invertibly on
$x_i^j W_{i,0}$ for $j>0$ since $x_i \del_i$ acts on $W_{i,0}$
with prepared eigenvalues. Thus if we formally decompose
$f(\bv)$ into components $\sum_{j=0}^\infty \bv_j$ with $\bv_j \in x_i^j
W_{i,0}$
(as in Remark~\ref{R:formal product}),
then $\bv_j = 0$ for $j>0$. In other words, $f(\bv) \in W_{i,0}$.

Since $M_0$ is a finite $R_{n,0}$-module,
$M_0/(x_1,\dots,x_n)M_0$ is finite-dimensional over $k$.
Choose a $k$-basis of $M_0/(x_1,\dots,x_n)M_0$ and let
$\bv_1,\dots,\bv_e \in M_0$ be the images under $f$  of the elements
of this basis. Then Nakayama's lemma implies that $M_0$ is generated by 
$\bv_1,\dots,\bv_e$. On the other hand, suppose that
$r_1 \bv_1 + \cdots + r_e \bv_e = 0$ with $r_1,\dots,r_e \in R_{n,0}$.
We wish to check that $r_1 = \cdots = r_e = 0$, so suppose the contrary.
Since $M_0$ is torsion-free
(by virtue of sitting inside the locally free $R_{n,m}$-module $M$),
we may divide out any common factors of $x_1$ among the $r_j$.
Then reducing into $W_1/x_1 W_1$, we obtain a nontrivial relation
$r_1 \bv_1 + \cdots + r_e \bv_e = 0$ 
with $r_i \in R_{n,0}/x_1 R_{n,0}$.
Since each $\bv_j$ belongs to $W_{1,0}$,
this relation must lift to a nontrivial relation
$r_1 \bv_1 + \cdots + r_e \bv_e = 0$ 
in $W_1$ in which each $r_j$ is an element of $R_{n,0}$ in which 
$x_1$ does not appear. This relation must hold also in $M_0$
since $M_0$ injects into $W_1$. Repeating the argument, we may eliminate
$x_2,\dots,x_n$ from the relation while preserving its nontriviality.
We end up with a nontrivial relation $r_1 \bv_1 + \cdots + r_e \bv_e = 0$
with $r_1,\dots,r_e \in k$; projecting into
$M_0/(x_1,\dots,x_n)M_0$  yields a contradiction.

We conclude that $\bv_1,\dots,\bv_e$ freely generate $M_0$ over $R_{n,0}$.
This implies at once that $\bv_1,\dots,\bv_e$ are linearly independent
over $R_{n,m}$, as otherwise we could rescale by a monomial in
$x_1,\dots,x_n$ to get a nontrivial relation over $R_{n,0}$.
Again by rescaling into $M_0$, we see that every element of $M$
is an $R_{n,m}$-linear combination of $\bv_1,\dots,\bv_e$.
Hence $\bv_1,\dots,\bv_e$ form a basis of $M$ satisfying (b$'$),
as desired.
\end{proof}
\begin{cor} \label{C:regular homs1}
Let $M$ be a regular finite differential module over $R_{n,m}$. 
Embed $F_{(n)}$ into the field $E = k((x_1))\cdots((x_n))$.
Then $H^0(M) = H^0(M \otimes_{R_{n,m}} E)$.
\end{cor}
\begin{proof}
By Theorem~\ref{T:regular}, we may choose a basis of $M$ over
which each $x_i \del_i$ acts via a matrix over $k$.
In terms of this basis, we may write an element $\bv$ of $M \otimes_{R_{n,m}} E$
as a sum $\sum_{J \in \ZZ^n} \bv^{(J)} x^J$ of $k$-vectors times monomials; the
$x_i \del_i$ act independently on each $\bv^{(J)}$
(as in Remark~\ref{R:formal product}).
If $\bv \in H^0(M \otimes_{R_{n,m}} E)$, then for each index $J$ such that
$\bv^{(J)} \neq 0$, for $i=1,\dots,n$, 
the $i$-th component $j_i$ of $J$ must equal the negation of an eigenvalue
of the action of $x_i \frac{\del}{\del x_i}$ on the original basis.
In particular, only finitely many $\bv^{(J)}$ are nonzero,
so we have an element of $M$ itself.
\end{proof}
\begin{cor} \label{C:regular homs}
Let $M$ be a regular finite differential module over $R_{n,m}$. 
Then $H^0(M) = H^0(M \otimes_{R_{n,m}} F_{(i)})$ for $i=1,\dots,n$.
\end{cor}
\begin{proof}
It suffices to check the case $i=n$, for which we apply
Corollary~\ref{C:regular homs1}.
\end{proof}
\begin{cor}\label{C:big regular}
Let  $0 \to M_1 \to M \to M_2 \to 0$ be a short exact sequence of finite
differential modules over $R_{n,m}$. Then $M$ is regular if and only if
both $M_1$ and $M_2$ are regular.
\end{cor}
\begin{proof}
Using condition (c) from Theorem~\ref{T:regular}, this follows from
Corollary~\ref{C:regular}.
\end{proof}

\begin{defn}
Let $M$ be a regular finite differential module over $R_{n,m}$. 
We refer to any $M_0$ as in Theorem~\ref{T:regular}(a$'$) as a
\emph{regulating lattice} for $M$.
For $M_0$ a regulating lattice, for $i=1,\dots,m$,
we call the eigenvalues of $x_i \del_i$ on $M/(x_1,\dots,x_n)M$
the \emph{exponents} of $x_i \del_i$ on $M_0$. (One may extend the
definition to $i > m$, for which the exponents will always be zero.)
\end{defn}

\subsection{Twist-regular connections}

We now can define twist-regular differential modules over $R_{n,m}$
and give a partial analogue of the Hukuhara-Levelt-Turrittin theorem.
\begin{defn}
Let $M$ be a finite differential module over $R_{n,m}$.
We say that $M$ is \emph{twist-regular} 
if $\End(M)$ is regular.
For example, $E(r) \otimes_{R_{n,m}} M$ is twist-regular for any
$r \in R_{n,m}$ and any regular $M$.
\end{defn}

\begin{lemma} \label{L:twist-regular rank 1}
Let $M$ be a nonzero twist-regular finite differential module over $R_{n,m}$.
Then for some finite extension $k'$ of $k$,
$M \otimes_k k'$ has a differential submodule of rank $1$.
\end{lemma}
\begin{proof}
Rather than include $k'$ in the notation, we allow
$k$ to be replaced by a finite extension during the proof.
We induct on $\rank(M)$. If $\rank(M) = 1$ there is nothing to check,
so we assume $\rank(M) > 1$.

Let $V$ be the $k$-span of a basis of $\End(M)$ of the form
described in Theorem~\ref{T:regular}(b$'$).
We may replace $k$ by a finite extension so that the eigenvalues
of each $x_i \del_i$ on $V$ belong to $k$.
Since the $x_i \del_i$ commute, we can decompose $V$ as a direct
sum of joint generalized eigenspaces for the $x_i \del_i$.
Suppose first that there exists such a generalized eigenspace $W$ 
with eigenvalue $\lambda_i$ for $x_i \del_i$,
such that the $\lambda_i$ are not all zero. Pick an eigenvector $\bw$
in $W$. The $h$-fold composition of $\bw$ is again an eigenvector,
with eigenvalue $h\lambda_i$ for $x_i \del_i$; since $W$ is finite-dimensional
and there is a nonzero $\lambda_i$, for sufficiently large $h$,
$h \lambda_i$ is not an eigenvalue. This forces the $h$-fold composition of 
$\bw$ to be zero.

Let $N$ be the differential module over $R_{n,m}$ free on one generator
$\bv$ satisfying $\del_i(\bv) = -\lambda_i x_i^{-1} \bv$
for $i=1,\dots,n$. (Note that $\lambda_i = 0$ for $i > m$.)
Then for each $j$, $\bw$ corresponds to a nonzero morphism
$f_j: M \otimes_{R_{n,m}} N^{\otimes j}$ to $M \otimes_{R_{n,m}} N^{\otimes (j+1)}$;
by the previous paragraph, the composition $f_{h-1} \circ \cdots \circ f_0$
is the zero map. Hence some $f_j$ is not invertible, which means that $f_0$
is not invertible. Since $f_0$ is nonzero, its kernel must be a nontrivial
proper differential submodule $P$ of $M$. Applying the induction hypothesis to
$P$ yields the claim.

The remaining case is the one for which $x_i \del_i$ acts
on $V$ via a nilpotent matrix for each $i$. Let $N_0$ be the trace-zero
summand of $\End(M)$. Then the 
$x_i \del_i$ have a common kernel on $V \cap N_0$; any nonzero
element of that kernel corresponds to a nonzero nilpotent endomorphism of $M$.
Applying the induction hypothesis to the kernel of this endomorphism yields
the claim.
\end{proof}

\begin{theorem} \label{T:turrittin multi}
Let $M$ be a twist-regular finite differential module over $R_{n,m}$.
Then for some $s \in R_{n,m}$, $E(-s) \otimes_{R_{n,m}} M$ is regular. 
\end{theorem}
\begin{proof}
By Lemma~\ref{L:twist-regular rank 1},
there exist a finite extension $k'$ of $k$ 
and a differential submodule $N$ of $M \otimes_k k'$ of rank 1;
we may assume $k'$ is Galois over $k$.
Then $N^\dual \otimes_{R_{n,m}} M$ is a quotient of $\End(M) \otimes_k k'$,
so Corollary~\ref{C:big regular} implies that
$N^\dual \otimes_{R_{n,m}} M$ is regular.

Put $R' = R_{n,m} \otimes_k k'$.
We next check that there exists $s \in R'$ such that
$E(-s) \otimes_{R'} N$ is regular.
Let $\bv$ be a generator of $N$, and write
$\del_i(\bv) = r_i \bv$ for some $r_i \in R'$.
For the actions of $\del_i$ and $\del_j$
to commute, we must have 
\begin{equation} \label{eq:mixed partials}
\frac{\del r_i}{\del x_j} = 
\frac{\del r_j}{\del x_i}.
\end{equation} 
Write $r_i = \sum_{j_1,\dots,j_n \in \ZZ} c_{i,j_1,\dots,j_n} x_1^{j_1}
\cdots x_n^{j_n}$ with $c_{i,j_1,\dots,j_n} \in k'$. 
For $j_1,\dots,j_n \in \ZZ$ all nonnegative, put $s_{j_1,\dots,j_n} = 0$;
otherwise, put
\[
s_{j_1,\dots,j_n} = \frac{c_{i,j_1,\dots,j_{i-1},j_i-1,j_{i+1},\dots,j_n}}{j_i}
\]
for any index $i$ for which $j_i < 0$. This does not depend on $i$
by \eqref{eq:mixed partials}. Put 
$s = \sum_{j_1,\dots,j_n \in \ZZ} s_{j_1,\dots,j_n} x_1^{j_1} \cdots
x_n^{j_n}$; then
$E(-s) \otimes_{R'} N$ is regular, as then is
$E(-s) \otimes_{R_{n,m}} M$.

To conclude, let $\tau$ be any element of the Galois group of $k'$
over $k$, and put
$\tau(s) = \sum_{j_1,\dots,j_n \in \ZZ} \tau(s_{j_1,\dots,j_n}) 
x_1^{j_1} \cdots x_n^{j_n}$. Then
$E(-s) \otimes_{R_{n,m}} M$ and
$E(-\tau(s)) \otimes_{R_{n,m}} M$ are both regular, so
$E(s-\tau(s))$ is regular. However, 
this implies that $\tau(s_{j_1,\dots,j_n}) = s_{j_1,\dots,j_n}$
whenever $j_1,\dots,j_n$ are not all nonnegative; this is also true if
$j_1,\dots,j_n$ are all nonnegative because then both sides are zero.
We conclude that $\tau(s) = s$ for all $\tau$, and hence
$s \in R_{n,m}$. This proves the desired result.
\end{proof}

\subsection{Good decompositions}

We now introduce the notion of a good decomposition of a 
finite differential module over $R_{n,m}$, following
Sabbah \cite{sabbah}. Note that Mochizuki works with a slightly
different definition; see Remark~\ref{R:mochizuki def}.

\begin{defn} \label{D:local model1}
Let $M$ be a finite differential module over $R_{n,m}$.
An \emph{admissible decomposition} of $M$ is an isomorphism
\begin{equation} \label{eq:local model1}
M \cong \bigoplus_{\alpha \in A} E(\phi_\alpha) \otimes_{R_{n,m}} \calR_\alpha
\end{equation}
for some $\phi_\alpha \in R_{n,m}$
(indexed by an arbitrary set $A$)
and some regular differential modules $\calR_\alpha$.
(This corresponds to the notion of an \emph{elementary local model} in 
\cite{sabbah}.)
A \emph{good decomposition} is an admissible decomposition satisfying
the following
two additional conditions.
\begin{enumerate}
\item[(a)]
For $\alpha \in A$, if $\phi_\alpha \notin R_{n,m}$, then
$\phi_\alpha$ has the form $u x_1^{-i_1} \cdots x_m^{-i_m}$
for some unit $u$ in $R_{n,0}$ and some nonnegative integers $i_1,\dots, i_m$.
\item[(b)]
For $\alpha, \beta \in A$, if $\phi_\alpha - \phi_\beta \notin R_{n,m}$,
then $\phi_\alpha - \phi_\beta$ has the form $u x_1^{-i_1} \cdots x_m^{-i_m}$
for some unit $u$ in $R_{n,0}$ and some nonnegative integers $i_1,\dots, i_m$.
\end{enumerate}
\end{defn}

Let us record an important consequence of this definition.
\begin{lemma} \label{L:comparable}
Let $M$ be a finite differential module over $R_{n,m}$
admitting a good decomposition \eqref{eq:local model1}.
Then the $R_{n,0}$-submodules of $R_{n,m}$ generated by those
$\phi_\alpha \notin R_{n,0}$ (for $\alpha \in A$)
are totally ordered under containment.
\end{lemma}
\begin{proof}
(Compare \cite[\S 3.1.2]{mochizuki2}.)
Choose $\alpha,\beta \in A$ such that $\phi_\alpha, \phi_\beta 
\notin R_{n,0}$.
Suppose first that $\phi_\alpha - \phi_\beta \in R_{n,0}$.
In this case, $\phi_\alpha$ and $\phi_\beta$ cannot have distinct
$x_i$-adic valuations for any $i \in \{1,\dots,m\}$, otherwise 
the lesser of the two valuations would be negative and would coincide
with the $x_i$-adic valuation of $\phi_\alpha - \phi_\beta$.
By condition (a) in Definition~\ref{D:local model1},
this is enough to ensure that
$\phi_\alpha R_{n,0} = \phi_\beta R_{n,0}$.

Suppose next that $\phi_\alpha - \phi_\beta \notin R_{n,0}$.
By conditions (a) and (b) of Definition~\ref{D:local model1}, the quantities
\[
\phi_\alpha/\phi_\beta, \qquad
1-\phi_\alpha/\phi_\beta = (\phi_\beta - \phi_\alpha)/\phi_\beta
\]
are units in $R_{n,m}$.
Since $\phi_\alpha/\phi_\beta$ is a unit in $R_{n,m}$, when we write 
$\phi_\alpha/\phi_\beta$
as a sum of monomials, there must be a 
least monomial $\mu$ under divisibility.
Similarly, $1 - \phi_\alpha/\phi_\beta$ must have a least monomial
under divisibility;
this forces $R_{n,0}$ and $\mu R_{n,0}$ to be comparable under 
containment, 
yielding the claim.
\end{proof}

\begin{remark} \label{R:mochizuki def}
The notion of a good decomposition used here, while consistent with that of
Sabbah, has a slight mismatch
with the notion of a \emph{good set of irregular values} introduced
by Mochizuki
in \cite[Definition~3.1.1]{mochizuki2}. In our language, Mochizuki
adds the additional restriction that the $R_{n,0}$-submodules of $R_{n,m}$
generated by those $\phi_\alpha - \phi_\beta \notin R_{n,0}$
(for $\alpha,\beta \in A$) are totally ordered under containment.
On one hand, any good decomposition under this extra condition is also
a good decomposition in our sense. On the other hand, if both
 $M$ and $\End(M)$ admit a good decomposition in our sense, then the
decomposition of $M$ is also good in Mochizuki's sense,
by Lemma~\ref{L:comparable}. This discrepancy will not matter for our
final result (see Remark~\ref{R:mochizuki def2});
our choice of Sabbah's definition is partly influenced by the fact
that this definition leads to a relatively simple
numerical criterion for good formal structures
(Theorem~\ref{T:criterion}).
\end{remark}

\subsection{A numerical criterion}

We now give our numerical criterion for existence
of good decompositions in terms of the variation of irregularity.

\begin{defn}
For $k'$ a finite extension of $k$ and $h$ a positive integer,
define
\[
R_{n,m}(k',h) = k' \llbracket x_1^{1/h}, \dots, x_m^{1/h}, x_{m+1},
\dots, x_n \rrbracket [x_1^{-1/h}, \dots, x_m^{-1/h}],
\]
again viewed as a differential ring with derivations
$\del_1,\dots,\del_n$. For $M$ a finite differential
module over $R_{n,m}(k',h)$, define $f_i(M,r)$ and $F_i(M,r)$
as in Definition~\ref{D:spectral}. This differs from what you get
if you identify $R_{n,m}(k',h)$ with a copy of $R_{n,m}$ over the
field $k'$ in the variables $x_1^{1/h},\dots,x_m^{1/h},x_{m+1},\dots,x_n$,
but only up to a change of normalization in $r$-space.
\end{defn}

\begin{theorem} \label{T:criterion}
Let $M$ be a differential module over $R_{n,m}$ of finite rank $d$.
Then the following conditions are equivalent.
\begin{enumerate}
\item[(a)]
There exist a finite extension $k'$ of $k$ and a
positive integer $h$ such that $M \otimes_{R_{n,m}} R_{n,m}(k',h)$
admits a good decomposition.
\item[(b)]
The functions $F_1(M,r), \dots, F_d(M,r)$
and $F_{d^2}(\End(M),r)$ are all
linear in $r$ (and constant in $r_{m+1},\dots,r_{n}$).
\item[(c)]
The functions $F_d(M,r)$
and $F_{d^2}(\End(M),r)$ are both
linear in $r$ (and constant in $r_{m+1},\dots,r_{n})$.
\end{enumerate}
\end{theorem}
Recall that in (b) and (c), the linear condition implies the constant
condition; see Remark~\ref{R:m equals 0}.
\begin{proof}
Suppose (a) holds.
To check (b), by Lemma~\ref{L:base change}, we may reduce to the case
where $k = k'$ and $h=1$, i.e., where $M$ itself admits a good decomposition.
Set notation as in \eqref{eq:local model1}.
In particular, each $\phi_\alpha$ or $\phi_\alpha - \phi_\beta$
not belonging to $R_{n,0}$ is a unit in $R_{n,m}$.

By Lemma~\ref{L:exp image},
\[
F_1(E(\phi_\alpha),r) = \begin{cases} 0 & \phi_\alpha \in R_{n,0} \\
\log |\phi_\alpha|_r & \phi_\alpha \notin R_{n,0}.
\end{cases}
\]
In particular, $F_1(E(\phi_\alpha),r)$ is linear
in $r$.
By Lemma~\ref{L:comparable},
the $R_{n,0}$-submodules of $R_{n,m}$ generated by those
$\phi_\alpha \notin R_{n,0}$
are totally ordered under containment;
it follows that each $F_i(M,r)$ is equal to
$F_1(E(\phi_\alpha), r)$ for some $\alpha \in A$ depending only on $i$, not
on $r$. Hence the $F_i(M,r)$ are linear in $r$.
Since $F_{d^2}(\End(M), r)$ is a weighted sum of
the $F_1(E(\phi_\alpha-\phi_\beta),r)$ over $\alpha, \beta \in A$, 
it too is linear in $r$. This implies (b).

Since (b) implies (c) trivially, it suffices to assume (c) and deduce (a),
which we now do.
By hypothesis, $F_{d^2}(\End(M), r)$ is linear in $r$.
Let $j$ be the index defined by Corollary~\ref{C:convex};
then for all $r = (r_1,\dots,r_n)$ with $r_1, \dots, r_n > 0$,
the absolute scale multiset of $\End(M)$ contains 1 with multiplicity
exactly $d^2-j$.
By Theorem~\ref{T:refined decomp}, we obtain a direct
sum decomposition $\End(M) \cong N_1 \oplus N_2$ with
$\rank(N_2) = j$, in which
for all $r = (r_1,\dots,r_n)$ with $r_1, \dots, r_n > 0$,
the absolute scale multiset of $N_{1,r}$ has all elements equal to 1
while the absolute scale multiset of $N_{2,r}$ has all elements strictly greater than 1.
In particular, $F_1(N_1,r) = 0$ for all $r \in (0, +\infty)^n$;
by continuity (Theorem~\ref{T:convex}),
$F_1(N_1,r) = 0$ for all $r \in (0, +\infty)^n$.
By Theorem~\ref{T:convex}, $N_1$ is regular.

Put $r = (1, \dots, 1)$. Embed $F_r$
into the field
\[
E = k((x_2/x_1))\cdots((x_n/x_1))((x_1)).
\]
By Theorem~\ref{T:Turrittin}, 
for some finite extension $E'$ of $E$, there is a decomposition of
$M \otimes_{R_{n,m}} E'$ into twist-regular diferential submodules.
The projectors onto the summands in this decomposition
correspond to horizontal elements of $\End(M) \otimes_{R_{n,m}} E'$,
which must necessarily belong to $N_1 \otimes_{R_{n,m}} E'$
(or else the scale multiset of $N_{2,r}$ would contain 1
with nonzero multiplicity).

By Remark~\ref{R:puiseux},
the extension $E'$ embeds into
\[
k'((x_2^{1/h}/x_1^{1/h}))\cdots((x_n^{1/h}/x_1^{1/h}))((x_1^{1/h}))
\]
for some finite extension $k'$ of $k$ and some positive integer $h$.
By Corollary~\ref{C:regular homs1} applied to $N_1$,
the decomposition obtained above descends to a decomposition of
\[
M \otimes_{R_{n,m}} k' \llbracket x_1^{1/h},\dots,x_n^{1/h} \rrbracket
[x_1^{-1/h}, \dots, x_m^{-1/h}].
\]
We next wish to descend this decomposition to
$M \otimes_{R_{n,m}} R_{n,m}(k', h)$, i.e., to eliminate
$x_i^{1/h}$ for $i=m+1,\dots,n$. We may do
this by inspecting the proof of Corollary~\ref{C:regular homs1}:
if $i > m$, then the exponents of $x_i \del_i$
are integers, so we never encounter a fractional power of $x_i$ in the
series expansion of a horizontal element.

We thus have a decomposition $M \otimes_{R_{n,m}} R_{n,m}(k',h)
\cong \oplus_i M_i$
in which $F_1(\End(M_i), (1,\dots,1)) = 0$ for each $i$.
Put $d_i = \rank(\End(M_i))$.
Since $\End(M_i)$ is a summand of $\End(M)$, 
by Corollary~\ref{C:drop affine},
$F_{d_i}(\End(M_i), r)$ is linear in $r$.
By Corollary~\ref{C:convex}, the equality
$f_{1}(\End(M_i), (1,\dots,1)) = 0$
implies $f_{1}(\End(M_i),r) = 0$ for all $r$.
By Theorem~\ref{T:regular}, $M_i$ is twist-regular.

By Theorem~\ref{T:turrittin multi}, 
we can choose $s_i \in R_{n,m}(k',h)$ so that
$E(-s_i) \otimes_{R_{n,m}} M_i$ is a regular finite differential
module over $R_{n,m}(k',h)$.
By Corollary~\ref{C:drop affine}, 
$F_1(E(s_i),r)$ is linear in $r$.
On the other hand, by Lemma~\ref{L:exp image}, we have
\[
F_1(E(s_i),r) = \max\{0, \log |s_i|_r\}.
\]
If $F_1(E(s_i),r) = 0$ identically, this forces $s_i \in R_{n,0}$.
Otherwise, we must have
$F_1(E(s_i),r) = a_1 r_1 + \dots + a_m r_m$ for some
nonnegative integers $a_1,\dots,a_m$ which are not all zero.
In this case, $t = x_1^{a_1} \cdots x_m^{a_m} s_i$ 
has the property that $|t|_r = 0$ for all $r \in [0, +\infty)^n$;
this forces $t$ to be a unit in $R_{n,0}$.

We thus conclude that condition (a) of Definition~\ref{D:local model1}
is satisfied.
Similarly, the fact that
$F_1(E(s_i - s_j),r)$ is linear in $r$ 
implies condition (b) of Definition~\ref{D:local model1}.
Hence $M \otimes_{R_{n,m}} R_{n,m}(k',h)$ admits a good decomposition,
proving (a).
\end{proof}

\begin{remark}
The equivalence between (b) and (c) in Theorem~\ref{T:criterion}
may seem a bit surprising at first. It becomes less surprising when
one compares it to the analogous situation that occurs in trying to
factor a monic 
polynomial $P \in R_{n,m}[T]$. Let $r_1,\dots,r_n$ be the roots of
$P$ in an algebraic closure of $\Frac(R_{n,m})$. Let $Q$ be the polynomial whose
roots are $r_i - r_j$ for $i,j \in \{1,\dots,n\}$ distinct.
Then condition (c) is analogous to the hypothesis that each of $P$ and $Q$
is equal to a power of $T$ times a polynomial with 
invertible constant coefficient. 
This condition for $Q$ implies that the discriminant of the
radical $R$ of $P$ (the maximal square-free factor of $P$)
is a unit, so the ring extension $R_{n,m}[T]/(R(T))$
of $R_{n,m}$ is \'etale. 
By Abhyankar's lemma,
the $r_i$ must all belong to $R_{n,m}(k',h)$
for some positive integer $h$ and some finite extension $k'$
of $k$, and each $r_i$ and $r_i - r_j$ must be either zero or a unit.
\end{remark}

\begin{remark}
One can formulate an analogue of the equivalence between (b) and (c)
in Theorem~\ref{T:criterion} in the setting of differential modules
over $p$-adic polyannuli, as in \cite{kedlaya-xiao}. This analogue
is likely correct, but this has not yet been checked.
\end{remark}

\begin{example}
It was asserted in a prior version of this paper that condition (a)
of Theorem~\ref{T:criterion} is equivalent to the fact that
the functions 
\[
F_1(M,r), \dots, F_d(M,r), F_1(\End(M),r), \dots, F_{d^2}(\End(M),r)
\]
are all
linear in $r$.
This is false, as shown by the following example. Take $n = m = 2$, and
\[
M = E(x_1^{-3} x_2^{-3}) \oplus E(x_1^{-3} x_2^{-3} + x_1^{-1})
\oplus E(x_1^{-2} x_2^{-2}) \oplus E(x_1^{-2} x_2^{-2} + x_2^{-1}).
\]
Then $M$ admits a good decomposition, but
\[
F_5(\End(M), (r_1,r_2)) = 3r_1 + 3r_2 + \max\{r_1, r_2\}
\]
is not linear.
\end{example}

\subsection{Good decompositions and Deligne-Malgrange lattices}
\label{subsec:good DM}

One application of good decompositions is to the analysis of
Deligne-Malgrange lattices. This again follows \cite{malgrange-conn2}.
\begin{hypothesis}
Throughout \S~\ref{subsec:good DM},
let $\tau: \overline{k}/\ZZ \to \overline{k}$ be an admissible 
section of the quotient $\overline{k} \to \overline{k}/\ZZ$.
Such a section exists by Lemma~\ref{L:admissible}.
\end{hypothesis}

\begin{defn}
Let $M$ be a finite differential module over $R_{n,m}$.
The \emph{Deligne-Malgrange lattice} in $M$ 
(with respect to $\tau$) is the 
set $M_0$ of $\bv \in M$ such that for $i=1,\dots,m$, $\bv$ belongs to
the Deligne-Malgrange lattice of $M \otimes_{R_{n,m}} F_{(i)}$
(with respect to $\tau$). Note that $M_0$ is a finite $R_{n,0}$-module
(by Lemma~\ref{L:make lattice}) which is an $R_{n,0}$-lattice in
$M$ stable under $x_1 \del_1,\dots,x_m \del_m, \del_{m+1},
\dots,\del_n$. Moreover, for $i=1,\dots,m$, $M_0 \otimes_{R_{n,0}} 
\gotho_{(i)}$
is the Deligne-Malgrange lattice of $M_0 \otimes_{R_{n,m}} F_{(i)}$
(with respect to $\tau$). 
\end{defn}

\begin{lemma} \label{L:DM regular}
Let $M$ be a regular finite differential module over $R_{n,m}$.
Then there is a unique regulating lattice $M_0$ in $M$ such that
the exponents of $x_i \del_i$ on $M_0$ are in the image of $\tau$
for $i=1,\dots,m$.
\end{lemma}
\begin{proof}
Similar to Proposition~\ref{P:same exponents}.
\end{proof}

\begin{prop} \label{P:DM exists}
Suppose that $M$ satisfies the conditions of Theorem~\ref{T:criterion}.
Then the Deligne-Malgrange lattice $M_0$ of $M$ is a free $R_{n,0}$-module.
Moreover, if $M$ itself admits a good decomposition,
then this decomposition induces a direct sum decomposition of $M_0$.
\end{prop}
\begin{proof}
Suppose first that $M$ itself admits a good decomposition
with notation as in \eqref{eq:local model1}.
For each $\alpha \in A$, let $\calR_{\alpha,0}$ be the 
regulating lattice in $\calR_\alpha$ produced by Lemma~\ref{L:DM regular}.
Then $M_0$ is the direct
sum over $\alpha \in A$ of $\bv \otimes \calR_{\alpha,0}$,
for $\bv$ the canonical generator of $E(\phi_\alpha)$.
In particular, it is free over $R_{n,0}$.

In the general case, 
we are given that 
$M' = M \otimes_{R_{n,m}} R_{n,m}(k',h)$ admits a good decomposition
for some positive integer $h$ and some finite Galois extension $k'$ of $k$
containing a primitive $h$-th root of unity $\zeta_h$.
Let $M'_0$ be the Deligne-Malgrange lattice of $M'$; then
$M_0 = M \cap M'_0$.
To check that $M_0$ is free as an $R_{n,0}$-module, note that
$M'_0$ is finite free as an $R_{n,0}$-module. We may split $M'_0$ into 
a direct sum of isotypical representations for the Galois group of $k'$
over $k$ and the product of $m$ copies of $\ZZ/h\ZZ$, in which the
$i$-th copy acts with $j \in \ZZ/h\ZZ$ applying the substitution 
$x_i \mapsto \zeta_h^j x_i$. (Note the Galois group will not commute with
the product if $\zeta_h \notin k$.) In this direct sum, $M_0$ occurs 
as the direct summand corresponding to the trivial representation; since it
is a direct summand of a finite free $R_{n,m}$-module, it is finite
projective and hence free by Nakayama's lemma.
\end{proof}

\begin{remark}
To reconcile our work with that of Mochizuki \cite{mochizuki2},
it is important to check that (in Mochizuki's language)
the existence of a good decomposition implies the existence of a 
\emph{good Deligne-Malgrange lattice}.
The last assertion of Proposition~\ref{P:DM exists} 
asserts this in the purely formal setting, where we have completed
at a point. The general case can be deduced from this one by
a decompletion argument; however, we will defer this argument to
a subsequent paper. (It is similar to the argument promised by
Remark~\ref{R:formalizing}.)
\end{remark}

\subsection{Coordinate independence}
\label{subsec:coord ind}

Although the results in \S~\ref{sec:power series} and \S~\ref{sec:numerical}
are stated in terms of the differential ring $R_{n,m}$ equipped with a 
distinguished isomorphism 
with $k \llbracket x_1,\dots,x_n \rrbracket[x_1^{-1},\dots,x_m^{-1}]$,
we wish to apply it in circumstances where no such distinguished isomorphism
exists. It is thus important to identify notions which are invariant under
automorphisms of $R_{n,m}$.

\begin{lemma} \label{L:preserve subring}
Any $k$-linear automorphism of $R_{n,m}$ acts as an automorphism also on $R_{n,0}$.
\end{lemma}
\begin{proof}
Let $\phi: R_{n,m} \to R_{n,m}$ be an automorphism; then $\phi$ also
acts as an automorphism of the group of units $R_{n,m}^\times$.
We first verify that the subgroup $R_{n,0}^\times$ is stable under this action.
For $x \in R_{n,m}^\times$, there exists a unique $m$-tuple $(i_1,\dots,i_m) \in 
\ZZ^m$ such that $x_1^{-i_1} \cdots x_m^{-i_m} x \in R_{n,0}^\times$.
Now consider the set $S_x$ of those $c \in k^\times$ for which $x + cx^{-1}
\in R_{n,m}^\times$.
\begin{itemize}
\item
If $(i_1,\dots,i_m) = (0,\dots,0)$, then $S_x$ consists of $k^\times$
minus a single element. (If $x$ has constant
term $y$, the missing element is $-y^2$.)
\item
If $(i_1,\dots,i_m) \neq (0,\dots,0)$ but the nonzero elements of
$\{i_1,\dots,i_m\}$ all have the same sign, then $S_x = k^\times$.
\item
If $(i_1,\dots,i_m) \neq (0,\dots,0)$ but the nonzero elements of
$\{i_1,\dots,i_m\}$ do not all have the same sign, then $S_x = \emptyset$.
\end{itemize}
For $k \neq \FF_2$, this implies that $R_{n,0}^\times$ is stable under $\phi$.
For $k = \FF_2$, we may replace $R_{n,m}$ by $R_{n,m} \otimes_k \FF_4$ to reach
the same conclusion.

Since $R_{n,0}^\times$ is stable under $\phi$, so is $k + R_{n,0}^\times = R_{n,0}$.
This proves the claim.
\end{proof}

\begin{prop} \label{P:automorphism}
Let $M$ be a finite differential module over $R_{n,m}$.
Let $\phi$ be an automorphism of $R_{n,m}$.
\begin{enumerate}
\item[(a)]
If $M$ is regular, then so is $\phi^* M$.
\item[(b)]
If $M$ is twist-regular, then so is $\phi^* M$.
\item[(c)]
Any good decomposition of $M$ pulls back to a good decomposition of $\phi^* M$.
\end{enumerate}
\end{prop}
\begin{proof}
By Lemma~\ref{L:preserve subring}, $\phi$ acts also on $R_{n,0}$.
Then $\phi$ must also act on the set of principal prime ideals in $R_{n,0}$
which are not the contractions of prime ideals of $R_{n,m}$; these are
precisely $\{(x_1),\dots,(x_m)\}$. Consequently, there exists a permutation
$\pi$ of $\{1,\dots,m\}$ such that $|\cdot|_r = |\phi(\cdot)|_{\pi(r)}$.
Since it suffices to check condition (c) of Theorem~\ref{T:regular}
for $r = e_1,\dots,e_m$, this condition is invariant under $\phi^*$.
This proves (a).

Given (a), (b) follows from the fact that $\End(\phi^* M) \cong \phi^* \End(M)$.
Given (a), (c) follows by inspection of Definition~\ref{D:local model1}
and the fact that $\phi$ acts on $R_{n,m}^\times/R_{n,0}^\times \cong \ZZ^m$
via the permutation $\pi$ of the coordinates.
\end{proof}

\section{Valuative trees}
\label{sec:valuative}

Having established a numerical criterion for existence of good decompositions
of formal differential modules, we now turn to the question of ensuring
the existence of such decompositions after performing a suitable blowing up.
For this, we need some tools for studying the birational geometry of a surface;
the tool we will need is a version of the \emph{valuative tree}
described by Favre and Jonsson \cite{favre-jonsson}, although we prefer
to describe it in the language of Berkovich \cite{berkovich}
in terms of the geometry of a certain nonarchimedean analytic
space. Other useful sources are \cite{baker-aws} and \cite{baker-rumely}.
(The corresponding situation in higher dimensions is more complicated,
and will thus require a more sophisticated valuation-theoretic argument.)

\setcounter{theorem}{0}
\begin{notation}
Throughout \S~\ref{sec:valuative},
let $k$ be an algebraically closed field of characteristic $0$.
Let $\CC_x$ be a completed algebraic closure of $k((x))$,
identified with the field of formal Puiseux series
$a_0 x^{i_0} + a_1 x^{i_1} + \cdots$ with $a_j \in k$,
$i_j \in \QQ$, $i_0 < i_1 < \cdots$, and $i_j \to +\infty$ as $j \to \infty$.
Let $|\cdot|_x$ be the $x$-adic norm on 
$\CC_x$, normalized by $|x|_x = e^{-1}$.
\end{notation}

\subsection{The Berkovich closed unit disc}

Berkovich's notion of an nonarchimedean analytic space arises from
the failure of the classical Gelfand-Naimark duality theorem to carry over to
the realm of nonarchimedean analysis. In this paper,
we will only use this notion in the
simplest case, that of the closed unit disc.

\begin{defn}
The \emph{Berkovich closed unit disc} $\DD$ over $k((x))$
is the set of nonarchimedean multiplicative seminorms on $k((x))[y]$
compatible with $k((x))$ and bounded above by the 1-Gauss norm.
Similarly, the Berkovich closed unit disc $\DD'$ over $\CC_x$
is the set of nonarchimedean multiplicative seminorms on $\CC_x[y]$
compatible with $\CC_x$ and bounded above by the 1-Gauss norm.
There is a natural restriction map $\DD' \to \DD$; this turns
out to be the quotient by the natural Galois action on $\DD'$
\cite[Corollary~1.3.6]{berkovich}.
In $\DD$ and $\DD'$, we have respective points $\alpha_{\DD}$ and
$\alpha_{\DD'}$ corresponding to the 1-Gauss norm; we call these the
\emph{Gauss points}.
\end{defn}

\begin{remark}
The analysis made in \cite{berkovich} only covers the Berkovich
closed unit disc over an algebraically closed complete nonarchimedean
field, such as $\CC_x$. The supplemental analysis needed to deal with
general complete nonarchimedean base fields appears 
in \cite[\S 2]{kedlaya-part4}, and possibly elsewhere; an equivalent
analysis in different language appears in \cite{favre-jonsson}.
\end{remark}

One can give an explicit description of some of the points 
of $\DD$ and $\DD'$.
\begin{defn}
For $z \in \gotho_{\CC_x}$ and $r \in [0,1]$,
the function $\alpha'_{z,r}: 
\CC_x[y] \to [0, +\infty)$ which takes $P(y)$ to 
the $r$-Gauss norm
of $P(y+z)$ is a multiplicative seminorm, so defines a point
of $\DD'$. Let $\alpha_{z,r} \in \DD$ be the restriction of this point.
We call $r$ the \emph{radius} of such a point. 
\end{defn}

\begin{remark}
It can be shown (though we will not need it here) that
$\alpha_{z,r}(P)$ computes the supremum of $|P(t)|$ for
$t$ running over the closed disc 
\[
D_{z,r} = \{t \in \CC_x: |t-z| \leq r\}.
\]
In particular, $\alpha_{z,0}(P) = |P(z)|$;
that is, the $\alpha_{z,0}$ correspond to the points of the \emph{rigid}
analytic closed unit disc, in the sense of Tate.
\end{remark}

Using the explicit points $\alpha_{z,r}$, 
we have the following classification.
\begin{prop} \label{P:classify points}
Each element of $\DD$ is of exactly one of the following four types.
\begin{enumerate}
\item[(i)]
A point of the form $\alpha_{z,0}$ for some $z \in \gotho_{\CC_x}$.
\item[(ii)]
A point of the form $\alpha_{z,r}$ for some $z \in \gotho_{\CC_x}$
and $r \in (0,1]$ with $\log r \in \QQ$.
\item[(iii)]
A point of the form $\alpha_{z,r}$ for some $z \in \gotho_{\CC_x}$
and $r \in (0,1]$ with $\log r \notin \QQ$.
\item[(iv)]
The infimum of a sequence $\alpha_{z_i,r_i}$ in which the closed 
discs $D_{z_i,r_i}$
form a decreasing sequence
with positive limiting radius, whose intersection contains no $\CC_x$-points.
\end{enumerate}
The same classification (with suitable primes added)
holds over $\DD'$, and the type is preserved
by the restriction $\DD' \to \DD$.
\end{prop}
\begin{proof}
See \cite[1.4.4]{berkovich} for the primed case and
\cite[Proposition~2.2.7]{kedlaya-part4} for the unprimed case.
\end{proof}

\begin{remark}
One may identify $\DD$ with the \emph{relative valuative tree} of Favre-Jonsson
\cite{favre-jonsson}. The classification made by Favre-Jonsson compares to
ours as follows.
\begin{enumerate}
\item[(i)] includes all curve valuations and some infinitely singular 
valuations.
\item[(ii)] includes all divisorial valuations.
\item[(iii)] includes all irrational quasimonomial valuations.
\item[(iv)] includes some infinitely singular valuations.
\end{enumerate}
\end{remark}

\begin{defn}
Define the \emph{degree} of $\alpha \in \DD$, denoted $\deg(\alpha)$,
as the number of preimages
of $\alpha$ in $\DD'$, or $+\infty$ if this set is not finite. 
The degree of a point of type (ii) or (iii) is always finite,
since we can always write $\alpha = \alpha_{z,r}$ for $z$ in a \emph{finite}
extension of $k((x))$.
\end{defn}

\subsection{A partial ordering}

So far we have only discussed the structure of the Berkovich
closed unit disc as a set. We now equip it with the additional
structure of a partial ordering. There are additional useful
structures one can impose, including metric and topological
structures; see \cite{berkovich} and \cite{favre-jonsson}
for further discussion. 

\begin{defn} \label{D:domination}
For $\alpha, \beta \in \DD$, we say that \emph{$\alpha$ dominates
$\beta$}, denoted $\alpha \geq \beta$,
if for all $P \in k((x))[y]$, we have $\alpha(P) \geq \beta(P)$;
we write $\alpha > \beta$ if $\alpha \geq \beta$ and $\alpha \neq \beta$.
The relation of dominance is transitive, and the
Gauss point $\alpha_\DD$ is maximal.
We define domination for $\DD'$ similarly; then
$\alpha, \beta \in \DD$ satisfy $\alpha \geq \beta$ if and only if
they have lifts $\alpha', \beta' \in \DD'$ satisfying
$\alpha' \geq \beta'$ \cite[Lemma~2.2.9]{kedlaya-part4}.
Moreover, in this case, \emph{every} lift $\beta'$ is dominated by some
$\alpha'$, and likewise every lift $\alpha'$ dominates some $\beta'$.

It is evident that points of type (i) are minimal under domination,
whereas points of type (ii) and (iii) are not.
It is less evident, but true, that points of type (iv)
are minimal \cite[Proposition~2.2.7]{kedlaya-part4}.
\end{defn}

\begin{defn} \label{D:radius}
Define the \emph{radius} of $\alpha \in \DD$ 
as the infimum of the values of $r$ for which $\alpha_{z,r} \geq \alpha$
for some $z \in \gotho_{\CC_x}$.
Given $\alpha, \beta \in \DD$ with $\alpha \geq \beta$,
we necessarily have $r(\alpha) \geq r(\beta)$; 
conversely, for $\beta \in \DD$ and $r \in [r(\beta),1]$, there is a
unique $\alpha \in \DD$ with $r(\alpha) = r$ and $\alpha \geq \beta$
%(The proof will appear in an upcoming version of \cite{kedlaya-part4},
%so we omit it here.)
\cite[Lemma~2.2.12]{kedlaya-part4}.
For $\alpha \geq \beta$,
using Proposition~\ref{P:classify points}, we may construct a 
map $h_{\alpha,\beta}$
from $[-\log r(\alpha), -\log r(\beta)]$ to a subset of
$\DD$ such that $(-\log r) \circ h_{\alpha,\beta}$ is the identity map,
and $\alpha \geq h_{\alpha,\beta}(s) \geq \beta$ for all
$s \in [-\log r(\alpha), -\log r(\beta)]$.
\end{defn}

\begin{defn}
From the previous observations, it follows that for any 
$\alpha, \beta, \gamma, \delta \in \DD$ with
$\alpha \geq \beta \geq \delta$ and
$\alpha \geq \gamma \geq \delta$,
we must have either $\beta \geq \gamma$ or $\gamma \geq \beta$.
This allows us to view $\DD$ as a tree with each point at height
equal to its radius; the leaves of this tree of zero radius
are the points of type (i), while the leaves of positive radius are the
points of type (iv).

For $\alpha$ of type (ii) or (iii), we define the \emph{branches}
of a point $\alpha$ to be the equivalence classes of
the following equivalence relation. For $\beta, \gamma \in \DD$
with $\alpha > \beta, \gamma$,
we say that $\beta \sim \gamma$ if there exists
some $\delta \in \DD$ with $\alpha > \delta > \beta,\gamma$.
Note that a point $\alpha_{z,r}$ of type (ii) has infinitely many branches;
for instance, if we write $r = |x|^{a/b}$ and choose $r' \in (0,r)$,
then the points $\alpha_{z + cx^{a/b},r'}$ for $c \in k$
lie on different branches as long as they are distinct
(and each such point occurs for only finitely many $c$).
By contrast, a point $\alpha_{z,r}$ of type (iii) has only one branch, because 
any $z' \in \CC_x$ with $|z - z'| \leq r$ in fact satisfies
$|z-z'| < r$.
\end{defn}

\begin{lemma} \label{L:degree}
If $\alpha, \beta \in \DD$ and $\alpha \geq \beta$, then
$\deg(\alpha) \leq \deg(\beta)$.
\end{lemma}
\begin{proof}
If $\alpha_1', \dots, \alpha'_n$ are distinct elements of $\DD'$ lifting
$\alpha$, then as in Definition~\ref{D:domination},
we can find $\beta'_1,\dots,\beta'_n$ lifting $\beta$ with
$\alpha'_i \geq \beta'_i$ for $i=1,\dots,n$. 
In particular, we must have $\beta'_i \neq \beta'_j$ for $i \neq j$, otherwise
$\alpha'_i = \alpha'_j$ by Definition~\ref{D:radius}. This implies
the desired inequality.
\end{proof}

\begin{lemma} \label{L:degree2}
For  $\alpha \in \DD$ of type (ii) or (iii), for $r = r(\alpha)$, $\deg(\alpha)$
is the minimum of $[k((x))(z):k((x))]$ over all $z$ for which
$\alpha = \alpha_{z,r}$.
\end{lemma}
\begin{proof}
If $\alpha = \alpha_{z,r}$, then each lift of $\alpha$ to $\DD'$ has the form
$\alpha_{z',r}$ for some conjugate $z'$ of $z$. This forces
$\deg(\alpha) \leq [k((x))(z):k((x))]$. Conversely, since $k$ is algebraically
closed of characteristic $0$, the absolute Galois group
of $k((x))$ acts on the lifts of $\alpha$ to $\DD'$ via the quotient 
$\Gal(k((x^{1/n}))/k((x)))$ with $n = \deg(\alpha)$ (by Remark~\ref{R:puiseux}).
It suffices to check that $\alpha = \alpha_{z,r}$ for some
$z \in k((x^{1/n}))$; for this, we may reduce to the case $n = 1$.
Pick $z \in k((x^{1/m}))$ for some positive integer
$m$, such that $\alpha = \alpha_{z,r}$.
Write $z = \sum_{i \in \ZZ} z_{i/m} x^{i/m}$ with $z_{i/m} \in k$.
For any Galois conjugate $z'$ of $z$, we must have
$z' = \sum_{i \in \ZZ} \zeta^i z_{i/m} x^{i/m}$ for some
$m$-th root of unity $\zeta$.
Since $\alpha_{z,r} = \alpha_{z',r}$ by our hypothesis that $n=1$, 
we have $|x-z|_{\alpha_{z',r}} = r$;
this forces $|z-z'| \leq r$. 
Since this is true for all possible choices of $\zeta$, we must have
$z_{i/m} = 0$ for all $i$ such that $i/m \notin \ZZ$ and $|x^{i/m}| > r$.
Consequently, $\alpha = \alpha_{z'',r}$
for 
\[
z'' = \sum_{i \in \ZZ} z_i x^i \in k((x)),
\]
as desired.
\end{proof}

\begin{lemma} \label{L:infinite degree}
If $\alpha \in \DD$ is of type (iv), then 
for any sequence $\alpha_{z_i,r_i}$ as in Proposition~\ref{P:classify points},
we have $\deg(\alpha_{z_i,r_i}) \to +\infty$ as $i \to +\infty$.
(Hence $\deg(\alpha) = +\infty$ by Lemma~\ref{L:degree}.)
\end{lemma}
\begin{proof}
Suppose on the contrary that for some $n$, 
$\deg(\alpha_{z_i,r_i}) \leq n$ for all $i$.
By Lemma~\ref{L:degree2}, we can rechoose $z_i$ to be in
$k((x^{1/n!}))$ without changing $\alpha_{z_i,r_i}$ or $D_{z_i,r_i}$.
However, since $k((x^{1/n!}))$ is discretely valued, 
any decreasing intersection of balls in $k((x^{1/n!}))$ is nonempty
(that is, this field is \emph{spherically complete}).
In particular, the intersection of the $D_{z_i,r_i}$ cannot be empty,
a contradiction. 
\end{proof}

\begin{remark}
The Berkovich unit disc may be defined over any complete nonarchimedean field,
and many of the aforementioned properties carry over;
this is described in \cite[\S 2.2]{kedlaya-part2},
and will be used in a subsequent paper in this series to treat
higher-dimensional varieties.
One notable property that fails to generalize is
Lemma~\ref{L:infinite degree}, because
one can have a point of type (iv) of degree 1 if the base field is not 
spherically complete. This can even happen 
when the base field is $k((x))$ with $k$
of positive characteristic, because then the perfect closure of $k((x))$ is not
spherically complete.
\end{remark}

\subsection{Subharmonic functions and skeleta}

We now introduce the Berkovich \emph{open} unit disc,
and identify a useful class of real-valued functions on it.
\begin{defn}
Let $\DD_0$ be the set of $\alpha \in \DD$ for which $\alpha_{0,r} \geq \alpha$
for some $r \in (0,1)$, together with $\alpha_\DD$.
That is, $\DD_0$ is obtained from $\DD$ by removing the branches at the Gauss
point other than the branch represented by $\alpha_{0,0}$.
Define $\DD'_0$ analogously. 
\end{defn}

\begin{defn}
Let $E$ (resp.\ $E'$) be the subset of $\DD_0$ 
(resp.\ $\DD_0'$) consisting of points of type (ii) or (iii).
For $f: E \to \RR$ any function,
for $\alpha \in E$ with $\alpha \neq \alpha_\DD$,
let $f'_-(\alpha)$ denote the left derivative of 
$f \circ h_{\alpha_\DD,\alpha}$ at $-\log r(\alpha)$,
if it exists.
For $\alpha, \beta \in E$ with $\alpha > \beta$,
let $f'_+(\alpha, \beta)$ denote the right derivative of
$f \circ h_{\alpha,\beta}$ at $-\log r(\alpha)$,
if it exists. This depends only on $\alpha$ and the branch of $\alpha$
represented by $\beta$.
\end{defn}

\begin{defn} \label{D:mis}
A function $f: E \to [0, +\infty)$ is \emph{monotone
integral subharmonic} if it satisfies
the following conditions.
\begin{enumerate}
\item[(a)]
For any $\alpha, \beta \in E$ with $\alpha \geq \beta$,
the function $f \circ h_{\alpha,\beta}: [-\log r(\alpha), -\log r(\beta)]
\to \RR$ is piecewise affine with nonpositive integral slopes.
\item[(b)]
Let $\alpha',\beta'_1,\dots,\beta'_j \in E'$ be any
elements such that $\alpha' \neq \alpha_{\DD'}$,
the elements $\beta'_1, \dots, \beta'_j$ are pairwise incomparable, and
$\alpha' \geq \beta'_i$ for $i=1,\dots,j$.
Let $\alpha,\beta_1,\dots,\beta_j$ be the restrictions to $E$
of $\alpha',\beta'_1,\dots,\beta'_j$. Then
\[
f'_-(\alpha) \leq \sum_{i=1}^j f'_-(\beta_i).
\]
\end{enumerate}
\end{defn}

\begin{lemma} \label{L:mis}
Let $f: E \to [0, +\infty)$ be a monotone integral subharmonic function.
For $\beta \notin E$,
interpret $f \circ h_{\alpha_\DD,\beta}$ as
a function defined only on $[0, -\log r(\beta))$.
\begin{enumerate}
\item[(a)]
For any $\beta \in \DD_0$,
the function $f \circ h_{\alpha_\DD,\beta}$ is convex.
\item[(b)]
For $\beta$ of type (i) or (iv),
the function $f \circ h_{\alpha_\DD,\beta}$ is constant
in a neighborhood of $-\log r(\beta)$.
\end{enumerate}
\end{lemma}
\begin{proof}
To check (a), it suffices to consider $\beta$ of type (ii) or (iii).
In this case, the convexity follows because $f \circ h_{\alpha_\DD,\beta}$
is piecewise affine by condition (a) of Definition~\ref{D:mis},
but at any point the left slope is less than or equal to the right slope
by condition (b) of Definition~\ref{D:mis} (applied with $j=1$).

To check (b), note that by (a), $f \circ h_{\alpha_\DD,\beta}$ is convex
and piecewise affine on finite intervals, with nonpositive integral slopes.
These slopes form an increasing sequence, so there can only be finitely
many of them. This proves that $f \circ h_{\alpha_\DD,\beta}$ achieves
a terminal slope $t$ in a neighborhood of $-\log r(\beta)$.
If $\beta$ is of type (i), then we cannot have $t < 0$ or else
$f$ would take some negative values. Hence $t = 0$.

If $\beta$ is of type (iv), pick $\alpha > \beta$ such that $f \circ
h_{\alpha,\beta}$ is affine with slope $t$.
By Lemma~\ref{L:degree} and Lemma~\ref{L:infinite degree},
$\deg(h_{\alpha,\beta}(s))$ is a nondecreasing integer-valued function
on $(-\log r(\alpha), -\log r(\beta))$ whose values are unbounded.
We can thus choose $\gamma, \delta$ with $\alpha > \gamma > \delta > \beta$
such that $\deg(\gamma) < \deg(\delta)$. In particular, some preimage
$\gamma'$ of $\gamma$ in $E'$ dominates two distinct preimages
$\delta'_1, \delta'_2$ of $\delta$. Applying condition (b)
of Definition~\ref{D:mis} to $\gamma', \delta'_1, \delta'_2$ yields
$t \leq 2t$, which forces $t = 0$. This implies (b).
\end{proof}

One can give some alternate characterizations of the monotone integral
subharmonic condition.
\begin{lemma} \label{L:mis2}
Let $f: E \to [0, +\infty)$ be a function satisfying
condition (a) of Definition~\ref{D:mis}. Then $f$ is 
monotone integral subharmonic if and only if any one of 
the following conditions holds.
\begin{enumerate}
\item[(b$'$)]
Let $\alpha',\beta'_1,\dots,\beta'_j \in E'$ be any
elements such that $\alpha' \neq \alpha_{\DD'}$,
we have
$\alpha' > \beta'_i$ for $i=1,\dots,j$,
and $\beta'_1, \dots, \beta'_j$ represent different branches
of $\alpha'$.
Let $\alpha,\beta_1,\dots,\beta_j$ be the restrictions to $E$
of $\alpha',\beta'_1,\dots,\beta'_j$. Then
\[
f'_-(\alpha) \leq \sum_{i=1}^j f'_+(\alpha, \beta_i).
\]
\item[(b$''$)]
Let $\alpha',\beta'_1,\dots,\beta'_j, \gamma' \in E'$ be any
elements such that $\alpha' \neq \alpha_{\DD'}$,
we have 
$\alpha' > \beta'_i$ for $i=1,\dots,j$,
the elements
$\beta'_1, \dots, \beta'_j$ represent different branches of $\alpha'$,
and $\gamma' > \alpha'$.
Let $\alpha,\beta_1,\dots,\beta_j, \gamma$ be the restrictions to $E$
of $\alpha',\beta'_1,\dots,\beta'_j, \gamma'$. Then
\[
\frac{f(\alpha) - f(\gamma)}{-\log r(\alpha) + \log r(\gamma)}
\leq \sum_{i=1}^j
\frac{f(\beta_i) - f(\alpha)}{-\log r(\beta_i) + \log r(\alpha)}.
\]
\end{enumerate}
\end{lemma}
\begin{proof}
Assuming either (b) or (b$''$), we deduce (b$'$) by choosing 
$\beta'_1,\dots,\beta'_j$ within their branches so that
$f \circ h_{\alpha,\beta_i}$ is affine for $i=1,\dots,j$,
and in the case of (b$''$) choosing $\gamma'$
so that $f \circ h_{\gamma,\alpha}$ is affine.

Assuming (b$'$), we
deduce (b) by induction on $j$, as follows.
For $j=1$, the claim again follows by the convexity of
$f \circ h_{\alpha,\beta}$. Given an instance of (b) with $j > 1$,
let $\gamma' \in E'$ be the minimal element
dominating all of $\beta'_1,\dots,\beta'_j$,
and let $\gamma \in E$ be the restriction of $\gamma'$.
Let $T_1,\dots,T_m$
be the distinct branches of $\gamma'$ represented by 
$\beta'_1,\dots,\beta'_j$; by the choice of $\gamma'$, we must have
$m > 1$.
For $h=1,\dots,m$, choose $\delta'_h \in T_h$
which dominates each $\beta'_i \in T_h$, and let $\delta_h \in E$ be
the restriction of $\delta'_h$. Then
\begin{align*}
f'_-(\alpha) &\leq f'_-(\gamma) & \mbox{(by convexity of $f \circ h_{\gamma,\alpha}$)} \\
&\leq \sum_{h=1}^m f'_+(\alpha,\delta_h) & \mbox{(by (b$'$))} \\
&\leq \sum_{h=1}^m f'_-(\delta_h) & \mbox{(by convexity of $f \circ h_{\alpha,\delta_h}$)} \\
&\leq \sum_{h=1}^m \sum_{i:\, \beta'_i \in T_h} f'_-(\beta_i)
= \sum_{i=1}^j f'_-(\beta_i) &  \mbox{(by the induction hypothesis)}.
\end{align*}
Hence (b) follows.

Assuming (b$'$) (which now implies (b)), we deduce (b$''$) because the functions
$f \circ h_{\alpha,\beta}$ are convex by Lemma~\ref{L:mis}.
\end{proof}
\begin{cor} \label{C:mis}
Let $f_i: E \to [0,+\infty)$ be a sequence of monotone integral
subharmonic functions.
\begin{enumerate}
\item[(a)]
Suppose that the pointwise supremum $f: E \to [0, +\infty)$ of $\{f_i\}$
exists. Then $f$ is monotone integral subharmonic.
\item[(b)]
Suppose that the pointwise limit $f: E \to [0, +\infty)$ of $\{f_i\}$
exists. Then $f$ is monotone integral subharmonic.
\end{enumerate}
\end{cor}
\begin{proof}
The class of convex functions is closed under pointwise
suprema and limits, so in either case $f$ satisfies condition (a) of
Definition~\ref{D:mis}. By Lemma~\ref{L:mis2}, in both cases it now suffices 
to check condition (b$''$) of that lemma. This is immediate in our case (b).

In case (a), given an instance of (b$''$) for $f$ and a choice of
$\epsilon > 0$, choose an index $i$ with $f_i(\alpha) > f(\alpha) - \epsilon$;
then invoking (b$''$) for $f_i$ yields
\[
\frac{f(\alpha) - f(\gamma)}{-\log r(\alpha) + \log r(\gamma)}
\leq \sum_{i=1}^j
\frac{f(\beta_i) - f(\alpha)}{-\log r(\beta_i) + \log r(\alpha)}
+ C \epsilon
\]
for
\[
C = \frac{1}{-\log r(\alpha) + \log r(\gamma)}
+ \sum_{i=1}^j \frac{1}{-\log r(\beta_i) + \log r(\alpha)}.
\]
Since this holds for any $\epsilon >0$, we deduce (b$''$) for $f$.
\end{proof}

\begin{remark} \label{R:type 2}
Note that it suffices to check any of (b), (b$'$), (b$''$) at points
of type (ii). That is because a point of type (iii) has only one branch,
and condition (a) of Definition~\ref{D:mis} ensures that the left
and right slope at such a point coincide, because a piecewise integral
affine function cannot change slope at an irrational point in its domain.
\end{remark}

\begin{example} \label{exa:mis}
For any $P \in k\llbracket x,y\rrbracket[x^{-1}]$,
the function $\alpha \mapsto \max\{0, \log |P|_\alpha\}$ is monotone
integral subharmonic. 
Namely, condition (a) of Definition~\ref{D:mis}
follows from properties of Newton polygons.
As in Remark~\ref{R:type 2}, we need only check (b$'$) at points of type (ii),
where it reduces to the fact that a univariate polynomial over the 
algebraically closed field $k$ has as many roots (counting multiplicity)
as its degree. See \cite[\S 11.1]{kedlaya-course} for a more detailed
explanation.
\end{example}

\begin{defn}
For $f$ monotone integral subharmonic, define the \emph{skeleton}
of $f$ to be the subset $S_f$ of $E$ consisting of $\alpha_{\DD}$
plus all $\alpha \in E \setminus \{\alpha_\DD\}$ for which
$f'_-(\alpha) < 0$. 
We refer to $\alpha_{\DD}$ as the \emph{head} of $S_f$.
We refer to any minimal element of $S_f$ as an \emph{extremity} of $S_f$.
Any $\alpha \in S_f$ which is not the head or an extremity, but for
which there exists an extremity $\beta$ dominated by $\alpha$ for which
$f \circ h_{\alpha_\DD,\beta}$ has a change of slope at $-\log r(\alpha)$,
is called a \emph{joint} of $S_f$.
\end{defn}

\begin{lemma} \label{L:skeleton}
For $f$ monotone integral subharmonic, the skeleton
$S_f$ of $f$ has the following properties.
\begin{enumerate}
\item[(a)]
The set $S_f$ is up-closed. That is,
for $\alpha, \beta \in E$ with $\alpha \geq \beta$ and $\beta \in S_f$,
we have $\alpha \in S_f$.
\item[(b)]
The set $S_f$ is closed under taking infima in $E$ 
of decreasing sequences under  
$\geq$.
\item[(c)]
For any $\beta \in E$, the set of $\alpha \in S_f$ with $\alpha \geq \beta$
has a least element, which satisfies $f(\alpha) = f(\beta)$.
\item[(d)]
Any $\alpha \in S_f$ dominates at least one extremity of $S_f$.
Hence $S_f$ is the up-closure of the set of its extremities.
\end{enumerate}
\end{lemma}
\begin{proof}
We deduce (a) from Lemma~\ref{L:mis}(a). To deduce (b), let
$\beta$ be the infimum of the decreasing sequence $\beta_1,\beta_2,\dots$.
Then $h_{\alpha,\beta}$ is piecewise affine, so for $i$ sufficiently large
we must have $f'_-(\beta_i) = f'_-(\beta)$. Hence $f'_-(\beta) > 0$
and $\beta \in S_f$.
To deduce (c), simply note that the $\alpha \in S_f$ with $\alpha \geq \beta$
are all comparable, so a least element exists by (b).

To deduce (d), suppose the contrary. We can then choose a sequence
$\alpha = \alpha_0 > \alpha_1 > \cdots$ of elements of $S_f$ with no infimum
(e.g., by constructing a countable transfinite sequence with no infimum,
then picking out a subsequence). By (b), this can only happen if the infimum
is a point in $\DD$ of type (i) or (iv). However, by
Lemma~\ref{L:mis}(b), this would imply $f'_-(\alpha_i) = 0$ for $i$ large,
a contradiction. Hence (d) holds.
\end{proof}

\begin{lemma} \label{L:finite extremity}
For $f$ monotone integral subharmonic, the skeleton $S_f$ of $f$
has only finitely many extremities and joints,
all of which are of type (ii).
\end{lemma}
\begin{proof}
Put $m = f'_+(\alpha_\DD, \alpha_{0,0})$.
By condition (b) of Definition~\ref{D:mis} plus the fact that the slopes
are all nonpositive and integral, 
the number of preimages in $E'$ of the extremities of $S_f$ is at most $-m$.
Each extremity $\beta$ is only dominated by finitely many joints
because $f \circ h_{\alpha_\DD, \beta}$ is piecewise affine, so
there are also only finitely many joints.
Finally, no joint or extremity can be of type (iii)
because a piecewise integral affine function can only change
slope at rational points.
\end{proof}

\begin{remark}
Definition~\ref{D:mis} is rather artificial, and is customized to
our present purposes. A more robust notion of subharmonicity has been 
introduced  by Thuillier \cite{thuillier}.
\end{remark}

\subsection{Scales and subharmonicity}

We now quantify the variation of irregularity of a
differential module over a two-dimensional power series ring,
using the device of monotone integral subharmonic functions.

\begin{prop} \label{P:scales sub}
Let $M$ be a differential module of finite rank $d$ over 
$k \llbracket x,y \rrbracket[x^{-1}] \cong R_{2,1}$
(where $k$ is by assumption an algebraically closed field of characteristic
$0$).
For $\alpha \in E$, let $F_\alpha$ be the completion of $\Frac(R_{2,1})$
with respect to $\alpha$, and let $F_d(M, \alpha)$ be the irregularity
of $M_\alpha = M \otimes_{R_{2,1}} F_\alpha$. Then $F_d(M, \cdot)$ is a 
monotone integral subharmonic function.
\end{prop}
\begin{proof}
To check condition (a) of Definition~\ref{D:mis}, 
we may replace $k((x))$ by a finite extension first. We may thus reduce
to the case where $\beta = \alpha_{0,r}$ for some $r$.
In this case, the claim follows from Theorem~\ref{T:convex}.

To finish, it is enough to check condition (b$'$) of
Lemma~\ref{L:mis2}; moreover, it is enough to consider points
$\alpha$ of type (ii), since the assertion for points of type (iii)
follows from (a) (as noted in Remark~\ref{R:type 2}).
We may again replace $k((x))$ by a finite extension first, so 
we may reduce to the case $\alpha = \alpha_{0,e^{-h}}$ for some
positive integer $h$.

Consider the derivations
\[
\del_1 = \frac{\del}{\del x} + h \frac{y}{x} \frac{\del}{\del y},
\qquad
\del_2 = x^h \frac{\del}{\del y}.
\]
On $F_\alpha$, these are of rational type with respect to
$x, y/x^h$.
For any $\beta \in E$, the absolute scale multiset of $M_\beta$
dominates the scale multiset of
$x \del_1$ on $M_\beta$, with equality for $\beta = \alpha$ by
Proposition~\ref{P:scale from h}.

Put $F = \Frac(R_{2,1})$, viewed as a based differential field of 
order 1 equipped with $x \del_1$.
Apply Lemma~\ref{L:cyclic} to construct an isomorphism
$M \otimes_{R_{2,1}} F \cong F\{T\}/F\{T\}P$ for some monic twisted
polynomial $P(T) = T^d + \sum_{i=0}^{d-1} P_i T^i \in F\{T\}$. 
Let $i$ be the total multiplicity of all elements of the absolute scale
multiset of $M_\alpha$ which are strictly greater than 1.
By Proposition~\ref{P:read slopes}, for any $\beta \in E$,
the product over the scale multiset of $x \del_1$ on $M_\beta$
is at least $|P_{d-i}|_\beta |x\del_1|_{F_\beta}^{-i}$, 
with equality for $\beta = \alpha$.

Note that for any $c \in k$,
\[
(x \del_1)(y - cx^h) = h(y - cx^h).
\]
It follows that $|x\del_1|_{F_\beta} = 1$
for any $\beta$ of the form $\alpha_{cx^h,r}$ with $c \in k$
and $r \in [0,1]$. Note that for any $\gamma$ with $\alpha > \gamma$,
there exists a $\beta$ of this form with $\alpha > \beta > \gamma$;
moreover, any $\beta$ which dominates $\alpha$ also has this form.

We conclude from the previous two paragraphs that the inequality
(b$'$) for $F_d(M, \alpha)$ follows from the analogous
inequality for $\log |P_{d-i}|$. The latter holds
by Example~\ref{exa:mis}, so the claim follows.
\end{proof}

\begin{remark} \label{R:alternate proof}
One can give an alternate proof of Proposition~\ref{P:scales sub}
using Proposition~\ref{P:limit scale} to check condition (b$''$)
of Lemma~\ref{L:mis2}. We will return to this point in a subsequent paper.

One can give another alternate proof of Proposition~\ref{P:scales sub}
using \cite[Theorem~2.7.6]{kedlaya-xiao} to check condition (b$'$)
of Lemma~\ref{L:mis2}. However, as noted in the published erratum to \cite{kedlaya-xiao},
\cite[Theorem~2.7.6]{kedlaya-xiao} only becomes correct upon adding the hypothesis that
for every $x \in K^\times$ with $|x| \neq 1$, we have $|u_j \del_j(x)| = |x|$ for
some index $j$. Fortunately, for this application of \cite[Theorem~2.7.6]{kedlaya-xiao},
we may take $K = k((x))$ and $\del_1 = \frac{\del}{\del x}$,
so the extra hypothesis is satisfied.
\end{remark}

\begin{remark} \label{R:extend scalars}
For the reader familiar with at least the structure of 
one-dimensional Berkovich analytic spaces, it may be helpful to comment
on the passage from a module $M$ over $R_{2,1}$ to a module over the
open unit disc. Namely, since $M$ is not defined over the closed unit
disc, it does not admit a base extension to the residue field of the
Gauss point, which is the completion of $k((x))(y)$ for the
$1$-Gauss norm (with $|y| = 1$). However, that field embeds isometrically into
the complete field $k((y))((x))$,
and $M$ does admit a base extension
to the latter. 

The legitimacy of this extension of the base field is provided by
Lemma~\ref{L:base change}, which asserts that the base extension does not
change the irregularity as long as the rational type condition is preserved
in a suitable set of coordinates. For instance, if
$M$ is defined over $k[x,y][x^{-1}]$, 
then we may use the derivations $\frac{\del}{\del x},
\frac{\del}{\del y}$ in Lemma~\ref{L:base change} to extend scalars
from $k[x,y][x^{-1}]$ to $R_{2,1}$ without changing irregularities.
\end{remark}

\begin{remark}
Following up on the previous remark,
we note that Proposition~\ref{P:scales sub} includes the assertion that
$F_d(M, \cdot)$ is piecewise affine on some segment starting from
the Gauss point $\alpha_\DD$, even though $M$ is not defined over the full
closed unit disc. It is worth commenting a bit more closely on this fact.

One may consider more generally a differential module $M$ of finite rank $d$
over the open unit disc over any complete nonarchimedean field $K$ of
characteristic $0$. Let $t$ be a coordinate for the disc; we may then use
the scale multiset of the action of $\frac{\del}{\del t}$ to define partial
irregularity functions $F_i(M,r)$ for $r \in (0, +\infty)$ (corresponding
to the $(e^{-r})$-Gauss norm), but not \emph{a priori} at $r=0$.
On $(0, +\infty)$,
the function $F_d(M,r)$ is then continuous, convex, and piecewise
affine with integral slopes \cite[Theorem~11.3.2]{kedlaya-course}.
However, with no additional hypotheses, it is possible for $F_d(M,r)$
to diverge to $+\infty$ as $r \to 0^+$.

Suppose now that $M$ arises from a finite differential module over the ring
of \emph{bounded} rigid-analytic functions over the open unit disc.
Then it makes sense to compute $F_d(M,r)$ at $r=0$, and it turns out to be
continuous there; this follows from the continuity of the Gauss norm
of a fixed ring element as $r$ varies (compare
Lemma~\ref{L:varies continuously}). For general $K$, it still
can happen that the \emph{slopes} of $F_d(M,r)$ diverge to $-\infty$
as $r \to 0^+$, again because this occurs for the Gauss norm of a ring
element. (That is, one can observe this phenomenon already
in the case $M = E(f)$.)

However, when $K$ is \emph{discretely valued}, this cannot occur,
so $F_d(M,r)$ is indeed piecewise affine at $r=0$.
In the case of residual characteristic 0, as in this paper
(where we take $K = k((x))$), this follows from the integrality
condition $F_d(M, r) \in \ZZ + \ZZ r$ (from Proposition~\ref{P:read slopes})
and an elementary analysis argument (which takes place
within the proof of 
Theorem~\ref{T:convex}
via \cite[Theorem~2.4.2]{kedlaya-part3}). 

In the case where $K$ is discretely valued of positive residual characteristic,
the relationship between partial irregularities and Newton
polygons is somewhat more indirect, leading to a weaker integrality condition:
the function $F_d(M,r)$ is piecewise represented by an affine function
of integral slope but not necessarily integral constant term.
However, one can bound the denominator of the constant term as a function
of $f_d(M,r) = F_d(M,r) - F_{d-1}(M,r)$ \cite[Theorem~10.7.1]{kedlaya-course},
and thus obtain piecewise affinity at $r=0$ in case
$f_d(M,0) > 0$.
By combining this argument with a bit of extra analysis (as in
\cite[Lemma~11.6.3]{kedlaya-course}), one can recover piecewise
affinity at $r=0$ even if $f_d(M,0) = 0$.
\end{remark}

\section{Good formal decompositions of connections on surfaces}
\label{sec:surfaces}

In this section, we establish existence of good decompositions
(after pullback along a suitably ramified cover) for formal
meromorphic connections. In so doing, we resolve a
conjecture of Sabbah \cite[Conjecture~2.5.1]{sabbah}
and reproduce a result of Mochizuki \cite[Theorem~1.1]{mochizuki},
both concerning connections on surfaces.
We will return to this topic in the higher-dimensional case in
a subsequent paper.

\setcounter{theorem}{0}
\begin{convention} \label{conv:either}
Throughout \S~\ref{sec:surfaces}, we may work \emph{either} in 
the category of algebraic varieties
(or more exactly, reduced separated schemes of finite type)
over an algebraically closed field $k$ of characteristic $0$,
or the category of
complex analytic varieties (in which case we take $k = \CC$).
\end{convention}

\begin{hypothesis} \label{H:geometric}
Throughout \S~\ref{sec:surfaces},
let $X$ be a smooth variety (in whichever category we are working),
let $Z$ be a normal crossings divisor on $X$, and let
$\calZ = \cup_i Z_i$ be a locally closed stratification of $Z$.
\end{hypothesis}

\subsection{Formal meromorphic functions}

\begin{defn}
The \emph{formal completion} of $X$ along $\calZ$,
denoted $\widehat{X|\calZ}$, consists of, for each $i$,
the sheaf $\calO_{\widehat{X|Z_i}}$ on $Z_i$ of formal functions
on $X$ along $Z_i$.
Let $\calO_{\widehat{X|\calZ}}(*Z)$ be the collection of the
sheaves of \emph{formal meromorphic functions}
$\calO_{\widehat{X|Z_i}}(*Z)$.
\end{defn}

\begin{defn}
A \emph{$\nabla$-module} over $\calO_{\widehat{X|\calZ}}(*Z)$
is a coherent sheaf $\calE$ over $\calO_{\widehat{X|\calZ}}(*Z)$ 
equipped with a flat $k$-linear connection
$\nabla: \calE \to \calE \otimes \Omega^1_{X/k}$. The flatness condition
(sometimes called \emph{integrability} to avoid confusion with the
algebro-geometric notion of a flat morphism) asserts that the composition
of $\nabla$ with the induced map
$\nabla^{(1)}: \calE \otimes \Omega^1_{X/k} \to \calE \otimes \Omega^2_{X/k}$ 
is zero.
By Proposition~\ref{P:locally simple}, $\calE$ is locally free.
\end{defn}

\begin{defn}
We similarly define the notion of a \emph{log-$\nabla$-module}
over $\calO_{\widehat{X|\calZ}}$, with respect to the log-structure
defined by $Z$. However, a log-$\nabla$-module is not necessarily
locally free.
\end{defn}

\begin{remark}
For $Y$ a stratum of $\calZ$, we can view a (log-)$\nabla$-module over
$\calO_{\widehat{X|Y}}(*Z)$ also as a
(log-)$\nabla$-module over $\calO_{\widehat{X|\calZ}}(*Z)$ by setting the
components over the other strata to be zero. We will do this frequently
and implicitly in what follows.
\end{remark}

\subsection{Good elementary models}

\begin{defn}
Let $Y$ be a stratum of $\calZ$. We say that a $\nabla$-module
$\calE$ over $\calO_{\widehat{X|\calZ}}(*Z)$ is
\emph{regular along $Y$} if $\calE|_Y$ is isomorphic to the restriction
of a locally free log-$\nabla$-module over
$\calO_{\widehat{X|Y}}$ with respect to the log-structure
defined by $Z$.
\end{defn}

\begin{defn} \label{D:exp}
Let $Y$ be a stratum of $\calZ$.
For $\phi$ a section of $\calO_{\widehat{X|Y}}(*Z)$, let $E(\phi)$ denote the
$\nabla$-module over $\calO_{\widehat{X|Y}}(*Z)$
free on one generator $\bv$ satisfying
$\nabla(\bv) = \bv \otimes d\phi$.
\end{defn}

\begin{defn} \label{D:local model}
Let $Y$ be a stratum of $\calZ$. Let
$\calE$ be a $\nabla$-module over $\calO_{\widehat{X|\calZ}}(*Z)$.
An \emph{elementary local model} of $\calE$ along $Y$ is an isomorphism
\begin{equation} \label{eq:local model}
\calE|_Y \cong \bigoplus_{\alpha \in A} E(\phi_\alpha) \otimes \calR_\alpha
\end{equation}
for some sections $\phi_\alpha$ of $\calO_{\widehat{X|Y}}(*Z)$
(indexed by an arbitrary set $A$)
and some regular $\nabla$-modules $\calR_\alpha$.
An elementary local model is \emph{good} if it satisfies the following
two additional conditions.
\begin{enumerate}
\item[(a)]
For $\alpha \in A$, if $\phi_\alpha$ is not a section of
$\calO_{\widehat{X|Y}}$, then
the divisor of $\phi_\alpha$ is anti-effective (has all
multiplicities nonpositive) with support in $Z$.
\item[(b)]
For $\alpha, \beta \in A$, if $\phi_\alpha - \phi_\beta$ is not a section of
$\calO_{\widehat{X|Y}}$, then
the divisor of $\phi_\alpha - \phi_\beta$ is anti-effective
 with support in $Z$.
\end{enumerate}
Note that if $Y = \{y\}$ is a point, 
and we choose an identification of $\calO_{\widehat{X|Y}}(*Z)$
with the ring $R_{n,m}$ of Notation~\ref{N:numerical},
then the notion of a good elementary local model of
$\calE$ along $Y$ corresponds precisely to the notion of a good decomposition
of $\calE|_Y$. Moreover, the choice of the identification does 
not matter thanks to Proposition~\ref{P:automorphism}.
\end{defn}

\begin{defn}
Let $\calE$ be a $\nabla$-module over $\calO_{\widehat{X|\calZ}}(*Z)$.
We say that $\calE$ admits a \emph{good formal structure} at a point
$y \in Z$ if there exists a finite cover of some neighborhood of 
$y$ in $X$, ramified only along $Z$, on which the pullback of
$\calE$ admits a good elementary local model along the stratum containing
some inverse image of $y$.
If this holds for all $y \in Z$, we simply
say that $\calE$ admits a good formal structure.
\end{defn}

\begin{remark} \label{R:formalizing}
Sabbah's original conjecture concerns $\nabla$-modules over
the sheaf $\calO_X(*Z)$ of meromorphic functions; it requires
the good elementary local models to be the formalizations of
modules in which the $\phi_\alpha$ are meromorphic sections, not just
formal meromorphic sections. However, for $X$ a surface,
Sabbah has proved in the analytic setting
\cite[Proposition~I.2.4.1]{sabbah} that 
for $\calE$ a $\nabla$-module over $\calO_X(*Z)$,
the formalization of $\calE$ is isomorphic to the formalization
of a good elementary local model (in the sense over $\calO_X(*Z)$)
if and only if it itself has a good elementary local model (in the sense
over $\calO_{\widehat{X|\calZ}}(*Z)$). That is, if the formalization
of a convergent connection has a good elementary local model, then the
components of that model may themselves be taken to be
convergent, even though the
isomorphism is in general not convergent. The argument in the algebraic
setting is similar, except that one must allow the $\phi_\alpha$ to be
\emph{algebraic} functions, rather than regular functions (i.e., they
are regular functions on \emph{\'etale} opens rather than Zariski opens).
Consequently, we may treat Sabbah's problem
by working exclusively in the formal setting.
The higher dimensional analogue of this reduction is also
similar; we will write it down explicitly in a subsequent paper.
\end{remark}

The following result of Sabbah \cite[Th\'eor\`emes~2.3.1,~2.3.2]{sabbah}
implies that $\calE$ admits a good formal structure everywhere outside
a discrete (or finite, in the algebraic case) set of points. 
\begin{theorem} \label{T:open good}
Suppose $X$ is a surface. In the analytic case, assume also that
$\calZ$ is the trivial stratification. For each $z \in Z$, there exists a
(topological or Zariski) open neighborhood $U$ of $z$ in $X$ such that
$\calE$ admits a good formal structure at each point of $(U \cap Z) \setminus \{z\}$.
\end{theorem}

\begin{remark}
Sabbah also introduces the notion of a \emph{very good formal structure}
\cite[\S I.2.2]{sabbah}, but one cannot always achieve such a structure
even after a blowup \cite[Lemme~I.2.2.3]{sabbah}.
\end{remark}

\begin{remark}
In the analytic case, one cannot expect Theorem~\ref{T:open good}
(or Theorem~\ref{T:global good}) to hold for a nontrivial stratification. For instance,
if $z$ is an isolated point in the stratification, there could be an infinite
sequence of points in another stratum accumulating at $z$, at each of which 
$\calE$ fails to admit a good formal structure.
\end{remark}

\subsection{The main local theorems}

We now use our numerical criterion to prove that in some special cases,
one can arrange to have good formal structure everywhere by performing
a suitable blowing up.

\begin{defn}
Since we will only discuss surfaces, we may define a \emph{modification}
of a variety to be a composition of point blowups.
\end{defn}

\begin{theorem} \label{T:local good1}
Put $X = \AAA^2_k$ (in the algebraic category)
with coordinates $x,y$,
let $Z$ be the line $x=0$,
and let $Y$ consist solely of the origin.
Let $\calE$ be a $\nabla$-module over $\calO_{\widehat{X|Y}}(*Z)$.
Then there exists a modification $f: X' \to X$ with
$X'$ smooth and $Z' = f^{-1}(Z)$ a normal crossings divisor,
such that $f^* \calE$ admits a good formal structure for
some stratification of $Z'$.
\end{theorem}
\begin{proof}
Identify $\calE$ with a finite differential module $M$ of rank $d$ over
$k\llbracket x,y \rrbracket[x^{-1}] \cong R_{2,1}$. 
By Proposition~\ref{P:scales sub}, the irregularity of $M$ determines
a monotone integral subharmonic function on $E$ (the subset of the
Berkovich open unit disc consisting of points of type (ii) or (iii)), 
and similarly for $\End(M)$.
Let $V$ be the set
of divisorial valuations on $R_{2,1}$ corresponding to joints or extremities
of the resulting skeleta.
We may choose $f$ such that each $v \in V$ corresponds to
an exceptional divisor on $X'$. 

We now check that $f^* \calE$ admits a good
formal structure for some stratification of $Z'$.
For each component of $Z'$, 
Theorem~\ref{T:open good} gives good formal structures everywhere
once we exclude a finite (in the algebraic case) or discrete
(in the analytic case) set of points.
We put each of the remaining points in its own stratum.

In this case, we claim that good formal structures exist by
Theorem~\ref{T:criterion}.
At a crossing point, the irregularity function for $\calE$ 
considered in Theorem~\ref{T:criterion}
corresponds to the restriction of $F_d(M, \cdot)$ to a
segment of the skeleton of $M$ between two joints.
On such a segment, $F_d(M, \cdot)$ is affine.
At a noncrossing point, the irregularity function for $\calE$
corresponds to the restriction of $F_d(M, \cdot)$ to
a segment in $E$ which lies outside of the skeleton except
possibly at its upper endpoint.
On such a segment, $F_d(M, \cdot)$ is constant by Lemma~\ref{L:skeleton}(c).
Similar arguments apply to $\End(\calE)$ and $F_{d^2}(\End(M),\cdot)$.
(Note that Theorem~\ref{T:convex} by itself is unable to guarantee
this favorable behavior at noncrossing points; this is why
the analysis of the Berkovich disc is needed. Note also that
we are implicitly using Remark~\ref{R:extend scalars} in the form that
refining the stratification of $Z'$ does not change irregularities.)

In all cases, we deduce that the numerical criterion of
Theorem~\ref{T:criterion} is verified. The desired result thus follows.
\end{proof}

\begin{theorem} \label{T:local good2}
Put $X = \AAA^2_k$ (in the algebraic category)
with coordinates $x,y$,
let $Z$ be the union of the coordinates axes in $X$,
and let $Y$ consist solely of the origin $x=y=0$.
Let $\calE$ be a $\nabla$-module over $\calO_{\widehat{X|Y}}(*Z)$.
Then there exists a modification $f: X' \to X$ with
$X'$ smooth and $Z' = f^{-1}(Z)$ a normal crossings divisor,
such that $f^* \calE$ admits a good formal structure for
some stratification of $Z'$.
\end{theorem}
\begin{proof}
Identify $\calE$ with a differential module $M$ of finite rank $d$ over
$k\llbracket x,y \rrbracket[x^{-1},y^{-1}] \cong R_{2,2}$. 
Set notation as in Theorem~\ref{T:convex}, and let $V$
be the set of divisorial valuations on $R_{2,2}$ which correspond
to values of $r$ at which $F_d(M,r)$ or $F_{d^2}(\End(M),r)$ changes slope.
We may choose $f$ to be a toroidal blowup at $Y$,
such that each $v \in V$ corresponds to
an exceptional divisor on $X'$. If we put each crossing point
of $Z'$ in its own stratum, then $f^* \calE$ admits a good
formal structure at each crossing point by 
Theorem~\ref{T:criterion}.
We may thus apply Theorem~\ref{T:open good} to reduce the claim
to finitely many instances of Theorem~\ref{T:local good1}.
\end{proof}

\begin{remark}
The reduction of Theorem~\ref{T:local good2} to Theorem~\ref{T:local good1}
has also been shown by Sabbah \cite[Proposition~4.3.1]{sabbah}
and Andr\'e \cite[Th\'eor\`eme~5.4.1]{andre}. (Thanks to Takuro Mochizuki
for providing these references.)
\end{remark}

\subsection{The main global theorem}

To conclude, we state our main global theorem, reiterating 
Hypothesis~\ref{H:geometric} but specializing to the two-dimensional
case.

\begin{theorem} \label{T:global good}
Let $X$ be a smooth surface, and let $Z$ be a normal crossings
divisor on $X$. In the algebraic case, let $\calZ$ be any locally closed 
stratification of $Z$; in the analytic case,
take $\calZ$ to be the trivial stratification. 
Let $\calE$ be a $\nabla$-module over 
$\calO_{\widehat{X|\calZ}}(*Z)$.
Then there exist a modification $f: X' \to X$, which is the composition
of a discrete (i.e., locally finite in the analytic case,
finite in the algebraic case) sequence of point blowups, with
$X'$ smooth and $Z' = f^{-1}(Z)$ a normal crossings divisor,
and a refinement $\calZ'$ of the pullback stratification on $Z'$ induced
by $\calZ$,
such that the restriction of $f^* \calE$ to $\calO_{\widehat{X'|\calZ'}}(*Z')$
has a good formal structure.
\end{theorem}
\begin{proof}
By Theorem~\ref{T:open good}, we can get a good formal structure
away from a discrete set of points on $X$.
To resolve each of the
others, apply Theorems~\ref{T:local good1} and~\ref{T:local good2}.
\end{proof}

\begin{remark}
In its analytic aspect, Theorem~\ref{T:global good} resolves
\cite[Conjecture~2.5.1]{sabbah}; the case $\rank(\calE) \leq 5$ 
had been established by Sabbah 
\cite[Th\'eor\`eme~2.5.2]{sabbah}.
In its algebraic aspect, and further restricted to the case of a
$\nabla$-module obtained by base change over $\calO_X(*Z)$ (i.e., 
where the connection is actually meromorphic, not just formal meromorphic),
Theorem~\ref{T:global good} reproduces a result of Mochizuki 
\cite[Theorem~1.1]{mochizuki} (modulo Remark~\ref{R:mochizuki def2} below).
A higher-dimensional analogue of the latter is
\cite[Theorem~19.5]{mochizuki2}; we will generalize that result in 
a subsequent paper.
\end{remark}

\begin{remark} \label{R:mochizuki def2}
Recall (Remark~\ref{R:mochizuki def}) that there is a slight discrepancy
between the notions of good formal structures considered by Sabbah
and Mochizuki. In particular, if $f^* \calE$ has a good formal structure
in the sense of Mochizuki, then it also has one in the sense of 
Sabbah, but not conversely. However, if both $f^* \calE$ and $f^* \End(\calE)$
have good formal structures in the sense of Sabbah, then
$f^* \calE$ has a good formal structure in the sense of Mochizuki.
Hence Theorem~\ref{T:global good} implies the corresponding statement
using Mochizuki's definition of a good formal structure.
\end{remark}

\begin{remark}
Mochizuki's work in \cite{mochizuki} also applies to 
the case when $\calE$ is formal meromorphic
but defined over a subring of $k$ which is finitely generated over $\ZZ$.
However, one can only reduce the general case of Theorem~\ref{T:global good}
to working over subrings of $k$ which are \emph{countably} generated 
over $\ZZ$. This does not permit the sort of reduction modulo $p$
arguments used in \cite{mochizuki}.
\end{remark}

\end{document}